\documentclass{scrartcl}
\KOMAoptions{paper=a4,abstract=on}

\usepackage{bbm}
\usepackage{array,multirow,booktabs,ragged2e,tabularx}
\usepackage{subcaption}
\usepackage{eso-pic}
\usepackage{multicol}
\usepackage{amssymb}
\usepackage{amsmath}
\usepackage{amsthm}
\usepackage{faktor}
\usepackage{tikz-cd}
\usepackage{dsfont}
\usepackage{yfonts}
\usepackage[hidelinks]{hyperref}
\usepackage{hyperref}
\usepackage{graphicx}
\usepackage{enumerate}
\usepackage{xspace}
\usepackage[all]{xy}
\usepackage{faktor} 
\newcommand{\dtt}{\widetilde{\delta}}
\newcommand{\der}{\operatorname{d}}
\newcommand{\A}{\mathcal{A}}
\newcommand{\B}{\mathcal{B}}
\newcommand{\PA}{\mathcal{P}}
\newcommand{\jj}{\mathcal{J}}
\newcommand{\ii}{\mathcal{I}}

\newcommand{\CC}{\operatorname{\A_{0}}}
\newcommand{\CCC}{\operatorname{\A_{0}(\C)}}
\newcommand{\cc}{\operatorname{C_c}}
\newcommand{\czero}{\operatorname{C_0}}
\newcommand{\ccu}{\operatorname{C_c}(U)}
\newcommand{\C}{\mathcal{C}}
\newcommand{\D}{\mathcal{D}}
\newcommand{\E}{\mathcal{E}}
\newcommand{\G}{\mathcal{G}}

\newcommand{\CZ}{\mathcal{C}^{(0)}}
\newcommand{\DZ}{\mathcal{D}^{(0)}}
\newcommand{\EZ}{\mathcal{E}^{(0)}}
\newcommand{\GZ}{\mathcal{G}^{(0)}}

\newcommand{\ts}{\Theta_s}
\newcommand{\echap}{\widehat{E}}
\newcommand{\achap}{\widehat{A~}}

\newcommand{\drt}{D_{\rho(t)}}

\newcommand{\jrt}{J_{\rho(t)}}
\newcommand{\dls}{D_{\lambda(s)}}
\newcommand{\drs}{D_{\rho(s)}}
\newcommand{\jls}{J_{\lambda(s)}}
\newcommand{\jrs}{J_{\rho(s)}}

\newcommand{\CU}{\mathcal{C}^{(1)}}
\newcommand{\DU}{\mathcal{D}^{(1)}}
\newcommand{\EU}{\mathcal{E}^{(1)}}

\newcommand{\CD}{\mathcal{C}^{(2)}}

\newcommand{\comp}{\mathbb{C}}
\newcommand{\nat}{\mathbb{N}}

\newcommand{\pint}[2]{\left\langle #1,#2\right\rangle}

\newcommand{\cont}{\subseteq}
\newcommand{\sm}{\setminus}
\newcommand{\mitt}{\mu_{t}}
\newcommand{\miss}{\mu_{s}}

\newcommand{\pti}{\widetilde{\pi}}
\newcommand{\bis}{\operatorname{Bis}}
\newcommand{\bisc}{\operatorname{\bis(\C)}}

\newcommand{\ran}{\operatorname{ran}}
\newcommand{\id}{\operatorname{id}}
\newcommand{\s}{\operatorname{\mathbf{d}}}
\newcommand{\supp}{\operatorname{supp}}
\newcommand{\rr}{\operatorname{\mathbf{r}}}
\newcommand{\un}{\operatorname{\mathbf{u}}}
\newcommand{\m}{\operatorname{\mathbf{m}}}
\newcommand{\sps}{\operatorname{span}}
\newcommand{\spf}{\overline{\operatorname{span}}}

\newcommand{\bare}{\varsigma_{e}} 
\newcommand{\Bare}{\widetilde{E}} 
\newcommand{\barn}{\varsigma_{e_{\alpha}}} 
\newcommand{\barf}{\varsigma_{f}}
\newcommand{\barls}{\varsigma_{\lambda(s)}}
\newcommand{\barlt}{\varsigma_{\lambda(t)}}
\newcommand{\barlst}{\varsigma_{\lambda(st)}}

 \makeatletter
\newtheorem*{rep@theorem}{\rep@title}
\newcommand{\newreptheorem}[2]{%
\newenvironment{rep#1}[1]{%
 \def\rep@title{#2 \ref{##1}}%
 \begin{rep@theorem}}%
 {\end{rep@theorem}}}
\makeatother

\newreptheorem{teo}{Theorem}
\newreptheorem{cor}{Corollary}
\newreptheorem{defi}{Definition}
\newreptheorem{prop}{Proposition} 

\theoremstyle{plain}
	\newtheorem{teo}{Theorem}[section]
	\newtheorem{lemma}[teo]{Lemma}
	\newtheorem{cor}[teo]{Corollary}
	\newtheorem{prop}[teo]{Proposition}
	
\theoremstyle{definition}
	\newtheorem{defi}[teo]{Definition}
	\newtheorem{rem}[teo]{Remark}
	\newtheorem{exe}[teo]{Example}

\newcommand{\cf}{\mbox{cf.}\xspace}				

\title{Étale categories, restriction semigroups,\\ and their operator algebras}

\date{\today}

\allowdisplaybreaks

\begin{document}

\author{Natã Machado\thanks{ 
         Universidade Federal de Santa Catarina,
		\href{mailto:n.machado@ufsc.br}{\textcolor{blue}{\nolinkurl{n.machado@ufsc.br}}}
    	} \footnote{Bolsista CAPES/BRASIL.} 
     \and Gilles G. de Castro \thanks{
	    Universidade Federal de Santa Catarina,
		\href{mailto:gilles.castro@ufsc.br}{\textcolor{blue}{\nolinkurl{gilles.castro@ufsc.br}}}
	} 
}
\sloppy
\maketitle

\begin{abstract}
\noindent
    We define the full and reduced non-self-adjoint operator algebras associated with étale categories and restriction semigroups, answering a question posed by Kudryavtseva and Lawson in \cite{lawson}. Moreover, we define the semicrossed product algebra of an étale action of a restriction semigroup on a $C^*$-algebra, which turns out to be the key point when connecting the operator algebra of a restriction semigroup with the operator algebra of its associated étale category. We also prove that in the particular cases of étale groupoids and inverse semigroups our operator algebras coincide with the $C^*$-algebras of the referred objects. 
	    
\vspace*{0.3cm}

\vspace*{0.1cm}
\end{abstract}

\tableofcontents

\section*{Introduction}\addcontentsline{toc}{section}{Introduction}

The study of non-self-adjoint operator algebras began with the pioneering paper of Kadison and Singer in 1960 \cite{kadinsonsinger}. Furthermore, Sarason's paper in 1966 \cite{sarason1966invariant} on unstarred algebras also made significant contributions to the emerging field. Additionally, the works of Arveson and Hamana played a fundamental role in shaping the subject, particularly with regard to the $C^*$-envelope program \cite{arveson1972subalgebras,arveson1998subalgebras,arveson1969subalgebras,hamana1985injective,hamana1979injective}.

In 1990, Blecher, Ruan and Sinclair characterized abstractly an operator algebra up to completely isometric isomorphisms, providing new insights and directions for the community \cite{blecherruansinclair}. Also, the connection between non-self-adjoint operator algebras and dynamical systems has its roots in Arveson's papers \cite{arveson1,arveson2}. Since then, numerous interesting examples and generalizations have emerged (refer to \cite{kakariadis,davidson2009nonself,kat,peters88} and the references therein).

The concept of restriction semigroups seems to have appeared for the first time in Schweizer and Sklar's series of papers about partial functions \cite{schweizer1960algebra,schweizer1961algebra,schweizer1965algebra,schweizer1967function}. Its motivation is traced back to the independent works of Wagner and Preston on inverse semigroups (see also \cite{hollings} and the references therein). Restriction semigroups have many presentations and emerge in several different contexts, for more on this subject see Gould's survey \cite{gouldsurvey} and also \cite{lawson} for generalizations. 

There is a well-known connection between étale groupoids, inverse semigroups and $C^*$-algebras. Indeed, since Renault's thesis \cite{renault} was published, in 1980, the interplay between these objects has become very fruitful and has been explored by many authors (see for instance \cite{ARMSTRONG2022127,exel2010reconstructing, steinberg2010groupoid} and the references therein). Paterson \cite{paterson} introduced the groupoid of germs of an inverse semigroup action, which is one of the keys when studying how the $C^*$-algebra of an inverse semigroup can be realized as an étale groupoid $C^*$-algebra. In fact, one starts with an inverse semigroup $S$, then considers the \textit{canonical} action $\theta$ of $S$ on its spectrum $\widehat{E(S)}$ and obtains the correspondent groupoid of germs $\G$. Using Sieben's $C^*$-crossed product algebra \cite{sieben1997c} of the action $\theta$ as an intermediary step, Paterson proved that $C^*(\G)$ and $C^*(S)$ are isomorphic. Exel \cite{exel} improved this work by removing hypotheses in the inverse semigroup. There, Exel also provides a careful study of non-Hausdorff étale groupoids, and also introduces the tight $C^*$-algebra of an inverse semigroup.

Recently, Kudryavtseva and Lawson used techniques of pointless topology to establish a duality between complete restriction monoids and étale topological categories (\cite[Theorem 7.22]{lawson}). This duality extends the one originally presented in \cite{lawson2013pseudogroups} between pseudogroups and étale groupoids. In their paper one can find the following comment: 

\begin{quote}
\textit{``...there is the important question of how our work fits into the theory of operator algebras; the role of inverse semigroups and étale groupoids is of course well established but it is natural to ask if a theory of combinatorial non-self-adjoint operator algebras associated to étale categories could be developed that extended the theory developed in \cite{exel}.''}   
\end{quote}

The main goal of this paper is to give a positive answer to the above question and extend the works of Exel \cite{exel} and Paterson \cite{paterson} to the context of restriction semigroups, étale topological categories and non-self-adjoint operator algebras.

Here is an outline of this paper. In Section \ref{sec:ecoa}, we establish definitions for the full and the reduced operator algebras of an étale category $\C$, namely $\A(\C)$ and $\A_r(\C)$. 

Moving on to Section \ref{sec:rsoa}, we once again define the full and reduced operator algebras of a restriction semigroup $S$. Additionally, we prove that if $S$ is an inverse semigroup regarded as a restriction semigroup then the operator algebras we have defined coincide with the full and the reduced $C^*$-algebras of $S$ (Subsection \ref{subsection:inversesemigroupcase}). Furthermore, we define actions of restriction semigroups on both topological spaces and $C^*$-algebras and construct the semicrossed product algebra associated with these actions. We conclude the section by demonstrating Theorem \ref{thm:A(S)semicrossed}, which shows that the operator algebra of a restriction semigroup has a semicrossed product structure. 

In Section \ref{sec:catgerms}, we introduce the category of germs of a restriction semigroup action on a topological space. Our main result, Theorem \ref{thm:A(C)semicrossed}, states that the operator algebra of a category of germs is isomorphic to a semicrossed product. In Subsection \ref{groupoidcase}, we analyze the case where $\G$ is an étale groupoid and we prove that, under suitable conditions, $\A(\G)$ and $\A_r(\G)$ are the very well-known full and reduced $C^*$-algebras of $\G$.

\section{Étale categories and their operator algebras}\label{sec:ecoa}

In this section, we present étale categories and define their non-self-adjoint operator algebras. The familiar reader will note the similarity between the construction of our algebra and the $C^*$-algebra of an étale groupoid. 

Let $\C=\big(\CZ,~\CU\big)$ be a small category, where $\CZ$ and $\CU$ represent the set of all \textit{objects} and \textit{morphisms} of $\C$, respectively. There are some special functions and sets to consider when dealing with a category. Designated as  \textit{source} $\s: \CU\rightarrow \CZ$ and \textit{range} $\rr: \CU\rightarrow \CZ$, these maps assign the domain and codomain to morphisms. The \textit{unit} map $\un: \CZ\rightarrow \CU$  assigns $\id_u$ to an object $u$. Additionally, the set $\CD=\{(x,y)\in \CU\times\CU: \s(x)=\rr(y) \}$ represents the collection of all \textit{composable morphisms}, and, finally, $\m: \CD\rightarrow \CU$ assigns the composition $xy$ to every pair of $(x,y)$ in $\CD$. The maps $\s$, $\rr$, $\un$ and $\m$ are usually called \textbf{structure maps} of $\C$.
\begin{defi}
For $u$ and $v$ in $\CZ$, we define $\C_u=\{x\in \CU \ \mid \s(x)=u \}$, and  $\C^u=\{x\in \CU \ \mid \rr(x)=u \}$. Moreover, we define $\C_u^{v}$ to denote the usual Hom$_\C(u,v)$ set, which is precisely $\{x\in \CU \ \mid \s(x)=u \text{ and } \rr(x)=v \}$. Finally,
 for all $z \in \CU$, we define $M_z$ to be the set $\{(x,y)\in \CD\ : xy=z\}$.
\end{defi}

\begin{defi}
Let $\C$ be a small category. We call $\C$ a \textbf{topological category} if $\CZ$ and $\CU$ are topological spaces and the structure maps are continuous.
\end{defi}

Before we define étale categories, recall that if $X$ and $Y$ are topological spaces and $f:X\rightarrow Y$ is a map then we say that $f$ is a topological embedding if the map $f:X \rightarrow f(X)$ is a homeomorphism, where $f(X)$ is equipped with the relative topology.

\begin{defi}
Let $\C$ be a topological category. We call $\C$ an \textbf{étale  category} if the following conditions hold:

\begin{enumerate}[{\normalfont \rmfamily  (i)}]
    \item $\CZ$ is a locally compact Hausdorff space. 
    
    \item The maps $\s$ and $\rr$ are local homeomorphisms.
    
    \item $\un$ is a topological embedding.
\end{enumerate}
\end{defi}
 
 \begin{rem}This definition of étale category is stronger than the one presented in \cite{lawson} since they do not require $\CZ$ to be a locally compact Hausdorff space. Here we impose this condition due $C^*$-algebraic purposes.\end{rem} 

\begin{defi}
	A category $\C$ is called left (resp. right) cancellative if $xy=xw$ (resp. $yx=wx$) implies $y=w$, for every triple of morphisms $x,y,$ and $w$ in $\C$. A category is called \textbf{cancellative} if it is both left and right cancellative.
\end{defi}

\begin{exe}[Transformation category]
Let $X$ be a compact Hausdorff topological space and $f:X\rightarrow X$ be a local homeomorphism. Define 
$\CZ=X$ and $\CU=\{(y,n,x):\ f^n(x)=y,\ n\in \nat \}$. Here, $\CZ$ is equipped with the topology of $X$ and $\CU$ is equipped with the relative topology of the product space $X\times \nat\times X$, where $\nat$ is viewed as a discrete space. Moreover, we define the source and range maps to be $\s(y,n,x)=x$ and $\rr(y,n,x)=y$. The identity of $x$ is set to be $\un(x)=(x,0,x)$, and finally the composition of two triples is given by $(z,m,y)(y,n,x)=(z,m+n,x).$ It is routine to verify that $\C=\big(\CZ,\ \CU\big)$ is a category with these structure maps. Furthermore, $\C$ is left (and right) cancellative because $$(z,m,y)(y,n,x)=(z,m,y)(y,k,t)\Rightarrow (z,m+n,x)=(z,m+k,t) \Rightarrow k=n \text{ and } t=x.$$

Let $\gamma=(y,n,x)\in \C$ and $A\cont X$ be an open set such that $x \in A$ and $f_{A}^n$ is a homeomorphism. The open set $U:=(f^n(A)\times \{n\}\times A)\ \bigcap \CU=\{(f^n(a),n,a):\ a \in a\}$ is an open neighborhood of $\gamma$ and we have that both $\s_{U}:U\rightarrow A$ and $\rr_{U}:U\rightarrow f^n(A)$ are homeomorphisms, whose inverses are $a\mapsto (f^n(a),n,a)$ and $a \mapsto (a,n,(f^n)^{-1}(a))$, respectively. Moreover, if $V$ is an open subset of $X$ then $\un(V)=\{(v,0,v):\ v \in V\}=(V\times \{0\}\times V)\ \bigcap \CU$ which is open. Finally, the continuity of $\un$ and $\m$ can be observed noticing that both are compositions of continuous functions (projections, inclusions, the sum in $\nat$).
\end{exe}

From now on, let $\C=\big(\CZ,~\CU\big)$ be an étale category. Let us begin by showing that the range of $\un$ is an open subset of $\CU$. To this end, we present the following lemma whose proof is adapted from \cite[Proposition 3.2]{exel}.

\begin{lemma}\label{lemma: topology}
Let $X$ be a topological space and $Y \cont X$ a topological subspace of $X$. Suppose there exists a local homeomorphism $f: X\rightarrow Y$ such that $f(y)=y$ for every $y \in Y$. Then, $Y$ is open in $X$ and the extended map $f: X \rightarrow X$ is a local homeomorphism.
\end{lemma}

\begin{proof}
Consider $y \in Y$, and let $V$ be an open set of $X$ such that $y\in V$ and $f_{V}:V \rightarrow f(V)$ is a homeomorphism. Define $B$ to be the subset $V\cap f(V)$. Note that $y \in B$ and $B$ is open in $f(V)$. Hence, the subset $f_{V}^{-1}(B)$ is open in $V$. Also, $B$ is equal to $f_V^{-1}(B)$ since $B\cont V$, $f(B)=B$ and $f_V$ is bijective. Therefore, $B$ is open in $V$ and, consequently, $B$ is open in $X$. In conclusion, we have $y \in B \cont Y.$ This gives that $Y$ is open in $X$.
\end{proof}

\begin{prop}\label{prop:coment}
If $\C$ is an étale category, then $\un\big(\CZ\big)=\{\id_u\ :\ u \in \CZ\}$ is an open subset of $\CU$, and hence $\un$ is an open map.
\end{prop}

\begin{proof}
 Note that $\un\circ \rr: \CU\rightarrow \{\id_u \mid u \in \CZ\}$ satisfies the hypothesis of Lemma \ref{lemma: topology}. Hence $\{\id_u \mid u \in \CZ\}$ is open in $\CU$. Moreover, if $U \cont \CZ$ is an open subset, then $\un(U)$ is an open subset of $\un(\CZ)$ which is, in turn, open in $\CU$. Then, $\un(U)$ is open in $\CU$.
\end{proof}

For simplicity of the presentation, we will avoid working with two distinct sets $\CZ$ and $\CU$. In what follows we will identify an object $u\in \CZ$ with its associated identity morphism $\id_u$, and hence we will see  $\CZ$ as a subset of $\CU$. The fact that $\un$ is an embedding ensures that there is no loss of generality in this identification since $\un\big(\CZ\big)$ is homeomorphic to $\CZ$. Furthermore, we will just write $\C$ instead $\CU$. 
%
%
%
\begin{defi}
An open subset $U\cont \C$ is called a \textbf{bisection} if the restrictions $\s_U$ and $\rr_{U}$ are injective. We will denote by $\bisc$ the set of all bisections of $\C$.
\end{defi}

\begin{rem}
	Since local homeomorphisms are open maps, $\bisc$ coincides with the family of open sets $V$ such that  $\s_V:V\rightarrow \s(V)$ and $\rr_{V}:V\rightarrow \rr(V)$ are homeomorphisms.
\end{rem}

In \cite{exel}, the author proves several topological facts about étale groupoids. It occurs that many of these facts are valid (with similar proofs) for the more general context of étale categories.

\begin{prop} Let $\C$ be an étale category. Then:
\begin{enumerate}[{\normalfont \rmfamily  (i)}]
    \item\label{pi1} $\bisc$ forms a basis for the topology of $\C$.

    \item\label{pi2} Every bisection is a locally compact Hausdorff subspace.

    \item\label{pi3} Every open subset of $\C$ is a locally compact subspace.

    \item\label{pi4} Every open Hausdorff subset of $\C$ is a locally compact Hausdorff subspace.

    \item\label{pi5} For all $u,v \in \CZ$, $\C_u$, $\C^v$ and $\C_u^v$ are closed subsets.

    \item\label{pi6} For all $u,v \in \CZ$, $\C_u$, $\C^v$ and $\C_u^v$ are discrete subspaces.

    \item\label{pi7} $\CD$ is a closed subset of $\C\times\C$. 


    \item\label{pi8} If $\C$ is Hausdorff then $\CZ$ is closed in $\C$.
\end{enumerate}
\end{prop}

\begin{proof}
\begin{enumerate}[{\normalfont \rmfamily  (i)}]
    \item Fix $x \in \C$ and $W\cont \C$ an open subset such that $x\in W$. Let $U$ and $V$ be open subsets such that $x\in U\cap V$, $\s_U$ and $\rr_V$ are homeomorphisms. Then, $Z:= U\cap V\cap W$ is a bisection and $x\in Z\cont W$. 

    \item Recall that open subsets of locally compact Hausdorff spaces are also locally compact Hausdorff spaces with the subspace topology. If $U$ is a bisection, then it is homeomorphic to the open subset  $\s(U)$ of $\CZ$ and, consequently, it is a locally compact Hausdorff subspace. 

    \item Let $V$ be an open subset of $\C$ and $x \in V$. By \eqref{pi1} and \eqref{pi2}, there is a bisection $U$ and a compact $K$ such that $x\in\mathring{ K}\cont K\cont U \cont V$. Then $x$ has a compact neighborhood in $V$ and therefore $V$ is locally compact.

    \item Immediate.

    \item $\C_u=\s^{-1}(\{u\})$; $\C^v=\rr^{-1}(\{v\})$; $\C_u^v=\C_u\cap \C^v$.

    \item Fix $u,v\in \CZ$. Take $x \in \C_u$ and $U$ a bisection containing $x$. Then, $U\cap \C_u$ is equal to $\{x\}$ and it is open in $\C_u$. Similarly, $\C^v$ and $\C_u^v$ are discrete.

    \item  The  diagonal set $\Delta(\CZ)=\{(u,u)\mid u\in \CZ\}$ is closed because $\CZ$ is Hausdorff. Moreover, $f: \C\times\C\rightarrow \CZ\times\CZ$ given by $f(x,y)=(\s(x),\rr(y))$ is continuous. Thus, $\CD=f^{-1}\left(\Delta\left(\CZ\right)\right)$ is a closed subset of $\C\times\C$.

    \item Suppose $\C$ is Hausdorff and let $f: \C\rightarrow \C\times\C$ be the map given by $f(x)=(x,\rr(x))$. Note that $f$ is continuous and recall that the diagonal $\Delta(\C)$ is closed. Thus $\CZ=f^{-1}\left(\Delta(\C)\right)$ and therefore $\CZ$ is closed.


\end{enumerate}
\end{proof}

The first item above tells us that the family $\{(U\times V)\cap \CD\ :\ U,V \in \bisc \}$ forms a basis for the topology of $\CD$. Moreover, for $U, V \in \bisc$ note that $(U\times V)\cap \CD= (U_1\times V_1)\cap \CD$, where $U_1=\s_{U}^{-1}\left(\rr(V)\cap \s(U)\right)$ and $V_1=\rr_{V}^{-1}\left(\rr(V)\cap \s(U)\right)$. Hence, $\s(U_1)=\rr(V_1)=\rr(V)\cap \s(U)$ and, consequently, another basis for the topology of $\CD$ is the family $\{(U\times V)\cap \CD\ :\ U,V \in \bisc,\ \s(U)=\rr(V) \}.$ Furthermore, for any pair $U,V$ of subsets of $\C$, we define $UV$ to be the set $\{xy\ : (x,y) \in (U\times V) \cap \CD \}$. Hence, because $\m$ is associative, it becomes clear that the product of subsets is also associative.

\begin{prop}\label{prop: misopen}
The \textbf{composition} function $\m: \CD\rightarrow \C$ is open.
\end{prop}

\begin{proof}
    We start by showing that if $U,V$ is a pair of bisections such that $\s(U)=\rr(V)$ and $UV\cont W$, for another bisection $W$, then $UV$ is open. For $x\in U$, there exists $y \in V$ such that $\s(x)=\rr(y)$. Therefore $\rr(x)=\rr(xy)\in \rr(UV)$ and consequently $\rr(U)=\rr(UV)$. Then, $\rr_W(UV)$ is equal to $\rr(U)$, which is open. Hence, $UV$ is open as $\rr_W$ is a homeomorphism.
 
    Next, we show that the above fact implies $\m$ is open. Let $O$ be an open subset of $\CD$,  $(x,y)$ be in $O$, $W$ be a bisection such that $xy \in W$. Since  $\m$ is continuous, there exists a pair $U,V$ of bisections such that $\s(U)=\rr(V)$, $(x,y)\in U\times V\cap \CD \cont O$ and $UV\cont W$. By the previous case, $UV$ is open and $xy \in UV\cont \m(O)$. Hence, $\m(O)$ is open.
 \end{proof}

\begin{rem}\label{r1}
For bisections $U$ and $V$, Proposition \ref{prop: misopen} says $UV$ is open. Moreover, it is easy to verify that $\s_{UV}$ and $\rr_{UV}$ are injective maps. Therefore, $\bisc$ forms a semigroup or, more precisely, a monoid with identity $\CZ$. 
\end{rem}

\begin{lemma}\label{l1}
Let $\{U_1,\dots,U_n\}$ be a finite family of bisections and $z \in U_1\cdots U_n$. Then, there exists a unique $n$-tuple $(x_1,x_2,\dots,x_n)\in U_1\times ...\times U_n$ such that $x_1\dots x_n=z$. \end{lemma} 

\begin{proof}
 The proof is by induction on $n$. The case $n=1$ is trivial. Suppose the statement holds for $n=k-1$, and let $U_1,U_2,...,U_{k}$ be a family of bisections. For $k$-tuples $(x_1,...,x_k), (y_1,...,y_k)\in U_1\times ...\times U_k$ such that $x_1...x_k=y_1...y_k$, applying $\rr$ to both sides we get $x_1=y_1$ since $U_1$ is a bisection. Therefore, $\rr(x_2...x_{k})=\s(x_1)=\s(y_1)=\rr(y_2...y_{k})$ which implies that $x_2...x_{k}= y_2...y_{k}$. Hence by the induction hypothesis $x_i=y_i$, for all $1\leq i\leq k$.
\end{proof}

\subsection{Operator algebras associated with étale categories}

Given an open subset $V \cont \C$ and $f:\C\rightarrow \comp$, we will write $f \in \cc(V)$ if $f_{V} \in \cc(V)$ and $f_{\C\sm V}$ is identically zero.  As in \cite{exel}, let $\mathcal{U}$ to be the family of open Hausdorff subsets of $\C$, and define the vector space
    $$\CCC=\sps\{f:\ f \in \cc(U),\ U\in \mathcal{U} \}.$$
    
 We write $\CCC$ instead of $\cc(\C)$ because the elements of the former are not necessarily continuous. However, if $\C$ is Hausdorff then they coincide. 
  
 \begin{prop}\label{prop: sumofbisections}
    If $\mathcal{F}\cont \bisc$ covers $\C$, then $\CCC=\sps\{f:\ f \in \cc(U),\ U\in \mathcal{F} \}.$ In particular, one has                $$\CCC=\sps\{f:\ f \in \cc(U),\ U\in \bisc\}.$$ \end{prop}
 
\begin{proof}
Let $V$ be an open Hausdorff subset and $f\in \cc(V)$. There is a finite family $\{U_i\}_{i=1}^{n}$ of bisections in $\mathcal{F}$ such that $\supp(f)\cont \cup_{i=1}^{n} U_i$. Using partitions of unity (see \cite[Theorem 2.13]{rudin}), there are functions $\eta_i \in \cc(U_i\cap V)$ such that $\sum_{i=1}^n\eta_i(x)=1,$ for every $x\in \supp(f)$. Then, we have $f=\sum_{i=1}^nf\eta_i$ and $f\eta_i\in \cc(U_i\cap V)\cont \cc(U_i)$, for all $i \in \{1,...,n\}$.
\end{proof}
 For $f,g \in \CCC$, define the convolution product $f \ast g$ to be the function
    \begin{equation}\label{eq: convolutionproduct}
    f\ast g (z) = \sum_{(x,y)\in M_z}f(x)g(y).
    \end{equation}
We claim that $\CCC$ is an associative algebra endowed with this product. We will make use of the following lemma to show that.

\begin{lemma}\label{lemma: productwelldefined}
Suppose $U,V$ and $W$ are bisections and $f\in \cc(U)$, $g \in \cc(V)$ and $h \in \cc(W)$ then:
 \begin{enumerate}[{\normalfont \rmfamily  (i)}]
    \item\label{li1}  $|\supp(f)\cap \C_u|\leq 1$ and $|\supp(f)\cap \C^u|\leq 1 $ for every $u \in \CZ$,

    \item\label{li2} $f\ast g\in \cc(UV)$.

    \item\label{li3} If $U=\CZ$, then $f\ast g \in \cc(V)$ and  $(f\ast g)(z)=f(\rr(z))g(z)$, for every $z\in V$.

    \item\label{li4} If $V=\CZ$, then $f\ast g \in \cc(U)$ and $(f\ast g)(z)=f(z)g(\s(z))$, for every $z\in U$.

    \item\label{li5} If $U=V=\CZ$, then $f\ast g \in \cc(\CZ)$ and $(f\ast g)(z)=f(z)g(z)$, for every $z\in \CZ$.

    \item\label{li6} $(f\ast g)\ast h=f\ast (g\ast h)$.
 \end{enumerate}
\end{lemma}

\begin{proof}
\begin{enumerate}[{\normalfont \rmfamily  (i)}]
    \item 
    Immediate.
            
    \item 
    By Lemma \ref{l1}, given $z\in UV$, there is a unique pair $(x,y)$ belonging to $M_z\cap (U\times V)$. Consequently, $f\ast g(z)$ is equal to $f(x)g(y)$. Now, writing $x$ and $y$ as functions of $z$, we obtain 
        $$f\ast g(z)= f\circ\rr_U^{-1}\circ \rr(z). g\circ \s_V^{-1}\circ\s(z).$$ 
    
    Thus, the restriction of $f\ast g$ to $UV$ is continuous.
    If $f\ast g(z) \neq 0$, there must be a pair $(x,y)\in M_z$ such that $f(x)g(y)\neq 0$. Then, $x \in U$, $y\in V$ and, consequently, $z \in UV$. Moreover, $\{w \in \C\ |\ f\ast g(w)\neq 0\} \cont \supp(f)\supp(g)\cont UV$. Since UV is Hausdorff, $\supp(f)\supp(g)$ is compact and closed. Hence, $\supp(f\ast g)$ is compact.
  
    \item 
    Given $z\in V$, $(\rr(z),z)$ is the unique element in $(U\times V)\cap M_z$, then $f\ast g(z)=f(\rr(z))g(z)$.
    
    \item 
    Similar to the proof of \eqref{li3}. 
    
    \item 
    Immediate from \eqref{li3} and \eqref{li4}.
    
    \item 
    Note that $(f\ast g)\ast h$ and $f\ast (g\ast h)$ are both supported in $UVW$, by item \eqref{li2}. Moreover, for $z\in UVW$ there are unique pairs $(x,x_3)\in (UV\times W)\cap M_z$, and $(y_1,y)\in (U\times VW)\cap M_z$. Hence $(f\ast g)\ast h(z)=f\ast g(x)h(x_3)\text{ and }f\ast (g\ast h)(z)=f(y_1) g\ast h(y).$
    
    Now, since there are unique pairs $(x_1,x_2)$ and $(y_2,y_3)$ belonging to $(U\times V)\cap M_x$ and $(V\times W)\cap M_{y}$, respectively, we obtain  $(f\ast g)\ast h(z)=f(x_1)g(x_2)h(x_3)$ and $f\ast (g\ast h)(z)=f(y_1) g(y_2) h(y_3).$
    Thus, by Lemma \ref{l1}, we have $x_i=y_i$ for $1\leq i\leq 3$, completing the proof.
\end{enumerate}
\end{proof}

Since the product defined in Equation \eqref{eq: convolutionproduct} is bilinear, we just need to ensure that the sum appearing in Equation \eqref{eq: convolutionproduct} is finite and that $\ast$ is associative to obtain that it is well-defined.  The finiteness comes from combining Proposition \ref{prop: sumofbisections} with item \eqref{li1} of Lemma \ref{lemma: productwelldefined}. The associativity comes from combining the same proposition with item \eqref{li6} of Lemma \ref{lemma: productwelldefined}. Therefore, $\CCC$ is an associative algebra.

\begin{rem}
We can generalize items \eqref{li3} and \eqref{li4} of Lemma \ref{lemma: productwelldefined}. In fact, for $f\in \CCC$ and $g \in \cc(\CZ)$, we have $f\ast g(z)=f(z)g(\s(z))$ and $g\ast f(z)=g(\rr(z))f(z)$, for every $z\in \C$. To see this, just write $f$ as a sum $\sum_{i=1}^{n}f_i$, where each $f_i$ is supported on a bisection $U_i$, and apply the aforementioned items of Lemma \ref{lemma: productwelldefined}.
\end{rem}

\begin{defi}\label{definicaorepcate}
 A \textbf{representation} of $\CCC$ on a Hilbert space $H$ is an algebra homomorphism $\pi:\CCC\rightarrow B(H)$ such that:
 	\begin{enumerate}[{\normalfont \rmfamily  (i)}]
    \item\label{i1definicaorepcate}
        $||\pi(f)||\leq ||f||_{\infty}$, for every $f\in \ccu$ and for every $U\in \bisc$.
     \item\label{i2definicaorepcate} 
         $\pi(\bar{f})=\pi(f)^*$, for every $f\in \cc(\CZ)$.
 	\end{enumerate}
The class of all representations of $\CCC$ will be denoted by $\mathcal{R}(\CCC)$.
\end{defi}	
  
\begin{rem}\label{rem: Pi_to_pi}
A representation $\pi:\CCC\rightarrow B(H)$ gives rise to a $C^*$-algebra homomorphism $\pi_0:\czero(\CZ)\rightarrow B(H)$, where $\pi_0$ is obtained by extending the restriction $\pi_{|\cc(\CZ)}$.
\end{rem}

Our goal is to build a normed algebra from $\C$ in such a way every representation of $\CCC$ becomes a contractive map. Note that every representation $\pi$ of $\CCC$ generates a seminorm $\|f\|_\pi:=\|\pi(f)\|$, and hence we can form the set $\mathcal{S}(\CCC)$ of all seminorms on $\CCC$ arising from representations of $\CCC$. Define 
\begin{equation}\label{eq:universalnormcategory}||f||_0
:=\sup_{||\cdot|| \in \ \mathcal{S}(\CCC)} ||f||.\end{equation}
The above supremum is well-defined by \eqref{i1definicaorepcate} of Definition \ref{definicaorepcate}. Moreover, the following holds
\begin{equation}\label{eq:universalnormcategory2}
||f||_0= \sup_{\pi \in \ \mathcal{R}(\CCC)} ||\pi(f)||.
\end{equation}
We chose to define $||\cdot||_0$ in terms of seminorms 
 because $\mathcal{S}(\CCC)$ is a set whereas $\mathcal{R}(\CCC)$ is not necessarily set which may cause some confusion when taking the supremum. In general, $||\cdot||_0$ is just a semi-norm, and it is a norm if there exists a faithful representation of $\CCC$. Moreover, let $\{\pi_\lambda\}_{\lambda\in \Lambda}$ be a set of representations $\pi_\lambda:\CCC\rightarrow B(H_\lambda)$ formed by choosing one generator for every seminorm on $\mathcal{S}(\CCC)$. Thus taking the direct sum representation $\Pi:=\oplus_{\lambda\in \Lambda}\pi_\lambda:\CCC \rightarrow B \left(\oplus_{\lambda\in \Lambda}H_\lambda\right)$ we have that $\|\cdot\|_0$ is in $\mathcal{S}(\CCC)$ since $||f||_0:=\sup_{\lambda \in \Lambda} ||\pi_{\lambda}(f)||= \|\Pi(f)\|.$

 If $A$ is an algebra over $\comp$ and $\|\cdot\|_A:A\rightarrow \mathbb{R}_+$ is a seminorm, recall that the standard procedure to obtain a normed algebra from the pair $(A,\|\cdot\|_A)$ is to take the quotient of $A$ by the ideal $N=\{ a \in A \mid \|a\|_A=0\}$ and equip $\faktor{A}{N}$ with the norm $\|[a]\|=\|a'\|_A$, where $a'$ is any representative of $[a]$. The normed algebra obtained by completing $(\faktor{A}{N}, \|\cdot\|)$ is called the \textbf{Hausdorff completion} of $(A,\|\cdot\|_A)$.
 
  \begin{defi}
Let $\C$ be an étale category. The \textbf{operator algebra} of $\C$ is the Hausdorff completion $\A(\C)$ of the normed algebra $\left(\CCC,\|\cdot\|_0\right)$.
\end{defi}

Note that $\Pi$ induces an isometric homomorphism of $\A(C)$ on $B \left(\oplus_{\lambda\in \Lambda}H_\lambda\right)$ justifying the term \textit{operator algebra}. Moreover, note that every representation of $\pi:\CCC\rightarrow B(H)$ extends to a contractive homomorphism to $\pi:\A(\C)\rightarrow B(H)$.

\subsubsection{The reduced operator algebra of an étale category}

We now proceed to show that if $\C$ is a left cancellative étale category then there exists a faithful representation of $\CCC$. This gives that $\A(\C)$ is simply the completion of $(\CCC,\|\cdot\|_0)$. We start by recalling the definition of the Hilbert space associated with a set.

\begin{defi}\label{def:hilbertspaceofset}
Let $X$ be a set. We define $$\ell_2(X)=\left\{ f:X\rightarrow \mathbb{C} \mid \sum_{x\in X}|f(x)|^2<+\infty \right\}$$ to be the \textbf{Hilbert space associated with} $X$. The inner product on $\ell_2(X)$ is given by $\pint{f}{g}=\sum_{x\in X}f(x)\overline{g(x)}.$ Furthermore, if $\delta_x$ denotes the characteristic function of the singleton $\{x\}$, then $\{\delta_x\}_{x\in X}$ is an orthonormal basis for $\ell_2(X)$ and we can also write  $$\ell_2(X)=\left\{ \sum_{x\in X}a_x\delta_x \mid \sum_{x\in X}|a_x|^2<+\infty \right\}.$$
\end{defi}

Define $\pi:\CCC \rightarrow B(\ell_2(\C))$ to be the map such that, for every $f \in \CCC$, $\pi_f:= \pi(f)$ is given by
\begin{equation}\label{formularreprregularcategorias}\pi_f\left(\sum_{z \in \C} a_{z}\delta_{z}\right)=\sum_{z \in \C} a_{z}\sum_{x \in \C_{\rr(z)}}f(x)\delta_{xz}.\end{equation} 
Our strategy is to prove that for the bisections $U$ and $V\in \bisc$, and the maps $f \in \ccu$ and $g\in \cc(V)$, we have that $\|\pi_f\|\leq \|f\|_{\infty}$ and $\pi_{f\ast g}=\pi_f\pi_g$. Thus, the linearity of $\pi$ will ensure that it is a well-defined representation of $\CCC.$

Note that $\pi_f(\delta_z)$ is nonzero if and only if $\C_{\rr(z)}\cap f^{-1}\left(\{0\}\right)$ is non-empty. 
In such case $\pi_f(\delta_z)=f(x^{z})\delta_{x^{z}z}$, where $x^{z}$ is the unique element in $\C_{\rr(z)}\cap f^{-1}\left(\{0\}\right)$, by item \eqref{li1} of Lemma \ref{lemma: productwelldefined}. 
Now, if $\Gamma$ is the auxiliary set $\left\{u\in \CZ\mid \C_u\cap f^{-1}\left(\{0\}\right)\neq \emptyset\right\}$ and $v=\sum_{z \in \C} a_{z}\delta_{z}$ we have
    $$\pi_f\left(v\right)=\sum_{r(z)\in \Gamma} a_{z}f(x^{z})\delta_{x^{z}z}.$$
Suppose that $z$ and $z'$ are such that $\rr(z)\in \Gamma$, $\rr(z')\in \Gamma$, and $x^{z}z=x^{z'}z'$. 
In this case, $x^{z}\in U$, $x^{z'}\in U$, and $\rr(x^{z})=\rr\left(x^{z'}\right)$. 
Thus, $x^{z}=x^{z'}$ and $z=z'$ since $\C$ is left cancellative. Hence
    $$\left\|\pi_f\left(v\right)\right\|^2=\sum_{\substack{r(z) \in \Gamma}} |a_{z}|^2|f(x^{z})|^2\leq||f||_{\infty}^2 ||v||^2.$$
    
To see that $\pi_{f\ast g}=\pi_f\pi_g$, we prove $\pi_{f\ast g}(\delta_{z})=\pi_f\pi_g(\delta_{z})$ for every $z \in \C$. 
Suppose that $\pi_f(\pi_g(\delta_{z}))$ is nonzero. Then, $\pi_f\pi_g(\delta_{z})=f(x)g(y)\delta_{xy z}$,
for $x\in f^{-1}(\{0\})\cont U$, and $y\in g^{-1}(\{0\}) \cont V$. Note that $xy$ is the unique element in $\C_{\rr(z)}\cap UV$ and hence 
    $$\pi_{f\ast g}(\delta_{z})=f\ast g(xy)\delta_{xy z}=\pi_f(\pi_g(\delta_{z})).$$
On the other hand, suppose $\pi_{f\ast g}(\delta_{z})$ is nonzero. In this case $\pi_{f\ast g}(\delta_{z})=f\ast g(w)\delta_{w z}$, for the unique element $w$ in $\C_{\rr(z)}\cap (f\ast g)^{-1}(\{0\})$. Moreover, there exists a unique pair $(x,y)\in U\times V$ such that $f\ast g(w)=f(x)g(y)$, by Lemma \ref{l1}. Finally, $f(x)$ and $g(y)$ are nonzero and hence
    $$\pi_f(\pi_g(\delta_{z}))=f(x)g(y)\delta_{xy z}=f\ast g(w)\delta_{w z}=\pi_{f\ast g}(\delta_{z}).$$
We have proven that $\pi_{f\ast g}(\delta_{z})$ is nonzero if, and only if, $\pi_f(\pi_g(\delta_{z}))$ is nonzero and in such case they coincide. Therefore, $\pi_{f\ast g}(\delta_{z})=\pi_f\pi_g(\delta_{z})$ for every 
$z\in \C$.

The straightforward calculation below shows that $\pi_{\overline{f}}=\pi_f^*.$ For $f\in \cc(\CZ)$ and $z,z'\in \C$ we have 
$$\begin{aligned}
\pint{\pi_f(\delta_{z})}{\delta_{z'}}
&=f(\rr(z))\pint{\delta_{z}}{\delta_{z'}}
=[z=z']f(\rr(z))
\\&=[z=z']f(\rr(z'))
=f(\rr(z'))\pint{\delta_{z}}{\delta_{z'}}
\\&=\pint{\delta_{z}}{\overline{f(\rr(z'))}\delta_{z'}}=\pint{\delta_{z}}{\pi_{\overline{f}}(\delta_{z'})}.\end{aligned}$$ 

Where in the above calculations $[z=z']$ denotes the boolean value of the expression inside the brackets. Recall that for a statement $P$ the \textbf{boolean value} of $P$ is   
     $$[\ P \ ]=
     \begin{cases}
     1,& \text{ if P is true}\\
     0,& \text{ otherwise}
     .\end{cases}$$

We finish this discussion by proving that $\pi$ is faithful. If $f\in \CCC$ is nonzero, there is $z\in \C$ such that $f(z)$ is nonzero and therefore $$\pi_f(\delta_{\s(z)})=\sum_{x\in \C_{\rr(\s(z))}}f(x)\delta_{x \s(z)}=\sum_{x\in \C_{\s(z)}}f(x)\delta_{x}= f(z)\delta_z + \sum_{\substack{x\in \C_{\s(z)}\\ x \neq z}}f(x)\delta_{x}\neq 0.$$ 

\begin{defi}\label{def:repregularcategory}
	The above defined map $\pi$ is called the \textbf{regular representation of $\C$}. The \textbf{reduced operator algebra of $\C$} is the closure of $\pi(\CCC)$ in $B(\ell_2(\C))$ and it is denoted by $\A_{r}(\C)$.
\end{defi}

Of course, we could also define the reduced operator algebra intrinsically as the completion of $\CCC$ in the norm induced by the regular representation. Moreover, the regular representation is usually presented as a direct sum. If one desires to recover the definition in terms of a direct sum of representations, it suffices to  note that for any $u \in \CZ$ the Hilbert (sub)space $\ell_2(\C_u)$ is an invariant subspace of $\pi$ and $\ell_2(\C)=\bigoplus\limits_{u\in \CZ}\ell_2(\C_u)$. Therefore, denoting by $\pi_u$ the restriction of $\pi$ to $\ell_2(\C_u)$ we obtain \begin{equation}\label{regularsomadereta}
\pi=\bigoplus_{u\in \CZ}\pi_u.
\end{equation}

Note that the cancellation assumption is crucial for the existence of the left regular representation. In fact, let $\C=\nat$ be the category of natural numbers, where $1$ is the unique object of $\C$ and the composition is given by multiplication. The topology in $\C$ is the discrete one and hence $\CCC$ is simply the set of all functions whose support is finite. Let $f=\chi_{0}$ be the characteristic function of $\{0\}$. Note that $\pi_f(\delta_1+\delta_2)=2\delta_0$, which implies that $\|\pi_f\|\geq \sqrt{2}.$ But, since $f$ has support contained on a bisection, we should have $\|\pi_f\|\leq \|f\|_{\infty}=1.$
\begin{prop}\label{ccucontidoalgebra}
Let $\C$ be a left cancellative étale category and $\pi$ be the regular representation. Then,  $||\pi(f)||=||f||_{\infty}$ for any bisection $U$ and  $f\in \cc(U)$. In particular, $\czero(U)$ is a linear subspace of both $\A_{r}(\C)$ and $\A(\C)$. Furthermore, if $U=\CZ$ then $\czero(\CZ)$ is a subalgebra of $\A_{r}(\C)$ and $\A(\C)$.
\end{prop}  

\begin{proof}
Let $U$ be a bisection, $f\in \cc(U)$ and $x\in U$ be such that $|f(x)|=||f||_{\infty}$. Note that $\pi_f(\delta_{\s(x)})=f(x)\delta_{x}$ and, consequently, $||f||_{\infty}=|f(x)|\leq ||\pi(f)||\leq||f||_{\infty}.$ This proves that $f$ attains its supremum norm on $\A_r(\C)$ and $\A(\C)$ (cf.~ item \eqref{i1definicaorepcate} of Definition \ref{definicaorepcate}). In conclusion, $\ccu$ is an isometric linear subspace of both the full and reduced algebras of $\C$. The result then follows on from taking closure.

If $U=\CZ$, item \eqref{li5} of Lemma \ref{lemma: productwelldefined} shows that the convolution product of functions supported on $U$ reduces to the pointwise product. Therefore $\cc(\CZ)$ is a subalgebra of $\A_r(\C)$ and $\A(\C)$, as well as its closure $\czero(\CZ)$.
\end{proof}

\subsection{Graph algebras}

Let $E=(E^0,E^1,d,r)$ be a directed graph. For $n\geq2$, let $E^n$ denote the set of paths of length $n$ in $E$, that is,
$$E^n=\{a_1\cdots a_n|\{a_i\}_{i=1}^n\cont E^1,\ d(a_{i})=r(a_{i+1})\}.$$
The \textit{category of the graph} $E$ is the pair $\C_E=\big(\CZ,~\CU\big)$ where $\CZ=E^0$, $\CU=\bigcup_{n \in \nat}E^n$ and the structure maps are as stated below. the unit map $\un$ is simply the inclusion of $E^0$ in $E^1$, the \textit{source} and \textit{range} maps are given by  $$\s(x)=\left\{\begin{array}{cc}
	x,&\ \text{if}\ x\in E^0\\
	d(a_n),&\ \text{if}\ x=a_1\cdots a_n,\ n\geq 1
\end{array} \right. 
\text{ and} \qquad 
\rr(x)=\left\{\begin{array}{cc}
 x,&\ \text{if}\ x\in E^0\\
 r
 (a_1),&\ \text{if}\ x=a_1\cdots a_n,\ n\geq 1
 .\end{array}\right.$$

Recall that if $x$ and $y$ are finite paths in $E$, then $xy$ denotes the path obtained by concatenating $x$ and $y$. That said, we define the composition map to be

 $$\m(x,y)=\left\{\begin{array}{ccc}
x,&\ \text{if}\ y\in E^0&\\
y,&\ \text{if}\ x\in E^0&\\
xy,&\ \text{otherwise. }&
\end{array}\right.$$

Thus, equipping $\CZ$ and $\CU$ with the discrete topology, it is straightforward to prove that $\C_E$ is a cancellative étale category. 

For a path $x \in \C_E$, define $L_{x} \in B(\ell_2(\C_E))$ by $$L_{x}(\delta_{y})=\left\{\begin{array}{cc}
\delta_{xy},&\ \text{if}\ (x,y) \in \CD\\
0,&\ \text{otherwise}
.\end{array}\right.$$
In \cite{kat}, one can find the following definition.

\begin{defi}
The \textbf{tensor algebra of $E$} is the closed algebra generated by the family $\{L_{x}\mid\ x \in \C_E\}\ \cont B(\ell_2(\C_E))$ and it is denoted by $\mathcal{T}^{+}(E)$.
\end{defi}
Clearly, $\mathcal{T}^{+}(E)$ can also be viewed as the closed algebra generated by the smaller family $\big\{L_{x}\mid x \in E^0 \cup E^1 \big\}$ since $L_x=L_{a_1}\cdots L_{a_n}$ if $x=a_1\cdots a_n$. We now show that the tensor algebra of E is isomorphic to the reduced operator algebra of $\C_E$.

\begin{prop}\label{pro:tensorregular}
$\mathcal{T}^{+}(E)=\A_r(\C_E)$.
\end{prop}

\begin{proof}
Let $\pi:\CCC\rightarrow B(\ell_2(\C))$ be the regular representation of $\CCC$, $x$ be an element of $\C$ and $f$ be the characteristic function of $\{x\}$. Note that
 $$\begin{aligned}
 \pi_f(\delta_{z})
 &=\sum_{y\in \C_{\rr(z)}}f(y)\delta_{yz}
  =\left[(x,z) \in \CD\right]\ \delta_{xz}
  =L_x(\delta_{z}).\end{aligned}$$
Moreover, note that if $U$ is a bisection and $g\in \ccu$ then $g=\sum_{j=1}^na_i\chi_{\{x_j\}}$ and hence $\pi_g=\sum_{j=1}^na_iL_{x_i}$. The result then follows.
\end{proof}

\begin{exe}
Let $E$ be the graph of a unique vertex $e_0$ and a unique loop $e_1$. For each $n\geq 1$, there exists a unique path of length $n$, namely, $e_n=\overbrace{e_1e_1...e_1}^{n\text{ times}}$. Let $\delta_n$ denote the characteristic function of the singleton $\{e_n\}$ and $\C_E=\big(\CZ,~\CU\big)$ be the category of $E$. Note that $\CZ=\{e_0\}$, $\CU=\{e_0,e_1,e_2,\cdots\}$ and $\bis(\C_E)=\{\{e_n\}\mid n \in \nat\}$. In particular, $\bis(\C_E)$ and $\nat$ are isomorphic as monoids since $\{e_n\}\{e_m\}=\{e_{n+m}\}$. 

If $f$ is a nonzero function in  $\A_0(\C_E)$ whose support is contained in a bisection, then $f$ is a multiple of some $\delta_n$. Thus
$$\A_0(\C_E)=\sps\left\{\delta_i\mid i\in \nat\right\},$$
 and $\A_0(\C_E)$ is isomorphic to the polynomial algebra $\comp[x]$, where $\delta_1 \mapsto x$ and $\delta_0\mapsto 1$. Our goal is to describe the representations of $\comp[x]$.

\begin{prop}\label{prop:caracterizingrepsofce}
The representations of $\comp[x]$ are in correspondence with pairs $(P,T)$, where $P\in B(H)$ is a projection, $T\in B(H)$ is a contraction and $PT=T=TP$.
\end{prop}

\begin{proof}
Suppose $\pi:\comp[x]\rightarrow B(H)$ is a representation in the sense of Definition \ref{definicaorepcate}. Then, $\pi(x)$ is a contraction and $\pi(1)$ is a projection since it is idempotent and self-adjoint. 
Conversely if one has in hands a contraction $T$ and a projection $P$ such that $PT=T=TP$ then mapping $1\rightarrow P$ and $x\rightarrow T$ gives rise to a representation of $\comp[x]$.
\end{proof}

Let $H$ be a Hilbert space, $\pi$ be the representation generated by a pair $(P,T)$ and $\pi_1$ be the representation generated by $(\id_H,T)$. For a polynomial $p=\sum_{i=1}^na_ix^i$, the following holds
\begin{equation}\label{eq:wlogp=id}
    \begin{aligned}
\|\pi(p)\|
&=||a_0P+\sum_{i=1}^n a_iT^i||
=||a_0P+\sum_{i=1}^n a_i(PT)^i||
\\&=||P(a_0\id_H+\sum_{i=1}^n a_iT^i)||
\\&\leq||a_0\id_H +\sum_{i=1}^n a_iT^i||
=\|\pi_1(p)\|.
\end{aligned}
\end{equation}
Therefore, $\pi_1$ dominates $\pi$ and without loss of generality we may assume that every representation has the form of $\pi_1$, that is, $p\mapsto p(T)$, for a fixed contraction $T\in B(H)$. 
Through the identification $e_n \mapsto n$, $\ell_2(\C_E)$ is equal to the canonical Hilbert space $\ell_2(\nat)$, and we may consider the regular representation as a map $\pi:\comp[x]\rightarrow B(\ell_2(\nat))$. In Proposition \ref{pro:tensorregular}, we proved that $\pi(x)=L_1$. Note that $L_1(\delta_n)=\delta_{n+1}$, then $\pi(x)$ is the {(forward) shift operator} $S \in \B(\ell_2(\nat))$. Consequently, $\A_r(\C_E)$ is the closed subalgebra of $B(\ell_2(\nat))$ generated by the identity and $S$, which is called the \textit{disc algebra} and is a subalgebra of the Toeplitz $C^*$-algebra, i.e. the $C^*$-algebra generated by $S$.

We now show that $\A(\C_E)$ and $\A_r(\C_E)$ coincide. Let $\mathbb{D} $ be the closed unitary ball of $\comp$. Recall that $\sigma(S)=\mathbb{D}$, where $\sigma(S)$ denotes the \textit{spectrum} of the aforementioned shift operator $S$. For a polynomial $p \in \comp[x]$, define $\|p\|=\sup_{z\in \mathbb{D}} |p(z)|$. Von Neumann \cite{von} proved his famous inequality asserting that for a contraction $T$ on a Hilbert space, it always holds that $||p(T)||\leq \|p\|.$ Conversely, according to the spectral mapping theorem, we have:
$p(\mathbb{D})=p(\sigma(S))=\sigma(p(S)),$
which leads to $\|p\|=\|r(p(S))\|\leq \|p(S)\|$. Then, combining the facts listed above we get $$\|p\|_0=\sup\{\ \|p(T)\|\mid T \in B(H),\ ||T\|\leq 1\} \leq \|p\|\leq ||p(S)||.$$ Thus  $$\A(\C_E)=\A_r(\C_E)=\mathcal{T}^{+}(E).$$
\end{exe} 

\begin{exe}
Let $E$ be the graph of one vertex $e_0$ and $m$ loops $e_1$,...,  $e_m$. It is straightforward to verify that $\A_0(\C_E)$ is the algebra of noncommutative polynomials in $m$ variables, $\comp\{x_1,...x_m\},$ and that a representation of $\A_0(\C_E)$ on a Hilbert space $H$ corresponds to a $(m+1)-$tuple $(P,T_1,...,T_m)$ where $P$ is a projection, $T_i$ is a contraction and $PT_i=T_i=T_iP$, for every $i\in\{1,...,m\}$. A calculation similar to \eqref{eq:wlogp=id} proves that we may assume $P=\id_H$ without loss of generality. Next, we pass to describe the regular representation of $\comp\{x_1,...x_m\}$. Note that $\C_E=\bigcup_{n\in \nat} E^n$ and therefore
$$\ell_2\big(\C_E\big)=\bigoplus_{n\in \nat}\ell_2\big(E^n\big).$$ 
Denoting $\delta_{e_i}$ by $\delta_i$, we obtain that the space $\ell_2\big(E^0\big)$ is equal to $\comp\delta_0\cong \comp$. Moreover, $H:=\ell_2(E^1)$ is isomorphic to $\comp^m$ since its basis is $\{\delta_1,...,\delta_m\}$. For $n\geq 2$, note that a basic element of $\ell_2\big(E^n\big)$ is of the form $\delta_{e_{i_1}\cdots e_{i_n}}$. Hence, identifying $\delta_{e_{i_1}\cdots e_{i_n}}$ with the tensor $\delta_{{i_1}}\otimes\cdots \otimes\delta_{{i_n}}$, we have that $\ell_2\big(E^n\big)=H^{\otimes n}$. Therefore
$$\ell_2\big(\C_E\big)=\comp\oplus \bigoplus_{n\geq 1}H^{\otimes n}.$$
This Hilbert space is known as the \textit{full Fock space} of  $H$. Moreover, through this identification the concatenation operators $L_{e_j}$ become left creation operators $L_j$, where $L_{j}(\delta_0)=\delta_j$ and $L_{j}(\delta_{{i_1}}\otimes\cdots \otimes\delta_{{i_n}})=\delta_j\otimes \delta_{{i_1}}\otimes\cdots\otimes\delta_{{i_n}}$, for every $n\geq1$. Therefore $\A_r(\C_E)$ is the closed unital subalgebra of $B\big(\comp\oplus \bigoplus_{n\geq 1}H^{\otimes n}\big)$ generated by the left creation operators $L_1,...,L_m$. This algebra is called the noncommutative disc algebra $\A_m$ and it was studied by Popescu in \cite{popescu1991neumann,popescu}.

Consider the polynomial $p(x_1,...,x_m)=x_1+...+x_m$ and note that its reduced norm is $\sqrt{m}$ since 
$L_1,...,L_m$  are isometries with pairwise orthogonal ranges. On the other hand, the norm of $p$ in the full algebra $\A(\C_E)$ is $m$ since we can create a representation mapping each monomial $x_i$ to $\id_H$. This gives that $\A_r(\C_E)$ and $\A(\C_E)$ are not isomorphic. 
\end{exe}

\subsection{Relational covering morphisms} 

In this subsection, we set aside our convention and treat the objects and morphisms as two distinct sets. We introduce the class of relational covering morphisms between étale categories as described in \cite[Section 7.2]{lawson}, followed by a demonstration of the functorial nature of the construction of $\A(\C)$. Throughout our discussion, $\PA(X)$ represents the power set of $X$.

\begin{defi}
	Let $\C=\big(\CZ,~\CU\big)$ and $\D=\big(\DZ,\DU\big)$ be étale categories. A \textbf{relational covering morphism} from $\C$ to $\D$ is a pair $(\varphi_0,\varphi_1)$ where $\varphi_0:\CZ\rightarrow \DZ$ is a proper continuous function and $\varphi_1:\CU\rightarrow \PA\big(\DU\big)$ is a function such that the following conditions hold:
		\begin{enumerate}
		\item[(M1)] $\un(\varphi_0(v))\in \varphi_1(\un(v))$, for every $v \in \CZ$.
		\item[(M2)] If $b\in \varphi_1(a)$ then $\s(b)=\varphi_0(\s(a))$ and $\rr(b)=\varphi_0(\rr(a))$, for every $a\in \CU$. 
		\item[(M3)] If $c\in \varphi_1(a)$ and $d\in \varphi_1(b)$ then $cd \in \varphi_1(ab)$, for every $(a,b) \in \CD$.
		\item[(M4)] If $\s(a)=\s(b)$ (resp. $\rr(a)=\rr(b)$) and $\varphi_1(a)\cap\varphi_1(b)$ is non-empty then $a=b$, for every $a$ and $b$ in $\CU$.
		\item[(M5)] If $\s(x)=\varphi_0(v)$ (resp. $\rr(x)=\varphi_0(v)$) then there exists $a \in \C_v$ (resp. $a \in \C^v$) such that $x\in \varphi_1(a)$,  for every $v \in \CZ$ and $x \in \DU$.
		\item[(M6)] If $A\in \bis(\D)$ then $\achap:=\{z\in \CU:\ \varphi_1(z)\cap A\neq \emptyset \}\in \bis(\C)$.
	    \end{enumerate}
\end{defi}
Throughout the following let $\C$ and $\D$ be étale categories and let $\varphi=(\varphi_0,\varphi_1):\C\rightarrow \D$ be a relational covering morphism.

\begin{prop}\label{prop:coveringmorphism}
    Let $A\in \bis(\D)$ be a bisection and $f$ be a map in $\cc(A)$. Then, the map $\widehat{f}: \C\rightarrow \comp$, given by $\widehat{f}(z)=\sum_{x \in \varphi_1(z)}f(x)$, is in $\cc(\achap).$
\end{prop} 

\begin{proof}
  If $\widehat{f}(z)$ is nonzero, then there exists $x \in \varphi_1(z)$ such that $f(x)\neq 0$. In particular, $x \in A$ and, consequently, $x \in \varphi_1(z)\cap A$, $z\in \achap$ and $\widehat{f}_{\C\sm \hat{A~}}\equiv 0$.
	
Let us check that $\widehat{f}_{\hat{A~}}$, the restriction of $\widehat{f}$ to $\achap$, is a continuous function with compact support. Note that if $z\in \achap$ then $\varphi_1(z)\cap A$ has an unique element since $A$ is a bisection and (M2) holds. If $x_z$ denotes the unique element of $\varphi_1(z)\cap A$, then $\widehat{f}(z)=f(x_z)$. In addition, for an open subset $B\cont \comp$, note that $z\in \widehat{f}_{\hat{A~}}^{-1}(B)$ if and only if $x_z\in f_{A}^{-1}(B)$, and $x_z\in f_{A}^{-1}(B)$ if and only if $z \in \widehat{f_{A}^{-1}(B)}$, by (M6). Combining these two equivalences, we obtain $\widehat{f}_{\hat{A~}}^{-1}(B)= \widehat{f_{A}^{-1}(B)}$ which is open, again by (M6). This proves the continuity of $\widehat{f}$.
  
 To see that $\widehat{f}_{\achap}$ has compact support, note that
 $\rr_{\hat{A~}}\big(\{z: \widehat{f}(z)\neq 0\  \}\big )$ is a subset of $\varphi_0^{-1}\Big(\rr_{A}(\supp(f))\Big),$ 
   and that $\varphi_0^{-1}\big(\rr_{A}(\supp(f))\big)$ is compact since $\varphi_0$ is proper and $\rr_{\hat{A~}}$ is a homeomorphism. Therefore, $\rr_{\hat{A}}\big(\supp(\widehat{f})\big)$ is compact since
   $$ \rr_{\hat{A~}}\big(\supp(\widehat{f})\big)=\overline{\rr_{\hat{A~}}(\{z: \widehat{f}(z)\neq0\  \})}\cont \varphi_0^{-1}(\rr_{A}(\supp(f))).$$  
   It then follows that $\supp(\widehat{f})$ is compact because it is homeomorphic to $\rr_{\hat{A}}\big(\supp(\widehat{f})\big)$. 
\end{proof}

Let $T_{\varphi}:\CC(\D)\rightarrow \CCC$ be the map $f\mapsto \widehat{f}$, where $\widehat{f}(z)=\sum_{x \in \varphi_1(z)}f(x)$. By Propositions \ref{prop: sumofbisections} and \ref{prop:coveringmorphism}, it is easy to see that $T_{\varphi}$ is a well-defined linear map.  Indeed, we have even more.

\begin{prop}
$T_{\varphi}$ is an algebra homomorphism. 
\end{prop}

\begin{proof}
Let $f$ and $g$ be in  $ \CC(\D)$ and  note that 
\begin{equation}\label{1}
\widehat{f\ast g}(z)=\sum_{x \in \varphi_1(z)}f\ast g (x)=\sum_{x \in \varphi_1(z)}\sum_{(c,d)\in M_x}f(c)g(d).\end{equation}
On the other hand 
\begin{equation}\label{2}
\widehat{f}\ast\widehat{g}(z)=\sum_{(a,b)\in M_z}\widehat{f}(a)\widehat{g}(b)=\sum_{(a,b)\in M_z}\sum_{\substack{\xi\in \varphi_1(a)\\ \eta\in \varphi_1(b) }}f(\xi)g(\eta).
\end{equation}
By (M3), we obtain that every term of the sum \eqref{2} is also a term of the sum \eqref{1}. Conversely, if a composition $cd$ belongs to $\varphi_1(z)$ then $\s(d)=\s(cd)=\varphi_0(\s(z))$, by (M2). In this case, there is $b \in \C_{\s(z)}$ such that $d \in \varphi_1(b)$, by (M5). Moreover, $\s(c)=\rr(d)=\varphi_0(\rr(b))$ and hence there exists $a \in \C_{\rr(b)}$ such that $c \in \varphi_1(a)$, again by (M5). Therefore, $cd \in \varphi_1(ab)$, by (M3). Moreover $\s(ab)=\s(b)=\s(z)$, and hence $ab=z$, by (M4). This gives that whenever a composition $cd$ belongs to $\varphi_1(z)$ there exists a unique $(a,b)\in M_z$ such that $c \in \varphi_1(a)$ and $d \in \varphi_1(b)$. We conclude that every term of the sum \eqref{1} also appears uniquely in the sum \eqref{2}, and the result follows. 
\end{proof}
 
The next goal is to show that $T_{\varphi}$ extends to a morphism from $\A(\D)$ to $\A(\C)$. To this end, note that $\widehat{\un(\DZ)}=\un(\CZ)$. Indeed, the inclusion $(\supseteq)$ follows easily from (M1). To prove the reverse inclusion, suppose $\un(v)\in\varphi_1(z)$, for $v \in \DZ$ and $z\in \CU$. By (M2), we have that $v=\varphi_0(\s(z))$, and hence we have that $\un(v)=\un\big(\varphi_0(\s(z))\big)\in\varphi_1\big(\un(\s(z))\big)$, by (M1). Therefore $\un(v)\in \varphi_1\big(\un(\s(z))\big)\cap\varphi_1(z)$ and hence $z=\un\big(\s(z)\big)$, by (M4). With this result we can now prove the following proposition.

\begin{prop}
Let $\pi:\CCC\rightarrow B(H)$ be a representation of $\CCC$. Then, $\pi\circ T_{\varphi}$ is a representation of $\CC(\D)$ and, consequently, $T_{\varphi}$ gives rise to a contractive homomorphism $\overline{T_{\varphi}}:\A(\D)\rightarrow \A(\C)$.
\end{prop}

\begin{proof}
We need to check conditions \eqref{i1definicaorepcate} and \eqref{i2definicaorepcate} of Definition \ref{definicaorepcate}. For condition \eqref{i1definicaorepcate}, let $A$ be a bisection and $f\in \cc(A)$ be a map. From Proposition \ref{prop:coveringmorphism} and its proof, we have that $\widehat{f}=T_{\varphi}(f)$ has support contained on the bisection $\achap$ and $\|\widehat{f}\|_\infty\leq\|f\|_\infty$ since $\widehat{f}(z)=f(x_z)$. Hence $\|\pi\circ T_{\varphi}(f)\|\leq \|f\|_\infty $. 

To prove condition \eqref{i2definicaorepcate}, note that if $f\in \cc\big(\un\big(\DZ\big)\big)$ then $T_{\varphi}(f)\in \cc(\CZ)$ and hence 
$$\begin{aligned}
\pi\circ T_{\varphi}(\overline{f~})
&=\pi\left(\overline{T_{\varphi}(f)}\right)
=\left(\pi\circ T_{\varphi}(f)\right)^*.
\end{aligned}$$
\end{proof}

If $\E$ is an étale category and $\psi=(\psi_0,\psi_1):\D\rightarrow \E$ is a relational covering morphism, then $\psi\circ\phi=((\psi\circ\varphi)_0,(\psi\circ\varphi)_1)$ defined as follows is a relational covering morphism: $(\psi\circ\varphi)_0:\CZ\rightarrow \EZ$ is given by   $(\psi\circ\varphi)_0(v)=\psi_0(\varphi_0(v))$ and
$$(\psi\circ\varphi)_1:\C1\rightarrow \PA(\EU), \qquad (\psi\circ\varphi)_1(z)=\bigcup_{x\in \varphi_1(z)}\psi_1(x), \ z\in \CU.$$
It is straightforward to prove that (M1)-(M5) hold. Here, we just show that (M6) holds. For $A\in \bis(\E)$ and $B\in \bis(\D)$, let $\achap$ denote the subset $\{z\in \CU:\ (\psi\circ\varphi)_1(z)\cap A\neq \emptyset \}$,  $\achap_{\psi}$ denote the bisection $\{z\in \DU:\ \psi_1(z)\cap A\neq \emptyset \}$ and $\widehat{B}_{\varphi}$ denote the bisection $\{z\in \CU:\ \varphi_1(z)\cap B\neq \emptyset \}$. We claim that
$$\achap=\widehat{(\achap_{\psi})}_{\varphi}.$$
If $z\in\achap$, then there exists $x\in \varphi_1(z)$ such that $\psi_1(x) \cap A$ is non-empty. Therefore, $x\in \achap_{\psi}$ and, consequently, $z \in \widehat{(\achap_{\psi})}_{\varphi}$. The converse is similar.

Finally, for $f\in \CC(\E)$ and $z\in \C$, note that
$$\begin{aligned} 
T_{\psi\circ\varphi}(f)(z)
&=\sum_{x\in (\psi\circ\varphi)_1(z)}f(x)
=\sum_{y\in \varphi_1(z)}\sum_{x\in \psi_1(y)}f(x)
=(T_{\varphi}\circ T_{\psi})(f)(z).
\end{aligned}$$
This gives $\overline{T_{\psi\circ\varphi}}=\overline{T_{\varphi}}\circ \overline{T_{\psi}}$ and proves that we have a functor from the category of étale categories to the category of Banach algebras.

\section{Restriction semigroups and their operator algebras}\label{sec:rsoa}

In this section, we briefly summarize the basics of restriction semigroups, referencing \cite[Section 2]{lawson} for more detailed information. Subsequently, we define their operator algebras and actions. We then conclude by defining the semicrossed product algebra associated with such actions and proving that the operator algebra of a restriction semigroup carries a natural semicrossed product structure.

\begin{defi}\label{restrictionsemigroupdefinition}	
    Let $S$ be a semigroup, $E(S)$ be its set of idempotents, $E\cont E(S)$ be a non-empty commutative subsemigroup of $S$, and $\lambda:S \rightarrow E$ and $\rho:S \rightarrow E$ be functions satisfying:
\begin{multicols}{2}
\begin{enumerate}
\item
    [(P1)] $\lambda(f)=f,$ for every $f\in E$.
\item
    [(P2)] $\rho(f)=f,$ for every $f\in E$.
\item   
    [(P3)] $s=s\lambda(s),$ for every $s\in S$.
\item
    [(P4)] $s=\rho(s)s,$ for every $s\in S$.
\item
    [(P5)] $\lambda(st)=\lambda(\lambda(s)t),$ for every $s,t\in S$.
\item
    [(P6)] $\rho(st)=\rho(s\rho(t))$, for every $s,t\in S$.
\item
    [(P7)] $fs=s\lambda(fs),$ for every $f\in E$ and $s\in S$.
\item
    [(P8)] $sf=\rho(sf)s,$ for every $f\in E$ and $s\in S$.
\end{enumerate}
\end{multicols}
  The quadruple $(S,E,\lambda,\rho)$ is termed a \textbf{restriction semigroup}, with the maps $\lambda$ and $\rho$ being called structure maps. The elements $E$ are called projections and it conveys to note that $E$ carries a natural order structure, where $e\leq f$ if $ef=e$.
\end{defi}

An element $s\in S$ could be heuristically regarded as function whose domain and codomain are $\lambda(s)$ and $\rho(s)$, respectively. In this case, $E$ would represent a set of subsets or, more precisely, a set of identity functions since it is contained in $S$.

\begin{rem}
    Under the conditions of Definition \ref{restrictionsemigroupdefinition}, a quadruple $(S,E,\lambda,\rho)$ satisfying (P1)-(P6) is called an \textbf{Ehresmann semigroup}. In the literature, one can find examples in which the map $\rho$ is not defined. In such case, a triple $(S,E,\lambda)$ satisfying (P1), (P3), (P5) and (P7) is called a \textbf{left restriction semigroup}.
\end{rem}

\begin{exe}\label{monoidasrs}
    Let $S$ be a monoid and $e$ be its unit. Defining $E=\{e\}$, note that there is just one possibility for $\lambda$ and $\rho$. Thus, it can be easily seen that $(S,E,\lambda,\rho)$ is a restriction semigroup.
\end{exe}

\begin{exe}[cf. Proposition 3.12 of \cite{lawson}]\label{ex: bis_rest_semi}
	If $\C$ is an étale category then $\left(\bisc, ~E,~\lambda,~\rho\right)$ is a restriction semigroup, where $E=\{U\in \bisc\ | U\cont \CZ \}$, $\lambda (U)=\{\s(x)\ | x\in U \}$ and $\rho (U)=\{\rr(x)\ | x\in U \}$.
	\end{exe}

\begin{exe}[cf. \cite{lawson1998inverse}, \cite{paterson}]
    Suppose $S$ is an inverse semigroup. Defining $\lambda(s)=s^*s$ and $\rho(s)=ss^*$ to be the usual \textit{source} and \textit{range} maps on $S$, we have that $(S,E(S),\lambda,\rho)$ is a restriction semigroup. This will be the canonical way of viewing an inverse semigroup as a restriction semigroup. 
\end{exe}

 The following proposition shows that the structure maps are intrinsic to the pair $(S,E)$.
 
\begin{prop}\label{lambdarhominimal}
	Let $(S,E,\lambda,\rho)$ be a restriction semigroup and $s$ be an element of $S$. Then, $\lambda (s)=\min \{f\in E\ |\ sf=s\}$ and $\rho(s)=\min \{f\in E \ |\ fs=s\},$ 
	where $E$ is equipped with its semilattice order.
\end{prop}

\begin{proof}
	    Note that the $\lambda(s) \in \{f \in E\ |\ sf=s\}$, by property (P3). Moreover, if $f \in E$ is such that $sf=s$ we have that $\lambda(s)f \in E$, and hence
	    $\lambda(s)f=\lambda(\lambda(s)f)=\lambda(sf)=\lambda(s).$ Therefore $\lambda(s)\leq f$, proving that $\lambda (s)=\min \{f\in E\ |\ sf=s\}$. The result for $\rho$ is similar.
\end{proof}

Every restriction semigroup $(S,E,\lambda,\rho)$ has a category $\C_S$ associated with. More precisely, the objects are $\CZ_S=E$, the morphisms are  $\CU_S=S$, and the structure maps are  $\s(s)=\lambda(s)$, $\rr(s)=\rho(s)$, $\un(e)=e$ and $\m(s,t)=st$.

Furthermore, similar to the inverse semigroup case, the relation $s\leq t$ if $s=tf$ for some $f\in E$ defines a partial order on a restriction semigroup. Note that this order relation agrees with the semilattice order on $E$ and, moreover, the following equivalences hold:
\begin{multicols}{2}
\begin{enumerate}[{\normalfont \rmfamily  (i)}]
\item $s\leq t$.
\item $\exists\  f\in E \text{ such that } s=ft.$
\item $s=\rho(s)t$.
\item $s=t\lambda(s)$.
\end{enumerate}
\end{multicols}

The properties presented below are demonstrated in Section 2 of \cite{lawson} and readily follow from the preceding definitions. Following this, we present certain propositions whose proofs become quite intuitive if one heuristically considers the product in $S$ as partial composition of functions.

\begin{multicols}{2}
\begin{enumerate}
\item [(R1)] $sf\leq s,\ \forall s\in S, f\in E$.
\item [(R2)]$fs\leq s,\ \forall s\in S, f\in E$.
\item [(R3)]$\lambda$ and $\rho$ are order-preserving.
\item [(R4)]If $s\leq t$ and $\lambda(s)=\lambda(t)$ then $s=t$.
\item [(R5)]If $s\leq t$ and $\rho(s)=\rho(t)$ then $s=t$.
\item [(R6)]$\lambda(st)\leq \lambda(t),\ \forall s,t\in S$.
\item [(R7)]$\rho(st)\leq \rho(s),\ \forall s,t\in S$.
\end{enumerate}
\end{multicols}

\begin{prop}\label{propdomaincomp}
	Let $(S,E,\lambda,\rho)$ be a restriction semigroup and $s,t$ be elements of $S$. Then, $\rho(t)\leq \lambda(s)$ if and only if $\lambda(st)=\lambda(t).$
\end{prop}

\begin{proof}
If $\rho(t)\leq \lambda(s)$ then $\lambda(s)t=\lambda(s)\rho(t)t=\rho(t)t=t$. Applying $\lambda$ to both sides, we get $\lambda(st)=\lambda(t)$. On the other hand, suppose $\lambda(st)=\lambda(t)$. By (R2), $\lambda(s)t\leq t$ and in addition \linebreak $\lambda(\lambda(s)t)=\lambda(st)=\lambda(t)$. Then, by (R4), $t=\lambda(s)t$. Thus, Proposition \ref{lambdarhominimal} ensures $\rho(t)\leq\lambda(s)$.
\end{proof}

\begin{prop}\label{lemmareg}
	Let $(S,E,\lambda,\rho)$ be a restriction semigroup. For elements $s,t,y$ in $S$ the following holds: $\rho(y)\leq \lambda(st)$ if and only if $\rho(y)\leq \lambda(t)$ and $\rho(ty)\leq \lambda(s)$.
\end{prop}

\begin{proof}
Suppose $\rho(y)\leq \lambda(st)$. By (R6), $\rho(y)\leq \lambda(t)$ and 
$$\begin{aligned}
        \lambda(s)\rho(ty)&\stackrel{(P6)}{=}~\lambda(s)\rho(t\rho(y))
        \stackrel{(P2)}{=}~\rho(\lambda(s)\rho(t\rho(y)))
        \\&\stackrel{(P6)}{=}~\rho(\lambda(s)t\rho(y))
        \stackrel{(P7)}{=}~\rho(t\lambda(\lambda(s)t)\rho(y))
        \\& \stackrel{(P5)}{=}~\rho(t\lambda(st)\rho(y))
        \stackrel{hip.}{=}~\rho(t\rho(y))
        \stackrel{(P6)}{=}~\rho(ty).
\end{aligned}$$

On the other hand, suppose that $\rho(y)\leq \lambda(t)$ and that $\rho(ty)\leq \lambda(s)$. Then, we have 

$$\begin{aligned}
    \lambda(st)\rho(y)&\stackrel{(P1)}{=}\lambda(\lambda(st)\rho(y))
    \stackrel{(P5)}{=}~\lambda(st\rho(y))
    \stackrel{(P5)}{=}~\lambda(\lambda(s)t\rho(y))
    \stackrel{(P8)}{=}~\lambda(\lambda(s)\rho(t\rho(y))t)
    \\&\stackrel{(P8)}{=}~\lambda(\lambda(s)\rho(ty)t)
    \stackrel{hip.}{=}~\lambda(\rho(ty)t)
    \stackrel{(P6)}{=}~\lambda(\rho(t\rho(y))t)
    \stackrel{(P8)}{=}~\lambda(t\rho(y))
    \\&\stackrel{(P5)}{=}~\lambda(\lambda(t)\rho(y))
    \stackrel{(P1)}{=}~\lambda(t)\rho(y)
    \stackrel{hip.}{=}~\rho(y).
\end{aligned}$$
\end{proof}

Now, we show an example of a left restriction semigroup that is not a restriction semigroup.

\begin{exe}[Partial surjections on a set]
	Let $X$ be a non-empty set. Define 
	    $$\jj(X)=\{f:A\rightarrow B\ |\ A,B\cont X \text{ and } f \text{ is surjective} \}.$$ 
	A straightforward calculation shows that $\jj(X)$ is a semigroup equipped with the following product: for $f:A\rightarrow B$ and $g:C\rightarrow D$ in $\jj(X)$ we define 
	        \begin{equation}\label{productjx}\left.\begin{array}{rrc}
			fg: & g^{-1}(A\cap D)\longrightarrow &f(A\cap D)\\
			& x \ \ \ \ \ \longmapsto &f(g(x)).
		    \end{array} \right.\end{equation}
	Note that for $f:A\rightarrow B$ in $\jj(X)$ and $C\cont X$, we have 	        \begin{equation}\label{fidc}	
		    \left.\begin{array}{rrc}
			f\id_C: & A\cap C\longrightarrow &f(A\cap C)\\
			& x \ \ \longmapsto &f(x).
	    	\end{array} \right.\end{equation} 
	And, on the other hand
	                \begin{equation}\label{idcf}	\left.\begin{array}{rrc}
		    	    \id_C f: & f^{-1}(B\cap C)\longrightarrow & B\cap C \\
		        	& x \ \ \ \ \ \longmapsto &f(x).
		            \end{array} \right.\end{equation}
	In particular, for subsets $B$ and $C$ of $X$ we have $\id_C\id_B=\id_{B\cap C}$. Therefore, the set $E=\{\id_A\ | \ A \cont X\}$ is a subsemigroup of $\jj(X)$ whose elements are idempotents, that is, $E$ is semilattice. Moreover, for $f:A\rightarrow B\in \jj(X)$, $C\cont A$ and $D\cont B$ we have 
	            \begin{equation}\label{fidc1}	
		        \left.\begin{array}{rrc}
		    	f\id_C: & C\longrightarrow &f(C)\\
			    & x \ \longmapsto &f(x),
		        \end{array} \right.\end{equation}
	and also
	            \begin{equation}\label{idcf2}	\left.\begin{array}{rrc}
			    \id_D f: & f^{-1}(D)\longrightarrow & D \\
			    & x \ \ \ \ \ \longmapsto &f(x).
		        \end{array} \right.\end{equation}
In view of Proposition \ref{lambdarhominimal}, note that $\lambda$ and $\rho$ must be defined as follows for $f:A\rightarrow B \in \jj(X)$, $\lambda(f)=\id_A$ and $\rho(f)=\id_B$. Next, we prove that $(\jj(X),~E,~\lambda)$ is a left restriction semigroup but not a restriction semigroup. 
	\begin{enumerate}
		\item  
		    [(P1)] $\lambda(\id_A)=\id_A,$ for every $A \cont X$: Trivial.
		\item
		    [(P2)] $\rho(\id_A)=\id_A,$ for every $A \cont X$: Trivial.
		\item
		    [(P3)] $f=f\lambda(f),$ for every $f\in \jj(X)$: It follows easily from equation \eqref{fidc1}.
		\item
		    [(P4)] $f=\rho(f)f,$ for every $f\in \jj(X)$: It follows easily from equation \eqref{idcf2}.
		\item
	    	[(P5)] $\lambda(fg)=\lambda(\lambda(f)g)$, for every $f,g\in \jj(X)$:
	    	
	    	Consider $f:A\rightarrow B$ and $g:C\rightarrow D$ in $\jj(X)$. By \eqref{productjx}, we have that $\lambda(fg)=\id_{g^{-1}(A\cap D)}$. On the other hand $\lambda(f)g=\id_A g$ has domain $g^{-1}(A\cap D)$, by \eqref{idcf}. Thus, $\lambda(\lambda(f)g)=\id_{g^{-1}(A\cap D)}=\lambda(fg).$
		
		\item
		    [(P6)] $\rho(fg)=\rho(f\rho(g))$, for every $f,g\in \jj(X)$:
		    
		    Consider $f:A\rightarrow B$ and $g:C\rightarrow D$ in $\jj(X)$. By \eqref{productjx} we have that $\rho(fg)=\id_{f(A\cap D)}.$ On the other hand $f\rho(g)=f\id_D$ has codomain $f(A\cap D)$, by \eqref{fidc}. Thus, $\rho(f\rho(g))=\id_{f(A\cap D)}=\rho(fg).$
		
		\item
		    [(P7)] $\id_Cf=f\lambda(\id_Cf),$ for every $C\cont X $ and $f\in \jj(X)$:
		
		    By \eqref{idcf}, we have 
		    $$\left.\begin{array}{rrc}
			\id_C f: & f^{-1}(B\cap C)\longrightarrow & B\cap C \\
			& x \ \ \ \ \ \longmapsto &f(x).
		    \end{array} \right.$$
			On the other hand, since $\lambda(\id_Cf)=\id_{f^{-1}(B\cap C)}$ and $f^{-1}(B\cap C)\cont A$, by \eqref{fidc1} we have $$\left.\begin{array}{rrc}
			f\lambda(\id_Cf): & f^{-1}(B\cap C)\longrightarrow & B\cap C \\
			& x \ \ \ \ \ \longmapsto &f(x).
		\end{array} \right.$$
	
	\item[(P8)]	$f\id_C=\rho(f\id_C)f,\text{ for every }C\cont X \text{ and }f\in \jj(X):$ 
	
	So far, we have proved that $(\jj(X),~E,~\lambda)$ is a left restriction semigroup. To obtain that the quadruple $(\jj(X),~E,~\lambda,~\rho)$ is a restriction semigroup, we should be able to verify that property (P8) holds. However, this is not always true. Assume $X$ has at least two distinct elements $a,b\in X$ and let $f:\{a,b\}\rightarrow \{a\}$ be the constant map $a$. By \eqref{fidc1}, $f\id_{\{a\}}=\id_{\{a\}}$. However,  by \eqref{idcf2}, $\rho(f\id_{\{a\}})f=\id_{\{a\}}f=f$. Thus $\jj(X)$ is not a restriction semigroup. Moreover, we have shown $f\id_{\{a\}}\neq \id_{\{a\}}f$ and, in particular, the product of idempotents of $\jj(X)$ does not commute.
	\end{enumerate}
\end{exe}

 Now, let us define morphisms between restriction semigroups.

\begin{defi}
	Let $S$ and $T$ be left restriction semigroups. A \textbf{left restriction semigroup homomorphism} from $S$ to $T$ is a semigroup homomorphism $\phi:S\rightarrow T$ such that $\phi(\lambda(s))=\lambda(\phi(s))$, for every $s \in S$.
\end{defi}

\begin{defi}
Let $S$ and $T$ be restriction semigroups. A \textbf{restriction semigroup homomorphism} from $S$ to $T$ is a semigroup homomorphism $\phi:S\rightarrow T$ such that $\phi(\lambda(s))=\lambda(\phi(s))$ and  $\phi(\rho(s))=\rho(\phi(s))$, for every $s \in S$.
\end{defi}

\begin{exe}[cf. Example 2.41 of \cite{lawson}]
Let $G$ be a group with the identity $e$ and suppose $|G| \geq 2$. Let $S$ be the group $G$ with the adjoined zero $0$. It is a restriction semigroup with respect to $E=\{e, 0\}$ with $\lambda$ and $\rho$ mapping $0 \mapsto 0$ and each nonzero element to $e$. Then, $\phi:S\rightarrow S$, which maps $x\mapsto e$, $x\in G$ and $0\mapsto 0$, is a restriction semigroup homomorphism.
\end{exe} 

	We say that a restriction semigroup $S$ is left-ample (resp. right-ample) if $st=su$ (resp. $ts=us$) implies that $\lambda(s)t=\lambda(s)u$ (resp. $t\rho(s)=u\rho(s)$)  for every triple $(s,t,u) \in S^3$. Moreover, $S$ is called \textbf{ample} if it is both left and right ample. The following theorem is an analogous version of the Wagner-Preston theorem for left restriction semigroups. The machinery developed in the proof will be used in a subsequent section. 

\begin{teo}[Theorem 14 of \cite{schweizer1967function}]\label{lrrep}
 Let $(S,E,\lambda)$ be a left restriction semigroup. For  $s\in S$, define $B_s=\{\ \lambda(s)t\ |\ t\in S\}$, $C_s=\{\ st\ | t \in B_s\}$ and consider the map $\phi_s:B_s\rightarrow C_s$, given by $\phi_s(t)=st$. Then, the map $$\left.\begin{array}{rrc}
 \phi: & S \longrightarrow & \jj(S)\\
 & s \longmapsto &\phi_s
 \end{array} \right.$$ is an injective left restriction semigroup homomorphism. Moreover, if $S$ is left-ample then each map $\phi_s$ is a bijection.
\end{teo}

\begin{proof}
Clearly $\phi_s$ is surjective. Moreover note that $B_s$ is equal to $\{t \in S\ |\ t=\lambda(s)t\}$ and $C_s$ is equal to $\{ st \ |\ t \in S \}$. In particular, if $S$ is left-ample and $\phi_s(t)=\phi_s(t')$, then $t=\lambda(s)t=\lambda(s)t'=t'$. Therefore, $\phi_s$ is an injective map for every $s\in S$. Furthermore, it is easy to see that $\phi_e=\id_{B_e}$ for each $e\in E$. Note that for $s$ and $u$ in $S$     
    $$\left.\begin{array}{rrc}
    \phi_s\phi_u: & \phi_u^{-1}(B_s\cap C_u)\longrightarrow &\phi_s(B_s\cap C_u)\\
    & t \ \ \ \ \ \longmapsto &sut.
    \end{array} \right.$$ 
To show that $\phi_s\phi_u=\phi_{su}$, it suffices to ensure that $\phi_u^{-1}(B_s\cap C_u)=B_{su}$ and $\phi_s(B_s\cap C_u)=C_{su}$ since both functions have the same output. 
\begin{enumerate}[{\normalfont \rmfamily  (i)}]
	\item 
	    $\phi_u^{-1}(B_s\cap C_u)=B_{su}:$ For $t\in \phi_u^{-1}(B_s\cap C_u)$, we have that $t=\lambda(u)t$ and $ut=\lambda(s)ut$. Hence,  
	        $$\begin{aligned}
	            \lambda(su)t
	            &\stackrel{(P7)}{=}~t\lambda(\lambda(su)t)
	            \stackrel{(P5)}{=}~t\lambda(sut)
	            \stackrel{(P5)}{=}~t\lambda(\lambda(s)ut)
	           \\& ~=t\lambda(ut)
	            \stackrel{(P5)}{=}~t\lambda(\lambda(u)t)
	            \stackrel{(P7)}{=}~\lambda(u)t
	            =t.
	\end{aligned}$$
	Hence, $t$ belongs to $B_{su}$. On the other hand, if $x\in B_{su}$ then 
	         $$\begin{aligned}
	               \lambda(u)x
	               &=\lambda(u)\lambda(su)x
	               =\lambda(su)\lambda(u)x
	               =\lambda(\lambda(su)\lambda(u))x
	               \\&~=\lambda(su\lambda(u))x
	               =\lambda(su)x
	               =x.
	        \end{aligned}$$
	Therefore, $x\in B_u$. Moreover, $\lambda(s)ux=u\lambda(\lambda(s)u)x=u\lambda(su)x=ux,$ and hence $\phi_u(x)\in B_s.$
	
	\item 
	    $\phi_s(B_s\cap C_u)=C_{su}:$ Since $C_{su}=\{sut\ |\ t\in S \}$, it is clear that $\phi_s(B_s\cap C_u)\cont C_{su}.$ On the other hand $sut=s(\lambda(s)ut)=\phi_s(\lambda(s)ut)$, for every $t\in S$. Finally, note that $\lambda(s)ut \in B_s$ and that $\lambda(s)ut=u\lambda(su)t \in C_u$. Thus, $sut\in \phi_s(B_s\cap C_u)$ for every $t\in S$.
\end{enumerate} 
In conclusion, $\phi$ is a semigroup homomorphism and moreover a restriction semigroup homomorphism since
            $$  \lambda(\phi_s)=\id_{B_s}=
            \id_{B_{\lambda(s)}}=
            \phi_{\lambda (s)}.$$
Now, suppose $\phi_s=\phi_u$. In this case, since $B_s=B_u$, we have that $\lambda(s)\in B_u$ and $\lambda(u)\in B_s$. Thus, $\lambda(s)=\lambda(s)\lambda(u)=\lambda(u)$ and, moreover, $s=\phi_s(\lambda(s))=\phi_s(\lambda_u)=\phi_u(\lambda_u)=u,$ which proves the injectivity of $\phi$.
\end{proof}

Let $(S,E,\lambda,\rho)$ be a restriction semigroup. Our goal is to express the domain $B_s$ and the codomain $C_s$ of the maps $\phi_s$ in Theorem \ref{lrrep} in terms of the structure maps of $S$. For instance, note that 
$$B_s=\{t\in S\ |\ \rho(t)\leq \lambda(s)\}.$$ 
In fact, if $t=\lambda(s)t$ then $\rho(t)=\lambda(s)\rho(t)\leq \lambda(s)$ and, on the other hand, if $\rho(t)\leq \lambda(s)$ then $\lambda(s)t=\lambda(s)\rho(t)t=\rho(t)t=t$. Furthermore, this characterization of $B_s$ gives $C_s=\{st\ \mid\ \rho(t)\leq \lambda(s)\}$. And, in this case note that $C_s$ is contained in $\{t\in S\ |\ \rho(t)\leq \rho(s)\}.$ In fact, if $\rho(t)\leq \lambda(s)$ then $\rho(st)=\rho(s\rho(t))\leq \rho(s\lambda(s))=\rho(s)$. The reader familiar with the Wagner-Preston theorem may wonder under what conditions $C_s$ coincides with the set $\{t\in S\ |\ \rho(t)\leq \rho(s)\}$.

\begin{prop}\label{bscsvagner}
Let $S$ be a restriction semigroup. Suppose the sets $C_s=\{st\in S\ |\ \rho(t)\leq \lambda(s)\}$ and $\{y\in S\ |\ \rho(y)\leq \rho(s)\}$ coincide for every $s\in S$. Then, S is a regular semigroup. Conversely, if $S$ is regular and, in addition, if for every $s \in S$ there is an inverse $s^*$ for $s$ such that $ss^*\in E$ then $C_s=\{st\in S\ |\ \rho(t)\leq \lambda(s)\}$ coincides with the set $\{y\in S\ |\ \rho(y)\leq \rho(s)\},$ for every $s\in S.$ \end{prop}

\begin{proof}
Let $s\in S$ be such that $C_s=\{y\in S\ |\ \rho(y)\leq \rho(s)\}$. Choosing $y=\rho(s)$, we have that
    \begin{equation}\label{auxregular} 
    \rho(s)=y=ss^*\text{, for }s^*\in S\text{ such that }\rho(s^*)\leq \lambda(s).
    \end{equation}
Hence, applying $\rho$ to \eqref{auxregular}, we obtain that $\rho(s)=\rho(ss^*)=\rho(s\rho(s^*))$. Then, by (R1), $s\rho(s^*)\leq s$ and, by (R5), $s=s\rho(s^*)$. Thus, by Proposition \ref{lambdarhominimal}, $\lambda(s)=\rho(s^*)$. Applying $\lambda$ to \eqref{auxregular}, we get $\rho(s)=\lambda(\lambda(s)s^*)=\lambda(\rho(s^*)s^*)=\lambda(s^*)$. Therefore, $ss^*s=\rho(s)s=s$ and $s^*ss^*=s^*\rho(s)=s^*\lambda(s^*)=s^*.$

Conversely, consider $s\in S$ and an inverse $s^*$ for $s$ such that $ss^*\in E$. Note that if $t\in S$ is such that $\rho(t)\leq \rho(s)$ then $\rho(t)\leq ss^*$ and then $t=\rho(t)t=ss^*\rho(t)t=ss^*t$.
\end{proof}

\subsection{Operator algebras associated with restriction semigroups}

\begin{defi}\label{repsmgrest}
Let $(S,E,\lambda,\rho)$ be a restriction semigroup. A \textbf{representation} of $S$ on a Hilbert space $H$ is a semigroup homomorphism $\sigma:S\rightarrow B(H)$, $s\mapsto\sigma_s$, such that 
\begin{enumerate}[{\normalfont \rmfamily  (i)}]
    \item $\|\sigma_s\|\leq 1$ for every $s\in S$.
    \item $\sigma_e^*=\sigma_e^2=\sigma_e$ for every $e\in E$.
\end{enumerate} 
The class of all representations of $(S,E,\lambda,\rho)$ will be denoted by $\mathcal{R}(S)$.
\end{defi}

Every representation $\sigma:S\rightarrow B(H)$ extends  to a representation of the complex semigroup algebra $\widetilde{\sigma}:\comp[S]\rightarrow$ B(H), where $\widetilde{\sigma}\left(\sum_{s\in S}a_s\delta_s\right)=\sum_{s\in S}a_s\sigma_s$, which is contractive if $\comp[S]$ is equipped with the $\ell_1$ norm. Let $\mathcal{S}(\comp[S])$ denote the set of all seminorms on $\comp[S]$ induced by $\mathcal{R}(S)$. We can then define a seminorm $\|\cdot\|_0$ on $\comp[S]$ either in terms of seminorms or representations
    $$\left\|\mathbf{x}\right\|_0:=\sup_{\|\cdot\| \in \mathcal{S}(S)}\|\mathbf{x}\|=\sup_{\sigma \in \mathcal{R}(S)}\|\widetilde{\sigma}(\mathbf{x})\| .$$ 
Similar to the discussion after Remark \ref{rem: Pi_to_pi}, there is a representation $\Sigma:S\rightarrow B(H)$ of $S$ such that $\|\mathbf{x}\|_0=\|\widetilde{\Sigma}(\mathbf{x})\|$. 

\begin{defi}
Let $(S,E,\lambda,\rho)$ be a restriction semigroup. The \emph{operator algebra} of $S$ is the Hausdorff completion $\A(S)$ of $(\comp[S], \|\cdot\|_0)$. 
\end{defi}

Note that $\widetilde{\Sigma}$ induces an isometric homomorphism of $\A(S)$ on $B(H)$ justifying the term \textit{operator algebra}.

\subsubsection{The reduced operator algebra of a restriction semigroup}

Let $(S,E,\lambda,\rho)$ be a left-ample restriction semigroup. From Theorem \ref{lrrep} and its subsequent discussion, for each $s \in S$ there exists a bijective map $\phi_s$ from $B_s=\{t\in S\ |\ \rho(t)\leq \lambda(s) \}$ to $C_s=\{st\ |\ t\in B_s\}$, namely $\phi_s(t)=st$. Let $\ell_2(S)$ be the Hilbert space associated with $S$ (see Definition \ref{def:hilbertspaceofset}) and note that each map $\phi_s$ induces an operator \begin{equation}\label{formularegulardes}
\phi'_s:\ell_2(S)\rightarrow \ell_2(S)\text{ by }\phi'_{s}(\delta_t)=[\rho(t)\leq \lambda(s)]\ \delta_{st},
\end{equation}
which is a partial isometry such that $\ker(\phi'_s)^{\perp}= \spf\{\ \delta_t \mid t\in B_s\}$ and $\ran(\phi'_s)_s= \spf\{\ \delta_t \mid t\in C_s\}$. Thus,  the map $\phi':S\rightarrow B(\ell_2(S))$, $s \mapsto \phi'_s$, is a representation of $S$ called \textbf{the regular representation of $S$}. In fact, to see that $\phi' \in \mathcal{R}(S)$ it suffices to note that Proposition \ref{lemmareg} ensures $$\phi'_{st}(\delta_{y})=\phi'_s\phi'_t(\delta_y), \ \forall y \in S.$$

\begin{teo}\label{wordin}
The extension $\widetilde{\phi'}:\comp[S]\rightarrow B(\ell_2(S))$ is faithful.
\end{teo}

\begin{proof}
Replacing $s^*s$ by $\lambda(s)$, $ss^*$ by $\rho(s)$ and $E(S)$ by $E$, the proof of Theorem \ref{wordin} is the verbatim copy of Wordingham's theorem proof (see \cite[Theorem 2.1.1]{paterson}).
\end{proof}

\begin{defi}
	Let $(S,E,\lambda,\rho)$ be a left-ample restriction semigroup. The \textbf{reduced operator algebra} of $S$ is the closure $\A_{r}(S)$ of $\widetilde{\phi'(\comp[S])}$ in $B(\ell_2(S))$ .
\end{defi}

\subsubsection{The inverse semigroup case}\label{subsection:inversesemigroupcase}

We start by revisiting the concept of generalized inverse of an operator on a Hilbert space. We direct the reader's attention to pages $24$ to $26$ of \cite{paterson} for further information on $C^*$-algebras associated with inverse semigroups.

\begin{defi}\label{def:generalizedinverse}
Let $H$ be a Hilbert space and $T\in B(H)$ be a bounded operator. An operator $S\in B(H)$  will be called a \textbf{generalized inverse} of $T$ if $TST=T$ and $STS=S.$ 
\end{defi}
%
%
 \begin{cor}[Corollary 3.3 of \cite{mbekhta2004partial}]\label{mbekhtacorolario}
Let $T\in B(H)$ be a contraction. If a contractive generalized inverse $S$ of $T$ exists, then $S=T^*$.
\end{cor}


Suppose $S$ is an inverse semigroup and $\sigma$ is a representation of $S$ in the sense of Definition \ref{repsmgrest}. The maps $\sigma_{s^*}$ are generalized inverses of the maps $\sigma_s$. By virtue of Corollary \ref{mbekhtacorolario}, $\sigma_{s^*}=\sigma_{s}^*$ and therefore $\sigma$ is a $\ast$-representation of $S$ and  $$\A(S)=C^*(S).$$

In addition, the regular representation $\phi'$ presented in \eqref{formularegulardes} is precisely the \text{left regular representation} from the Wagner-Preston Theorem \cite[Proposition 2.1.3]{paterson}, by Proposition \ref{bscsvagner}. Then, $$\A_{r}(S)=C_{red}^*(S).$$

\subsection{Restriction semigroup étale actions}

We provide the precise definition of an étale action of a restriction semigroup. We emphasize that there are different notions of restriction semigroups and their actions, and that is the reason why we have decided to call our actions \textit{étale}.

Let $(S,E,\lambda,\rho)$ be a restriction semigroup, $X$ be a set and $\mathcal{I}(X)$ be the inverse semigroup of partial bijections of $X$. Suppose $\theta:S\rightarrow\mathcal{I}(X)$ is a restriction semigroup homomorphism. Note that for every $s\in S$, there exist $D_s\cont X$ and $R_s\cont X$ such that $\theta_s:D_s\rightarrow R_s$ is a bijection. If $e\in E(S)$, then $\theta_e \in E(\mathcal{I}(X))$, which gives that $D_e=R_e$ and $\theta_e=\id_{D_e}$. Thus, for every $s\in S$ 	
				$$\id_{D_{\lambda(s)}}
				=\theta_{\lambda(s)}
				=\lambda(\theta_s)
				=\theta_s^{-1}\theta_s
				=\id_{D_s},$$ 
			and
			 	$$\id_{D_{\rho(s)}}
			 	=\theta_{\rho(s)}
			 	=\rho(\theta_s)
			 	=\theta_s\theta_s^{-1}
			 	=\id_{R_s}.$$ 
Therefore, $D_s=D_{\lambda(s)}$ and $R_s=D_{\rho(s)}$.

\begin{defi}
Let $(S,E,\lambda,\rho)$ be a restriction semigroup and $X$ be a locally compact Hausdorff space. An \textbf{étale action} of $S$ on $X$ is a restriction semigroup homomorphism $\theta:S\rightarrow\mathcal{I}(X)$ satisfying:
\begin{enumerate}[{\normalfont \rmfamily  (i)}]
\item For every $e\in E$, $D_e$ is open and $X=\bigcup_{e\in E}D_e$.
\item For every $s\in S$, $\theta_s:D_{\lambda(s)}\rightarrow D_{\rho(s)}$ is a homeomorphism.
\end{enumerate}
\end{defi}

\begin{defi}\label{etaleactioncstaralgebras}
	Let $(S,E,\lambda,\rho)$ be a restriction semigroup and $A$ be a $C^*$-algebra. An \textbf{étale action} of $S$ on $A$ is a restriction semigroup homomorphism $\alpha:S\rightarrow\mathcal{I}(A)$, $\alpha_s:J_{\lambda(s)}\rightarrow J_{\rho(s)}$, satisfying:
	\begin{enumerate}[{\normalfont \rmfamily  (i)}]
		\item For every $e\in E$, $J_e$ is a closed ideal and $A=\spf \bigcup_{e\in E}J_e$.
		\item For every $s\in S$, $\alpha_s:J_{\lambda(s)}\rightarrow J_{\rho(s)}$ is a $*$-isomorphism.
	\end{enumerate}
\end{defi}

\begin{exe}\label{actionbisc}
If $\C$ is an étale category, then $\bisc$ is a restriction semigroup, by Example \ref{ex: bis_rest_semi}. Moreover a bisection $U$ defines a homeomorphism from $\s(U)$ to $\rr(U)$, namely $\theta_U=\rr_U \circ \s_U^{-1}$. Easy calculations show that $\theta:\bisc\rightarrow \mathcal{I}(\CZ)$, $U\mapsto \theta_U$, defines an étale action of $\bisc$ on $\CZ$.
\end{exe}

\begin{rem}\label{topinducesalg} For an étale action $\theta$ of $S$ on $X$ it is straightforward to check that setting $J_e:=\czero(D_e)$ and $\alpha_s:J_{\lambda(s)}\rightarrow J_{\rho(s)}$ by $\alpha_s(f)=f\circ \theta_s^{-1}$ gives an étale action of $S$ on $\czero(X)$. 
\end{rem}

\subsubsection{The canonical action of a restriction semigroup}\label{standardaction}
%
%
For a semilattice $E$, the set $\echap=\{\phi:E\rightarrow\{0,1\}\ |\ \phi(ef)=\phi(e)\phi(f) \text{ and } \phi\neq 0 \}$ equipped with the subspace topology of $\{0,1\}^E$ is called  the \textit{spectrum} of $E$ and its elements are called semicharacters. Its topology agrees with the \textit{pointwise convergence} topology. By Tychonoff's theorem, $\{0,1\}^{E}$ is a compact Hausdorff space and hence $\echap\cup \{0\}$ is a compact Hausdorff subspace, because it is closed. Thus, $\echap$ itself is a locally compact Hausdorff space.

Let $D_e$ denote the subset $\{\phi \in \echap\ |\ \phi(e)=1\}$. Each  $D_e$ is simultaneously an open and closed subset of $\echap$ since $D_e=P_e^{-1}(\{1\})\cap \echap$, where $P_e:\{0,1\}^E\rightarrow \{0,1\}$ denotes the projection onto coordinate $e$. Moreover, via convergence of nets, one easily sees that $D_e$ is closed in $\{0,1\}^E$ which implies that $D_e$ is compact. Therefore, the family $\{D_e\ |\ e\in E \}$ forms a cover of compact open sets for $\echap$.

 Denoting by $1_e$ the characteristic function of $D_e$, we have that $1_{ef}=1_e1_f$ since $D_e\cap D_f=D_{ef}$. In particular, $\sps \{1_e\ |\ e\in E\}$ is a subalgebra of $\czero(\echap)$. Moreover, note that \begin{equation}\label{avaliacao}
    1_e(\phi)=\phi(e), \forall \ \phi \ \in \echap.
    \end{equation} 
    
From \eqref{avaliacao}, we see that the family of functions $\{1_e\ |\ e\in E\}$ separates the points of $\echap$ and that for every semicharacter $\phi \in \echap$ there exists $e$ such that $1_e(\phi)\neq 0$. By the Stone-Weierstrass theorem, $\sps \{1_e\ |\ e\in E\}$ is a dense subalgebra of $\czero(\echap).$ Moreover, $\sps \{1_e\ |\ e\in E\text{ and }e\leq f\}$ is a dense subalgebra of $\operatorname{C}(D_f)$, for every $f\in E$.

\begin{teo}[Proposition 10.6 from
    \cite{exel}]\label{universal}
    Let $\sigma:E \rightarrow B(H)$ be a semigroup homomorphism whose range consists of (orthogonal) projections. Then, there is a unique $C^*$-algebra homomorphism  $\pi_{\sigma}:\czero(\echap)\rightarrow B(H)$ sending $\pi_{\sigma}(1_e)=\sigma_e.$
\end{teo}

\begin{proof}
Since $\sigma_e\sigma_f=\sigma_{ef}$ and $\sigma_e^*=\sigma_e$, the subspace $\sps\ \{\sigma_e\ |\ e\in E\}$ is already a self-adjoint subalgebra of $B(H)$ and hence the $C^*$-subalgebra of $B(H)$ generated by $\{\sigma_e\ |\ e\in E\}$ is simply $A:= \spf\ \{\sigma_e\ |\ e\in E\}$, which is commutative.  

Recall that the spectrum of $A$ is the locally compact Hausdorff space $\achap$ of all nonzero continuous $C^*$-algebra homomorphisms (characters) $\psi:A\rightarrow \mathbb{C}$. Moreover, define 
$$\iota_1:\achap\cup \{0\}\rightarrow \echap\cup \{0\}, \ \ \ \iota_1(\psi)(e)=\psi(\sigma_e).$$
Using nets, it becomes evident that $\iota_1$ is continuous. Additionally, it is injective as $\iota_1(\psi)=\iota_1(\psi')$ implies that $\psi$ and $\psi'$ coincide on the generating set $\{\sigma_e\ |\ e\in E\}$. Thus, $\iota_1$ is proper and restricts to a proper injective map $\iota:\achap\rightarrow \echap$ since $0$ is mapped to $0$. In this case, $\pi(f)=f\circ\iota$ defines a $C^*$-algebra homomorphism $\pi:\czero(\echap)\rightarrow \czero(\achap)$.  

 For $a \in A$, recall that $\widehat{a}\in \czero(\achap)$ denotes the function that evaluates a character $\tau$ on $a$, that is, $\widehat{a}(\tau)=\tau(a)$. Moreover, recall that if $T:\czero(\achap)\rightarrow A\cont B(H)$ denotes the inverse of the Gelfand transform we have $T(\widehat{a})=a$. Finally, for $e\in E$ and
 $\psi\in \achap$ we obtain
    $$\begin{aligned}
    \pi(1_e)(\psi)
    =1_e(\iota(\psi))
    \stackrel{\eqref{avaliacao}}{=}\iota(\psi)(e)
    =\psi(\sigma_e)
    =\widehat{\sigma_e}(\psi).
    \end{aligned}$$
Thus, defining $\pi_{\sigma}:=T\circ \pi$ we have the desired $C^*$-algebra homomorphism from $\czero(\echap)$ to $B(H)$ satisfying $\pi_{\sigma}(1_e)=\sigma_e.$ The uniqueness of $\pi_\sigma$ is established by observing that $\sps\{1_e\ |\ e\in E\}$ is a dense subalgebra of $\czero(\echap)$ and, consequently, any two continuous linear maps that coincide on $\{1_e\ |\ e\in E\}$ are necessarily identical.
\end{proof}
 It is well known that an inverse semigroup $T$ acts on $\widehat{E(T)}$ (see \cite[Section 10]{exel}). We proceed to show that a similar construction holds for a restriction semigroup $(S,E,\lambda,\rho)$.
 Recall that $\echap$ is covered by the family of compact open sets $D_e=\{\phi \in \echap\ |\ \phi(e)=1\}$ and define $\theta_s: D_{\lambda(s)}\rightarrow D_{\rho(s)}$  by $\theta_s(\phi)(f)=\phi(\lambda(fs))$. 
Note that, for $\phi\in D_{\lambda(s)}$, the following holds $$\theta_s(\phi)(\rho(s))=\phi(\lambda(\rho(s)s))=\phi(\lambda(s))=1.$$ Hence, $\theta_s(\phi)\in \drs$ and it is a homomorphism since
		$$\begin{aligned}
			\theta_s(\phi)(ef)
			&=\phi(\lambda(efs))
			=\phi(\lambda(es\lambda(fs)))
			\\&=\phi(\lambda(es)\lambda(fs)))
			=\phi(\lambda(es))\phi(\lambda(fs)))
			\\&=\theta_s(\phi)(e)\theta_s(\phi)(f).
		\end{aligned}$$
If $\{\phi_j\}_{j\in J}$ is a net in $D_{\lambda(s)}$ converging to $\phi$ then $\phi_j(\lambda(fs))$ converges to $\phi(\lambda(fs))$ for every $f\in E$. This convergence in turn implies that $\theta_s(\phi_j)$ converges to $\theta_s(\phi)$, demonstrating the continuity of $\theta_s$. 

Similarly, the map $\zeta_s$ below is well-defined and continuous.  
\begin{equation}\label{definitionzeta}
    \zeta_s: D_{\rho(s)}\rightarrow D_{\lambda(s)},\ \ \ \zeta_s(\phi)(f)=\phi(\rho(sf))
\end{equation}  

Let us prove that $\theta_s$ and $\zeta_s$ are inverses, confirming that they are homeomorphisms. We just provide the proof that $\zeta_s\circ\theta_s=\id_{\dls}$ and left the converse for the reader. For $\phi \in \dls$ we have	
$$\begin{aligned}
		(\zeta_s\circ \theta_s)(\phi)(e)
		&=\theta_s(\phi)(\rho(se))
		=\phi(\lambda(\rho(se)s))
		\\&=\phi(\lambda(se))
		=\phi(\lambda(\lambda(s)e))
		\\&=\phi(\lambda(s)e)
		=\phi(\lambda(s))\phi(e)
		\\&=\phi(e).
	\end{aligned}$$
 
 Next, we prove that $\theta$ is an étale action of $S$ on $\echap$.
 \begin{enumerate}[{\normalfont \rmfamily  (i)}]
 \item 
 $D_{\lambda(st)}=\theta^{-1}_t(\drt \cap \dls):$ 
 If $\phi \in \drt \cap \dls$, then $\theta_t^{-1}(\phi)=\zeta_t(\phi)$ and since  $t\lambda(st)=t\lambda(\lambda(s)t)=\lambda(s)t$ we have
 	 	$$\begin{aligned}
 			\zeta_t(\phi)(\lambda(st))
 			&=\phi(\rho(t\lambda(st)))
 			=\phi(\rho(\lambda(s)t))
 			\\&=\phi(\rho(\lambda(s)\rho(t)))
			=\phi(\lambda(s)\rho(t))
 			\\&=\phi(\lambda(s))\phi(\rho(t))=1.
 		\end{aligned} $$
Therefore $D_{\lambda(st)}\supseteq\theta^{-1}_t(\drt \cap \dls)$. To prove $D_{\lambda(st)}\subseteq\theta^{-1}_t(\drt \cap \dls)$, note that if $\phi \in D_{\lambda(st)}$ then $$\begin{aligned}
\theta_t(\phi)(\lambda(s))&=\phi(\lambda(\lambda(s)t))=\phi(\lambda(st))=1.
\end{aligned} $$

\item 
$\theta_s\theta_t=\theta_{st}:$ 
	$$\begin{aligned}\theta_s\theta_t(\phi)(e)
		&=\theta_t(\phi)(\lambda(es))
		=\phi(\lambda(\lambda(es)t))
		\\&=\phi(\lambda(est))
		=\theta_{st}(\phi)(e).
\end{aligned}$$
Hence, 
$\theta_{st}$ is equal to $\theta_{s}\theta_{t}$ since they are two bijections with the same domain and rule.

\item 
$\theta_{\lambda(s)}=\id_{D_{\lambda(s)}}$: For $\phi \in D_{\lambda(s)}$, we have 
	$$\begin{aligned}
		\theta_{\lambda(s)}(\phi)(e)
		&=\phi(\lambda(e\lambda(s)))
		=\phi(e\lambda(s))
		\\&=\phi(e)\phi(\lambda(s))
		=\phi(e).
\end{aligned}$$

\item 
$\theta_{\rho(s)}=\id_{D_{\rho(s)}}$: For $\phi \in D_{\rho(s)}$, we have
	 $$\begin{aligned}
		\theta_{\rho(s)}(\phi)(e)
		&=\phi(\lambda(e\rho(s)))
		=\phi(e\rho(s))
		\\&=\phi(e)\phi(\rho(s))
		=\phi(e).
	\end{aligned}$$
 \end{enumerate}

\subsubsection{The invariance of the tight spectrum}

 Suppose that $(S,E,\lambda,\rho)$ is a restriction semigroup with $0$ and assume that $0$ is in $E$. In what follows, we use the notation present in \cite[Section 11]{exel}.

Recall that $\echap_0=\{\phi \in \echap\ |\ \phi(0)=0 \}$ is a locally compact Hausdorff space and its elements are called \textit{characters}. Moreover, for $x,y \in E$ we say that $x$ \textit{intersects} $y$ if, and only if, $xy\neq 0$ and, in such case, we write $x\Cap y$. Otherwise, we say that $x$ and $y$ are \textit{orthogonal} and we write $x\perp y$. Let us see that $\echap_0$ is invariant under the canonical action, that is, if $\phi\in \echap_0\cap \dls$ then $\theta_s(\phi)\in \echap_0$. Note that
    $$\begin{aligned}
    \theta_s(\phi)(0)
    &=\phi(\lambda(0s))
    =\phi(\lambda(0))
    =\phi(0)
    =0.
    \end{aligned}$$

\begin{defi}
For $x\in E$ the \textbf{upper set} of $x$ is the set $x^{\uparrow}=\{f\in E\ | x\leq f \}$ and the \textbf{down set} of $x$ is the set $x^{\downarrow}=\{f\in E\ | f\leq x\}$.
\end{defi}

\begin{defi}
Let $F$ be a subset of $E$. A subset $Z\cont F$ is a \textbf{cover} for $F$ if for every nonzero $x\in F$, there exists $z\in Z$ such that $z\Cap x$. For $y\in E$, we say that $Z$ is a cover for $y$ if $Z$ is a cover for $y^{\downarrow}$.
\end{defi}

In view of \cite[Prop 11.8]{exel}, we define a tight character as follows.

\begin{defi}
Let $\phi$ be in $\echap_0$. We call $\phi$ a \textbf{tight character} if for every $x\in E$ and for every finite cover $Z$ for $x$ one has that $$\bigvee_{z\in Z}\phi(z)\geq \phi(x).$$ The \textbf{tight spectrum} of $E$ is the set of all tight characters and it is denoted by $\echap_{tight}$.
\end{defi}

\begin{lemma}
If $s\in S$, $x\in E$ and $Z$ is a cover for $x$, then $Z_s=\{\lambda(zs)\ |\ z \in Z\}$ is a cover for $\lambda(xs)$.
\end{lemma}

\begin{proof}
Since $\lambda$ is order-preserving, $Z_s\cont \lambda(xs)^{\downarrow}.$ Now, let $y\in \lambda(xs)^{\downarrow} $ be a nonzero element. Note that 
    $$\begin{aligned}
    x\rho(sy)
    &=\rho(xsy)
    =\rho(s\lambda(xs)y)
    =\rho(sy).
    \end{aligned}$$
Hence, $\rho(sy)\leq x$ and 
    $$\begin{aligned}
    y
    &=\lambda(xs)y
    =\lambda(xsy)
    =\lambda(x\rho(sy)s)
    =\lambda(\rho(sy)s).
    \end{aligned}$$

Thus, since $y$ is nonzero and $y=\lambda(\rho(sy)s)$, we obtain that $\rho(sy)$ is nonzero. Therefore, there exists $z\in Z$ such that $z\rho(sy)\neq 0$. Next,  note that $$\begin{aligned}z\rho(sy)=\rho(zsy)=\rho(s\lambda(zs)y)\end{aligned}.$$ In particular, $\lambda(zs)y\neq 0$ and $Z_s$ covers $\lambda(xs)$.
\end{proof}

\begin{prop}
If $\phi\in \echap_0\cap \dls$ is a tight character, then $\theta_s(\phi)$ is a tight character.
\end{prop}

\begin{proof}
Let $x\in E$ and $Z$ be a finite cover for $x$. By the previous Lemma, $Z_s=\{\lambda(zs)\ |\ z \in Z\}$ is a finite cover for  $\lambda(xs)$. Therefore $$\bigvee_{z\in Z}\theta_s(\phi)(z)=\bigvee_{z\in Z}\phi(\lambda(zs))\geq \phi(\lambda(xs))=\theta_s(\phi)(x).$$
\end{proof}

\subsection{The semicrossed product algebra}
Let $\alpha$ be an étale action of a restriction semigroup $S$ on a $C^*$-algebra $A$, according to Definition \ref{etaleactioncstaralgebras}. Following \cite[Proposition 4.1]{sieben1997c} and replacing $\alpha_{s^*}$ by $\alpha_{s}^{-1}$, we obtain that the vector space $L_{\alpha}$ of all finite formal sums $\sum_{s\in S}a_s\delta_s$, where $a_s \in J_{\rho(s)}$, is an associative algebra endowed with the product defined in Equation \eqref{productinLalpha}. 
\begin{equation}\label{productinLalpha}
    \left(\sum_{s \in S}a_s\delta_s\right)\left(\sum_{t \in S}b_t\delta_t\right):= \sum_{s \in S}\sum_{t \in S} \alpha_s(\alpha_s^{-1}(a_s)b_t)\delta_{st}.
\end{equation}
 \begin{defi}\label{defipar}
 A \textbf{covariant pair} for an étale action $(\alpha,S,A)$, or simply $\alpha$, on a Hilbert space $H$ is a pair $(\pi,\sigma)$ where $\sigma$ is a representation of $S$, $\pi$ is a representation of $A$ and
 \begin{enumerate}[{\normalfont \rmfamily  (i)}]
 \item \label{covariancerelation}$\pi(\alpha_s(a))\sigma_s=\sigma_s\pi(a)$, for every $a\in \jls,\ s \in S.$
 \item\label{defipari2} $\spf\ \pi(J_e)H=\sigma_e(H)$,  for every $e \in E.$ 
\end{enumerate}
The equality presented at item \eqref{covariancerelation} is usually called covariance relation.
 \end{defi}

Every covariant pair $(\pi,\sigma)$ on $H$ for $\alpha$ integrates to an algebra homomorphism  $\pi\times \sigma: L_{\alpha}\rightarrow B(H)$ given by $\pi\times\sigma(a\delta_s)=\pi(a)\sigma_s.$ Moreover, $\pi\times \sigma$ is contractive if we endow $L_{\alpha}$ with the $\ell_1$ norm. We will only show that $\pi\times \sigma$ separates the product:
    $$\begin{aligned}
    \pi\times \sigma(a\delta_s b\delta_t)
    &=\pi(\alpha_s(\alpha_s^{-1}(a)b))\sigma_{st}
    =\pi(\alpha_s(\alpha_s^{-1}(a)b))\sigma_{s}\sigma_{t}
    \\&=\sigma_s\pi(\alpha_s^{-1}(a)b)\sigma_{t}
    =\sigma_s\pi(\alpha_s^{-1}(a))\pi(b)\sigma_{t}
    \\&=\pi(\alpha_s(\alpha_s^{-1}(a)))\sigma_s\pi(b)\sigma_{t}
    =\pi(a)\sigma_s\pi(b)\sigma_{t}
    \\&=\pi\times\sigma(a\delta_s)\pi\times\sigma(b\delta_t).
    \end{aligned}$$

\begin{defi}
The \textbf{semicrossed product algebra} of $A$ by $S$ with respect to $\alpha$, denoted $A\rtimes_\alpha S$, is the \textit{Hausdorff completion} of $L_{\alpha}$ equipped with the seminorm induced by the class of all covariant pairs, which is $$\left\|\sum_{s\in S}a_s\delta_s\right\|_0=  \sup_{(\pi,\sigma) \text{ cov. pair}}  \left\|\pi\times \sigma(\sum_{s\in S}a_s\delta_s)\right\|.$$
\end{defi}

Would it be possible to find a ``faithful" covariant pair in such a way the quotient is not necessary? In general, the answer is no: note that if $a\in J_e$, for every covariant pair $(\pi,\sigma)$ one has $\pi(a)\sigma_e= \pi(a)$. Indeed $\pi(a)\sigma_e= \pi(\alpha_e(a))\sigma_e=\sigma_e\pi(a)=\pi(a).$ Where the last equality comes form the fact that $\sigma_e$ is the projection onto the subspace $\overline{\pi(J_e)H}$. Then, if $s\leq t$ and $a\in \jrs \cont \jrt$ one has \begin{equation}\label{nucleus}
\pi\times \sigma(a\delta_s)=\pi(a)\sigma_s=\pi(a)\sigma_{\rho(s)t}=\pi(a)\sigma_{\rho(s)}\sigma_t=\pi(a)\sigma_t=\pi\times \sigma(a\delta_t).\end{equation}

\begin{rem}\label{converg}
	Let $s$ be an element of $S$. Suppose that $\{a_n\}\cont J_{\rho(s)}$ converges to $a$. Then, for every covariant pair $(\pi,\sigma)$, one has $||\pi\times \sigma(a_n\delta_s-a\delta_s)||=||\pi(a_n-a)\sigma_s||\leq ||a_n-a||$. In particular, $a_n\delta_s$ converges to $a\delta_s$ in $A\rtimes_\alpha S.$ Thus, if we suppose that $\{J'_e\}_{e\in E}$ is a family of subsets such that $J'_e$ is dense in $J_e$ then the subset $L'_{\alpha}=\{f\in \cc(S,A)\ |\ f(s)\in J'_{\rho(s)} \}$ is dense in $A\rtimes_\alpha S$. In particular, if for every $e\in E$ the subset $J'_e$ is an ideal of $J_e$ then $L'_{\alpha}$ is a dense subalgebra of $A\rtimes_\alpha S$.

	In particular, when $\alpha$ arises from a topological étale action, as in Remark \ref{topinducesalg}, we obtain that 
    \begin{equation}\label{smallerbutdense}
	\czero(X)\rtimes_\alpha S=\overline{\left\{\sum_{s\in S}a_s\delta_s\ |\ a_s\in \cc(\drs) \right\}}.
	\end{equation}
\end{rem}

\begin{rem}\label{remarkbelow}
	Let $B$ a unital $C^*$-algebra and $\tau:B\rightarrow B(H)$ be a $C^*$-algebra homomorphism. Note that $\tau(1)H$ is closed because $\tau(1)$ is a projection. Now, for every $b\in B$ and $h\in H$, one has $\tau(b)h=\tau(1)\tau(b)h\in \tau(1)H$. Hence $\spf\ \tau(B)H=\tau(1)H$.
\end{rem}

\begin{teo}\label{thm:A(S)semicrossed}
 Let $(S,E,\lambda,\rho)$ be a restriction semigroup, $\theta:S\rightarrow \mathcal{I}(\echap)$ be the canonical action, and $\alpha:S\rightarrow \mathcal{I}(\czero(\echap))$ be the étale action induced by $\theta$. Then, $$\A(S)\cong \czero(\echap)\rtimes_\alpha S .$$
\end{teo}

\begin{proof}
Define $\psi: \comp[S]\rightarrow L_{\alpha}$ by $\psi(\sum_{s\in S}a_s\delta_s)=\sum_{s\in S}a_s1_{\rho(s)}\delta_s$, where $1_e:=1_{D_e}$ is the characteristic function of $D_e$. Since $\psi$ is obviously linear, we are reduced to prove that it preserves the product. Recall that the inverse of $\theta_s$ is the map $\zeta_s$ defined in \eqref{definitionzeta} and note that for $s,t\in S$ 

$$\begin{aligned}
\psi(\delta_s)\psi(\delta_t)
    &=1_{\rho(s)}\delta_s1_{\rho(t)}\delta_t
    =\alpha_s(\alpha_s^{-1}(1_{\rho(s)})1_{\rho(t)})\delta_{st}
    \\&=\alpha_s((1_{\rho(s)}\circ\theta_s)1_{\rho(t)})\delta_{st}
    =\alpha_s(1_{\lambda(s)}1_{\rho(t)})\delta_{st}
    \\&=1_{D_{\lambda(s)}\cap D_{\rho(t)}}\circ\zeta_s\delta_{st}
    =1_{\theta_s(D_{\lambda(s)}\cap D_{\rho(t)})}\delta_{st}
    \\&=1_{\rho(st)}\delta_{st}
    =\psi(\delta_{st}).
\end{aligned}$$
Therefore 
$$\begin{aligned}
    \psi\left(\sum_{s\in S} a_s\delta_s\sum_{t\in S} b_t\delta_t\right)
    &=\psi\left(\sum_{s,t\in S} a_sb_t\delta_{st}\right)
  =\sum_{s,t\in S} a_sb_t\psi(\delta_{st})
   \\& =\sum_{s,t\in S} a_sb_t\psi(\delta_s)\psi(\delta_t)
    =\sum_{s,t\in S} \psi(a_s\delta_s)\psi(b_t\delta_t)
    \\&=\psi\left(\sum_{s\in S}a_s\delta_s\right)\psi\left(\sum_{t\in S}b_t\delta_t\right).
\end{aligned}$$

Next, we show that there is a bijection between representations of $S$ and covariant pairs for $\alpha$. For a representation $\sigma:S\rightarrow B(H)$ of $S$, we have that $\sigma$ restricts to $E$ as a semigroup homomorphism whose range consists of (orthogonal) projections. Hence let $\pi_{\sigma}: \czero(\echap)\rightarrow B(H)$ denote the homomorphism from Theorem $\ref{universal}$. We claim that $(\pi_{\sigma},\sigma)$ defines a covariant pair for $\alpha$. In fact, for $s\in S$ let us check that the \textit{covariance relation} holds. Since $\sps \{1_e\ |\ e\in E\text{ and }e\leq \lambda(s)\}$ is dense in $\operatorname{C}(D_{\lambda(s)})$, we only need to check the equality for the family $\{1_e\ |\ e\in E\text{ and }e\leq \lambda(s)\}$. Take $e\leq \lambda(s)$ and  note that 
	$$\alpha_s(1_e)(\phi)
	=1_e\circ \zeta_s(\phi)
	=1_e(\zeta_s(\phi))
	\stackrel{\eqref{avaliacao}}{=}\zeta_s(\phi)(e)
	=\phi(\rho(se))=1_{\rho(s.e)}(\phi).$$
 Therefore
  	$$\pi_\sigma(\alpha_s(1_e))\sigma_s
  	=\pi_\sigma(1_{\rho(se)})\sigma_s
  	=\sigma_{\rho(se)}\sigma_s
  	=\sigma_{\rho(se)s}
  	=\sigma_{se}
  	=\sigma_{s}\sigma_{e}
  	=\sigma_s\pi(1_e).$$
Now, note that $\sigma_e(H)=\pi_{\sigma}(1_e)H$ and $\pi_{\sigma}(1_e)H=\spf\pi_{\sigma}(\operatorname{C}(D_e))H,$ 
where the latter equality comes from Remark \ref{remarkbelow}. 
Then, $(\pi_{\sigma},\sigma)$ indeed defines a covariant pair.

Let us see that the map $\sigma \rightarrow (\pi_{\sigma},\sigma)$ defines a bijective correspondence from the set of representation of $S$ to the set of covariant pairs for $\alpha$. Clearly, it is injective. 
Moreover, if $(\pi,\sigma)$ is a covariant pair then $\pi(1_e)$ is a projection such that $\pi(1_e)H=\spf\ \pi(\operatorname{C}(D_e))H$, by Remark \ref{remarkbelow}. Then, $\pi(1_e)$ and $\sigma_e$ are projections sharing the same range, which implies that  $\pi(1_e)=\sigma_e$ and $\pi=\pi_{\sigma}$.

Finally, for an element $x=\sum_{s\in S}a_s\delta_s\in \comp[S]$ and $\sigma$ a representation of $S$ note that
\begin{equation}\label{normaparrep}
    \begin{aligned}
    &\|\pi_{\sigma}\times \sigma(\psi(x))\|
    =\left\|\pi_{\sigma}\times \sigma\big(\sum_{s\in S}a_s1_{\rho(s)}\delta_s\big)\right\|
    \\&=\left\|\sum_{s\in S}a_s\pi_{\sigma}(1_{\rho(s)})\sigma_s\right\|
    =\left\|\sum_{s\in S}a_s\sigma_{\rho(s)}\sigma_s\right\|
    \\&=\left\|\sum_{s\in S}a_s\sigma_s\right\|
   =\left\|\widetilde{\sigma}(\sum_{s\in S}a_s\delta_s)\right\|
    =\|\widetilde{\sigma}(x)\| .
   \end{aligned}\end{equation}
Hence, taking supremum over $\sigma$ we obtain that $\psi$ is isometric and extends to an isometric homomorphism $\psi_1:\A(S)\rightarrow \czero(\echap)\rtimes_\alpha S.$

 We finish by showing that the range of $\psi_1$ contains a dense subset. If $s\in S$ and $e\leq \rho(s)$, we have $\rho(es)=\rho(e\rho(s))=e\rho(s)=e.$ Thus, $\psi_1(\delta_{es})=1_{\rho(es)}\delta_{es}=1_{e}\delta_{es}\stackrel{\eqref{nucleus}}{=}1_{e}\delta_{s}.$
Since $\sps \{1_e\ |\ e\in E\text{ and }e\leq \rho(s)\}$ is dense in $\operatorname{C}(\drs)$, by Remark \ref{converg} we have that $\psi_1$ is an isomorphism
\end{proof}

\section{The category of germs of a restriction semigroup action}\label{sec:catgerms}

Let $(S,E,\lambda,\rho)$ be a restriction semigroup, $X$ be a locally compact Hausdorff space and $\theta:S\rightarrow \mathcal{I}(X)$ be an étale action. To the end of constructing an étale category associated with $\theta$, let $\Xi_0$ denote the set $\{(s,x)\ |\ x\in \dls\} \subseteq S\times X$. We say that two pairs $(s,x)$ and $(t,y)$ in $\Xi_0$ are equivalent if $x=y$ and $sf=tf$ for some $f\in E$ such that $x\in D_f$. In this case, we denote $(s,x) \sim (t,y)$. 

For simplicity, here we just prove that $\sim$ is transitive and the other properties are left to the reader. If $(s,x)\sim(t,x)$ and $(t,x)\sim (w,x)$ then $x\in D_e\cap D_f$, $se=te$ and $tf=wf$ for some $e,\ f \in E$. Therefore, $(s,x)\sim (w,x)$ since $sef=wef$ and $D_{ef}=D_e\cap D_f$. We say that the equivalence class of a pair $(s,x)$ is the \textit{germ} of $s$ at the point $x$. Moreover, note that if $(s,x)\sim (t,x)$ and $f\in E$ is a projection implementing the equivalence then \begin{equation}\label{actigual}
\theta_s(x)=\theta_s(\theta_f(x))=\theta_{sf}(x)=\theta_{tf}(x)=\theta_t(\theta_f(x))=\theta_t(x).
\end{equation}

\begin{prop}\label{produtogerms}
 If $(s,x), (s',x), (t',y)$  and $(t,y)$ are elements of $\ \Xi_0$ such that \linebreak $x=\theta_t(y)$, $(s',x) \sim (s,x)$ and $(t',y)\sim (t,y)$ then $(s't',y)\sim (st,y)$.
\end{prop}

\begin{proof}
Let $e,f\in E$ be such that $s'f=sf$, $t'e=te$, $x\in D_f$ and $y\in D_e$. By (P8) of Definition \ref{restrictionsemigroupdefinition} $\rho(te)t=te$ and $\rho(te)t'=\rho(t'e)t'=t'e=te.$ If $h=f\rho(te)$, then $ht=fte$ and $ht'=ft'e$. Consequently, $ht=ht'$ and 
    $$\begin{aligned}
    st\lambda(ht)
    &=sht
    =(sf)(te)
    =(s'f)(t'e)
    \\&=s'ht'
    =s't'\lambda(ht')
    =s't'\lambda(ht).
    \end{aligned}$$ 
Note that $D_{\lambda(ht)}=D_{\lambda(ft)}\cap D_{e}$ since $\lambda(ht)=\lambda(fte)=\lambda(\lambda(ft)e)=\lambda(ft)e$. Moreover \linebreak
$x=\theta_f(x)=\theta_f(\theta_t(y))=\theta_{ft}(y)$ and consequently $y\in D_{\lambda(ft)}$. Therefore $y\in D_{\lambda(ht)}$ since it already belongs to $D_e$.
\end{proof}

\begin{rem}\label{unitsp}
	If $x\in D_e\cap D_f$ for some $e,\ f \in E$, then $(e,x)\sim(f,x)$ because $e(ef)=f(ef)$ and $x\in D_{ef}$. Moreover, for $s\in S$ and $e\in E$, the pairs $(s,x), (e,x)$ are equivalent if and only if there exists a projection $h$ such that $sh=h$ and $x\in D_h$. In fact, if $(s,x)\sim (e,x)$ then $h= ef$ where $f$ is a projection implementing the equivalence. Conversely, if
    $sh=h$ and $x\in D_h$ then $(s,x)\sim (h,x)\sim (e,x)$.
\end{rem}

Next, we proceed to define the category of germs associated with $\theta$. Astute readers will observe the resemblance between this construction and that of the groupoid of germs emerging from an inverse semigroup action. For $\CZ=X$ and $\CU=\faktor{\Xi_0}{\sim}$, define the source  and the range maps to be $\s([s,x])=x$ and $\rr([s,x])=\theta_s(x),$ respectively. Consequently, $\CD$ is the subset $\{([s,x],[t,y])\in \CU\times\CU:\ x=\theta_t(y)\}$. Define the product $\m([s,x],[t,y])$ by $[st,y]$ and note that it is well-defined by Proposition \ref{produtogerms}. Finally, define the unit map to be $\un(x)=[e,x]$ where $e\in E$ is any element such that $x\in D_e$. Note that $\un$ is well-defined by Remark \ref{unitsp}. Thus it is a simple matter to check that $\C(\theta,S,X)=\left(\CZ,~\CU, \s, \rr,\m, \un\right)$ forms a category. For simplicity of notation we will write $\C(\theta,S,X)$ just as $\C$ whenever the action is known by context.

Let us equip $\CU$ with a topology. For every $s\in S$ and every open set $U\cont\dls$ define
$$\Theta(s,U):=\left\{[s,x]\ |\ x\in U\right\}.$$ 
In the case, $U=\dls$, we denote $\Theta\left(s,\dls\right)$  by $\Theta_s.$ We claim that the subset family $\left\{\Theta(s,U)\ |\ s\in S,\ U\cont \dls \text{ open set } \right\}$ forms a basis for a topology on $\CU$ by verifying the conditions of \cite[Theorem 5.3]{willard}. In fact, this family is a cover for $\CU$. 
Moreover, if $[s,u]=[t,v]\in \Theta(s,U)\cap \Theta(t,V)$, there exists $e\in E$ such that $se=te$ and $u=v\in U\cap V \cap D_e$. Hence, defining $W=U\cap V \cap D_e$, we have that $[s,u]\in \Theta(s,W)\cont \Theta(s,U)\cap \Theta(t,V).$ 

\begin{prop}
Equip $\CU$ with the topology generated by the basis $$\left\{\Theta(s,U)\ |\ s\in S,\ U\cont \dls \text{ open set } \right\}.$$ Then $\C(\theta,S,X)$ is étale.
 \end{prop}

\begin{proof}
Let $\Theta(s,U)$ be a basic open set of $\CU$. By Remark \ref{unitsp}, one has
$$\un^{-1}(\Theta(s,U))=\bigcup_{\substack{h\in E\\\text{ s.t }sh=h}}D_h\cap U.$$
On the other hand, if $U\cont X$ is an open subset
$$\un(U)=\un(\bigcup_{e\in E}U\cap D_e)=\bigcup_{e\in E}\Theta(e,U\cap D_e).$$
Thus, $\un$ is a one-to-one continuous open map since the right sides of the above equation are open subsets. In particular, $\un$ is an embedding. Furthermore, note that
$$\rr^{-1}\left(U\right)=\bigcup_{s\in S}\Theta\left(s, \theta_s^{-1}\left( U\cap {\drs}\right)\right)\text{ and } \s^{-1}\left(U\right)=\bigcup_{s\in S}\Theta\left(s, U\cap \dls\right).$$ 
The above equations show that both the range and source are continuous maps. Moreover, for every $s\in S$ it is easy to see that $\s_{\Theta_s}:\Theta_s\rightarrow \dls$ and $\rr_{\Theta_s}:\Theta_s\rightarrow \drs$ are bijective maps, and hence $\s(\Theta(s,U))=U$ and $\rr(\Theta(s,U))=\theta_s(U)$, which gives $\s_{\Theta_s}$ and $\rr_{\Theta_s}$ are open maps since $\{\Theta(s,U)\ |\ U\cont \dls,\ U \text{ open} \}$ is a basis for the topology of $\Theta_s$. Therefore, $\s$ and $\rr$ are local homeomorphisms since $\left\{\Theta_s\right\}_{s\in S}$ is a cover of $\CU$.

We finish by showing the continuity of $\m$. Let $([s,\theta_t(x)],[t,x])$ a composable pair and $\Theta(r,V)$ be a basic open set containing $\m([s,\theta_t(x)],[t,x])=[st,x]$. There exists $e\in E$ such that $x\in  D_e\cap V\cap D_{\lambda(st)}$ and $ste=re$. Hence, defining $W:=D_e\cap V\cap D_{\lambda(st)}$, we easily obtain that $\m\left(\left(~\Theta_s\times \Theta\left(t, W\right)~\right)\cap \CD\right)\cont \Theta\left(r,V\right),$ with $e$ implementing the equivalences.
\end{proof}

We would like to draw attention to the fact that we will proceed as before (see the comment after proof of Proposition \ref{prop:coment}) and identify an object $x$ with $[e,x]$, where $e$ is any projection such that $x\in D_e$. Hence, we can without lost of generality say that $X$ is an open subset of $\C(\theta,S,X)$. Note that with this identification  we have $\Theta_e=D_e$, for every $e\in E$. Moreover, the following diagram commutes
\begin{equation}\label{diagram}
\begin{tikzcd}[column sep=small]
& \Theta_s \arrow[swap]{dl}{\s} \arrow{dr}{\rr} & \\
\dls  \arrow{rr}{\theta_s}& & \drs
\end{tikzcd}
\end{equation}

Suppose $\D$ is an étale category and let $\theta:\bis(\D)\rightarrow \mathcal{I}\big(\DZ\big)$ be the action of Example \ref{actionbisc}. Following the above construction, we have the category of germs $\C\big(\theta,\bis(\D),\DZ\big)$ whose elements are equivalence classes of pairs $(U,x)$, where $U\in \bis(\D)$ and $x\in \s(U)$.

\begin{teo}[cf. Proposition 5.4 of \cite{exel}]\label{thm:everycategoryisagermscategory}
If $\D$ is an étale category, then $\D$ is isomorphic to the category of germs $\C\big(\theta,\bis(\D),\DZ\big)$, where $\theta:\bis(\D)\rightarrow \mathcal{I}\big(\DZ\big)$ is the action of the restriction semigroup of bisections of $D$ on the space $\DZ$.
\end{teo}

\begin{proof}
For simplicity, let $\C$ denote $\C\big(\theta,\bis(\D),\DZ\big)$. We have to check that there exists a homeomorphism $\varphi:\C\rightarrow \D$ satisfying: 
\begin{enumerate}[{\normalfont \rmfamily  (i)}]
    \item \label{thm:everycategoryisagermscategory1}
        $\varphi\big(\CZ\big)=\DZ$
    \item \label{thm:everycategoryisagermscategory2}
        $\varphi([U,x][V,y])=\varphi([U,x])\varphi([V,y])$, for every $([U,x],[V,y])\in \CD$.
\end{enumerate}
To this end, define $\varphi:\C \rightarrow \D$, $\varphi([U,x])=\s_U^{-1}(x)$. Note that if $[U,x]=[V,x]$, then there exists $F\cont \DZ$ such that $UF=VF$ and $x\in D_F=\s(F)=F.$ In this case, $\s_U^{-1}(x)\in UF=VF$ and $\s_V^{-1}(x)\in VF=UF$. Therefore, $\s_U^{-1}(x)=\s_V^{-1}(x)$ since the source of these elements is $x$, and $UF$ is a bisection. Therefore, $\varphi$ is well-defined.

To see that $\varphi$ is surjective, note that for every $z \in \D$, we have that $z=\varphi([Z,\s(z)])$, where $Z\in \bis(\D)$ is a bisection such that $z\in Z$. To see the injectivity of $\varphi$, note that if $\varphi([U,x])=\varphi([V,y])$, then $\s_U^{-1}(x)=\s_V^{-1}(y)\in U\cap V$, and hence, defining $F=\s(U\cap V)\cont \DZ$, we have that $x=\s(\s_U^{-1}(x))=\s(\s_V^{-1}(y))=y\in F$. Moreover,
$$UF=U\s(U\cap V)=U\cap V=V\s(U\cap V)=VF.$$
Thus, $F$ implements the equivalence $(U,x)\sim(V,y)$ and $[U,x]=[V,y]$. Finally, since $\varphi$ is bijective, $\varphi^{-1}(U)=\Theta(U,\s(U))=:\Theta_U$, for every $U\in \bis(\D)$. On the other hand, $\varphi(\Theta(V,Z))=\s_V^{-1}(Z)$, which gives that $\varphi$ is continuous, open and therefore a homeomorphism.

Once $\varphi$ is a homeomorphism, it remains to show items \eqref{thm:everycategoryisagermscategory1} and \eqref{thm:everycategoryisagermscategory2} above. Consider the composable pair $([U,\theta_V(x)],[V,x])$. On the one hand, we have
$$\varphi([U,\theta_V(x)][V,x])=\varphi([UV,x])=\s_{UV}^{-1}(x).$$
On the other hand, 
    $$\begin{aligned}
        \s(\varphi([U,\theta_V(x)]))
        &=\s(\s_U^{-1}(\theta_V(x)))
        =\theta_V(x)
        \\&=\rr_V (\s_V^{-1}(x))
        =\rr(\varphi([V,x])),
    \end{aligned}$$
which gives that $\big(\varphi([U,\theta_V(x)]),\varphi([V,x])\big)\in \D^{(2)}$ is a composable pair. Note that the product $\varphi([U,\theta_V(x)])\varphi([V,x])$ is in $UV$ and, moreover, 
$\s\big(\varphi\left([U,\theta_V(x)]\right)
\varphi\left([V,x]\right)\big)
=\s\big(\varphi([V,x])\big)=x.$ 
Thus, $\varphi\big([UV,x]\big)=\varphi\big([U,\theta_V(x)]\big)\varphi\big([V,x]\big)$ since both are in $UV$ and have $x$ as source.

To finish, we show that $\varphi$ preserves objects. If $[F,u] \in \CZ$, then  $F$ is an open subset of $\DZ$ and $u\in \s(F)=F$. Therefore $\varphi\big([F,u]\big)=\s_F^{-1}(u)=u\in \DZ$.
\end{proof}

\subsection{The category of germs of the canonical action}
Let $(S,E,\lambda,\rho)$ be a restriction semigroup, $\theta$ be the canonical action of $S$ on $\echap$ (see Subsection \ref{standardaction}) and $\C$ denote the category of germs $\C(\theta,S,\echap)$. Our purpose here is to study some properties of $\C$ and determine which properties of $S$ are reflected in $\C$. For instance, we prove that $S$ is left-ample if and only if $\C$ is left cancellative.

Recall that for every $e\in E$, the \textit{upper set} of $e$ is the set $e^{\uparrow}=\{f\in E\ \mid\ e\leq f\}$. Let $\bare$ be the semicharacter $1_{e^{\uparrow}}\in \echap,$ which is simply the characteristic function of $e^{\uparrow}$. The set $\Bare=\{\bare\ \mid\ e\in E \}$ has a special importance when dealing with the Banach algebra $\ell_1(E)$ (e.g \cite[Lemma 2.1.1]{paterson}). Note that if $\bare=\barf$, then $e\leq f$ and $f\leq e$, which gives $e=f$. 
 
 \begin{prop}\label{baresdense}
 	If $E$ is finite, then $\echap=\Bare$. In general, $\Bare$ is dense in $\echap$ and the subset $\{\bare\ | e\in E, e\leq h \}$ is dense in $D_h$, for every $h\in E$.
 \end{prop}
 
 \begin{proof}
 	If $E$ is finite and $\phi \in \echap$, then $\phi= \bare$, where $e$ is the minimum of $\phi^{-1}(\{1\})$.
 	%
 	%
 	%
     %
     
   	For the general case, take $\phi \in \echap$ and define the set $\Lambda=\{\alpha \cont \phi^{-1}(\{1\})\ |\ |\alpha|<+\infty \}$. Note that $\Lambda$ is a directed set, ordered by inclusion. Moreover, for any $\alpha \in \Lambda$, define 
   	        \begin{equation}\label{porcima} 
   	            e_\alpha=\prod_{a\in \alpha}a.
   	        \end{equation}
   	 	Let us prove that $\barn$ converges to $\phi$. For $f\in E$, we have either $\phi(f)=0$ or $\phi(f)=1$. Assuming $\phi(f)=0$, we have that $\barn(f)=0$ for every $\alpha\in \Lambda$. Otherwise, we would have $e_\alpha\leq f$ for some $\alpha \in \Lambda$ and, in this case, $1=\phi(e_\alpha)\leq \phi(f)$. Hence, $\barn(f)$ converges to $\phi(f)$. On the other hand, suppose $\phi(f)=1$. In this case, $\alpha_0:=\{f\}$ belongs to $\Lambda,$ and, $\barn(f)=1$, for all $\alpha\geq \alpha_0$. Thus, $\barn(f)$ converges to $\phi(f)$ again. In particular, if $\phi \in D_h$, note that the subnet $\{\barn\}_{\alpha\geq \{h\}}$ still converges to $\phi$ and hence $\{\bare\ | e\in E, e\leq h \}$ is dense in $D_h$.
\end{proof}

 Consider the set $\widetilde{S}=\left\{[s,\barls]\ |\ s \in S\right\}\cont \C$ and define the map
    \begin{equation}\label{sstilde}
        \varPsi:S\rightarrow \widetilde{S},~~~\varPsi(s)=[s,\barls]
    \end{equation}
Note that if $[s,\barls]=[t,\barlt]$ then $\lambda(s)= \lambda(t)$ and there exists a projection $e\in E$ such that $\barls \in D_e$ and $se=te$. In particular, this means that $\lambda(t)=\lambda(s)\leq e$ and $s=se=te=t$. This proves that $\varPsi$ is injective and hence bijective since the surjectivity is trivial.

For $t\in S$ and $e\in E$, we have $\lambda(t)\leq \lambda(et)\Leftrightarrow\lambda(t)=\lambda(et)\stackrel{Prop. \ref{propdomaincomp}}{\Leftrightarrow}\rho(t)\leq e,$ 
and hence 
$$\begin{aligned}
\theta_t\left(\barlt\right)(e)=\barlt\left(\lambda(et)\right)&=
\left\{\begin{array}{cc}
1,&\ \text{if}\ \lambda(t) \leq \lambda(et)\\
0,&\ \text{otherwise}
\end{array}\right.
\\&=\left\{\begin{array}{cc}
1,&\ \text{if}\ \rho(t)\leq e\\
0,&\ \text{otherwise}
\end{array}\right.
\\&=\varsigma_{\rho(t)}(e).
\end{aligned}$$
Therefore, $\theta_t\left(\barlt\right)=\varsigma_{\rho(t)}$ and
\begin{equation}\label{stildeconditiontocompose}\theta_t\left(\barlt\right)\in \dls \Leftrightarrow \rho(t)\leq \lambda(s).
\end{equation}

\begin{lemma}\label{stilde equal}
	If $\bare\in \dls$, then $[s,\bare]$ is equal to $[se,\varsigma_{\lambda(se)}]$. In particular, if $E$ is finite then $\widetilde{S}=\C(\theta,S,\echap)$ and, in general, $\widetilde{S}$ is dense in $\C(\theta,S,\echap)$. Moreover, $\{[t,\barlt]\ |\ t\leq s\}$ is dense in $\Theta_s$.
\end{lemma}

\begin{proof}
	Suppose that for $s\in S$ and $e\in E$ one has $\bare\in \dls$, then $e\leq\lambda(s)$ and $e=\lambda(s)e=\lambda(se)$. Hence, $[s,\bare]=[s,\varsigma_{\lambda(se)}]=[se,\varsigma_{\lambda(se)}]$, where clearly $e$ implements the latter equivalence. By Proposition \ref{baresdense}, if $E$ is finite we have that any germ $[s,\phi]$ is of the form $[s,\bare]$ which is equal to $[se,\varsigma_{\lambda(se)}]\in \widetilde{S}.$
	
	In general, let $[s,\phi]$ be a germ and $\{\barn\}_{\alpha\in \Lambda}$ be a net converging to $\phi$ such that $\barn\in \dls, \forall\ \alpha \in \Lambda$. Thus, since $\s_{\Theta_s}$ is a homeomorphism, we obtain that $[s,\barn]$ converges to $[s,\phi]$. But $[s,\barn]=[se_\alpha,\varsigma_{\lambda(se_\alpha)}]\in \widetilde{S}$, which tells us that $\widetilde{S}$ in dense in $\C(\theta,S, \echap)$. In particular, $\left\{[t,\barlt]\ |\ t\leq s\right\}$ is dense in $\Theta_s$.
\end{proof}

\begin{prop}\label{leftmultinvariant}
	$\widetilde{S}$ is closed by left composition. Moreover, $\widetilde{S}$ is the subcategory of $\C(\theta,S,\echap)$ obtained by restricting the objects to $\{\bare\ |  e \in E\}$, that is, $\widetilde{S}=\rr^{-1}(\Bare)\cap\s^{-1}(\Bare)$. Furthermore, the category $\widetilde{S}$ is isomorphic to the restriction semigroup category $\C_S$ via $\varPsi$.
\end{prop}

\begin{proof}
	We will show that whenever an element $[s,x]$ composes with $[t,\barlt]$, the composition belongs to $\widetilde{S}$. So, assume the pair $([s,x], [t,\barlt])$ is composable and note that $x=\theta_t(\barlt)$. By Equation \eqref{stildeconditiontocompose}, $\rho(t)\leq \lambda(s)$. By Proposition \ref{propdomaincomp}, it is equivalent to say $\lambda(t)=\lambda(st)$. Hence, \begin{equation}\label{sigmainvariantstilde}
	[s,x][t,\varsigma_{\lambda(t)}]=[st,\varsigma_{\lambda(t)}]=[st,\varsigma_{\lambda(st)}].\end{equation} In particular, $\widetilde{S}$ is closed by multiplication. Furthermore, since $\rr([t,\barlt])=\theta_t(\barlt)=\varsigma_{\rho(t)}$, we have that $\widetilde{S}\cont \rr^{-1}(\Bare)\cap\s^{-1}(\Bare)$. The converse is given by the Lemma \ref{stilde equal}. In fact, if $[s,\bare]$ belongs to $\rr^{-1}(\Bare)\cap\s^{-1}(\Bare)$, then $[s,\bare]=[se,\varsigma_{\lambda(se)}]\in \widetilde{S}$.
	Next, note that a pair $([s,\barls],[t,\barlt])$ is composable if, and only, if $\theta_t(\barlt)=\varsigma_{\rho(t)}=\barls$ if, and only, if $\rho(t)=\lambda(s)$. In particular, $\varPsi$ is a bijective functor from $\C_S$ to $\widetilde{S}$.
\end{proof}

\begin{prop}\label{ampleimpliecancel}
Let $S$ be a restriction semigroup, $X$ be a locally compact Hausdorff space and $\theta:S\rightarrow \mathcal{I}(X)$ be an étale action. Moreover, suppose that $S$ is left-ample. Then for every triple $(s,t,v)\in S^3$ and $x\in X$ such that $[st,x]=[sv,x]$ we have $[t,x]=[v,x]$. In particular, $\C(\theta,S,X)$ is left cancellative.
\end{prop}

\begin{proof}
Suppose $[st,x]=[sv,x]$. In this case, there exists $e\in E$ such that $ste=sve$ and $x\in D_e.$ Therefore, $\lambda(s)te=\lambda(s)ve$ and, equivalently, $t\lambda(st)e=v\lambda(sv)e$. Multiplying both sides by $\lambda(st)\lambda(sv)$, we obtain $t\lambda(st)\lambda(sv)e=t\lambda(st)\lambda(sv)e.$ Thus, since $D_{\lambda(st)\lambda(sv)e}=D_{\lambda(st)}\cap D_{\lambda(sv)}\cap D_e$, we have $[t,x]=[v,x].$

Now, suppose that $[s,y][t,x]=[s,y][w,z]$. In this case,  $[st,x]=[sw,z]$ and $x=z$. Therefore $[t,x]=[w,z]$, by the previous calculation.
\end{proof}

The converse of the above Proposition is not valid. Let $S$ be a restriction semigroup with $0$, where $0$ is a projection, and $\theta: S\rightarrow \ii(X)$ be the trivial action of $S$ on a locally compact Hausdorff space $X$, that is, $\theta_s=\id_X, \forall \ s \in S$. The category of germs of $\theta$ consists only of objects since $(s,x)\sim (t,x)$ for all $s,t\in S$. In fact, it holds because $s0=0=t0$ and $D_0=X$. Therefore, $\C(\theta,S,X)=\C(\theta,S,X)^{(0)}\cong X$ and, in particular, it is (left) cancellative. However, it does not imply that $S$ is left-ample. To see this, it suffices to present an example of a non left-ample restriction semigroup with $0$ in which $0$ is a projection.

Let $X$ be a set with more than $2$ elements. The set of all functions from $X$ to $X$, $\mathcal{F}(X)$, is a monoid under composition. This monoid does not have a zero, i.e, it does not exist a function $\varsigma:X\rightarrow X$ such that $f\varsigma=\varsigma=\varsigma f, \forall f\in \mathcal{F}(X).$ Indeed, if $y_0,y_1\in X$ and $y_0\neq y_1$, considering the constant maps $f_0(x)=y_0$, $f_1(x)=y_1$ one have that $f_0\varsigma\neq f_1\varsigma$, for any function $\varsigma \in \mathcal{F}(X)$. 

In order to obtain a restriction semigroup with $0$, we add a formal $0$ to $\mathcal{F}(X)$ and define $E=\{0,\id_X\}$, $\lambda(s)=0\Leftrightarrow s=0$, and $\rho(s)=0\Leftrightarrow s=0$. This monoid has no zero divisors, then it is straightforward to verify the requirements of Definition \ref{restrictionsemigroupdefinition}. Thus, ($\mathcal{F}(X)\cup \{0\},~\{0,\id_X\},~\lambda,~\rho)$ is a restriction semigroup with $0$ in which $0$ is a projection. Of course $\mathcal{F}(X)\cup \{0\}$ is not left-ample since $f_0f=f_0 g$, for any pair $f,g \in \mathcal{F}(X)$.

We have seen that it is possible to find a left cancellative category of germs without assuming the semigroup to be left-ample. Indeed, being cancellative has to do not only with the semigroup, but also with the action. But, the next proposition shows us that being left-ample can be characterized in categorical terms

\begin{prop}
$S$ is left-ample if, and only if, $\C(\theta, S, \echap)$ is left cancellative.
\end{prop}

\begin{proof}
Proposition \ref{ampleimpliecancel} gives us one side of the proof. So, assume $\C(\theta, S, \echap)$ is left cancellative and suppose $st=sw$, for $s,t,w\in S$. Define $e=\lambda(st)=\lambda(sw)$ and note that 
    $$
    \theta_t(\bare)(\lambda(s))
    =\bare(\lambda(\lambda(s)t))
    =\bare(\lambda(st))
    =\bare(e)=1.
    $$
Hence, $\theta_t(\bare)\in \dls$ and, similarly, $\theta_w(\bare)\in \dls$. Note that  
    $$\begin{aligned}
    \rr([s,\theta_t(\bare)])
    &=\theta_s(\theta_t(\bare))
   =\theta_{st}(\bare)
    =\theta_{sw}(\bare)
    \\&=\theta_s(\theta_w(\bare))
    \rr([s,\theta_w(\bare)]).
    \end{aligned}$$
Then $[s,\theta_t(\bare)]$ is equal to $[s,\theta_w(\bare)]$ because they have the same range and belong to the bisection $\Theta_s$. Furthermore $[s,\theta_t(\bare)][t,\bare]=[st,\bare]=[sw,\bare]=[s,\theta_w(\bare)][w,\bare]$. Thus, by left cancellativity, we have $$[t,\bare]=[w,\bare].$$
Then, there is a projection $f\in E$ such that $e\leq f$ and $tf=wf$. Multiplying both sides by $e$, we obtain that $te=we$ and  
    $$\begin{aligned}
    \lambda(s)t
    &=t\lambda(\lambda(s)t)
    =t\lambda(st)
    =te
    =we
    \\&=w\lambda(sw)
    =w\lambda(\lambda(s)w)
    =\lambda(s)w.
    \end{aligned}$$
In conclusion, $S$ is left-ample.
\end{proof}

\subsection{The semicrossed product structure of the operator algebra of an étale category}
Throughout this section, let $(S,E,\lambda,\rho)$ be a restriction semigroup, $\theta:S\rightarrow \mathcal{I}(X)$ be an étale action of $S$ on a second countable locally compact Hausdorff space $X$ and $\C$ be the category of germs $\C(\theta,S,X)$. We have defined so far two operator algebras associated with the action $\theta$: the operator algebra $\A(\C)$ and the semicrossed product algebra $\czero(X)\rtimes_\alpha S$, where $\alpha$ is the étale action induced by $\theta$ (Remark \ref{topinducesalg}). The purpose of this section is to show that $\A(\C)$ and $\czero(X)\rtimes_\alpha S$ are isometrically isomorphic. The reader will note that the second countability assumption is crucial to obtain this isomorphism. In fact, we will strongly use that every open set of a locally compact Hausdorff space is itself a second countable locally compact Hausdorff space, and hence it is $\sigma$-compact. 

\begin{prop}[\cf Theorem 2.6.3\cite{arveson2002short}]\label{bfc}
Let $Y$ be a locally compact Hausdorff space and $\pi:\czero(Y)\rightarrow B(H)$ be a $C^*$-homomorphism. Then, $\pi$ extends to a $C^*$-homomorphism $\widetilde{\pi}:B(Y)\rightarrow B(H)$, where $B(Y)$ is the set of bounded Borel measurable functions. Moreover, if $g_n$ converges pointwise to $g$ and $\sup \|g_n\|_{\infty}<+\infty$ then $\widetilde{\pi}(g_n)$ converges to $\widetilde{\pi}(g)$ in the weak operator topology. 
\end{prop}

\begin{proof}[Sketch of the proof]
For every $\xi, \eta\in H$, we have the continuous linear functional $\tau_{\xi,\eta}:\czero(Y)\rightarrow B(H)$ given by $f\mapsto\pint{\pi(f)\xi}{\eta}$. Then, by Riesz representation theorem, there is a complex Borel regular measure $\nu_{\xi,\eta}$ such that
$$\tau_{\xi,\eta}(f)=\int f\der\nu_{\xi,\eta}, \forall \ f\in \czero(Y).$$
Next, for every $g\in B(X)$, we have the bounded sesquilinear form $\sigma_g:H\times H \rightarrow \comp$ given by
$$\sigma_g(\xi,\eta)=\int g\ \der\nu_{\xi,\eta}.$$ 
Again, by (another) Riesz representation theorem, there exists $\pti(g)\in B(H)$ such that
$$\sigma_g(\xi,\eta)=\pint{\pti(g)\xi}{\eta}.$$ 
Then, the map $\pti:B(X)\rightarrow B(H)$ is a $C^*$-homomorphism. Moreover, if $f\in \czero(Y)$ we have that 
$$\pint{\pti(f)\xi}{\eta}=\sigma_f(\xi,\eta)=\int f\ \der\nu_{\xi,\eta}=\tau_{\xi,\eta}(f)=\pint{\pi(f)\xi}{\eta}, \forall \ \xi,\eta\in H.$$ 
Then, $\pti(f)=\pi(f)$, which shows that $\pti$ extends $\pi$. To finish, if $g_n$ converges pointwise to $g$ and $\sup \|g_n\|_{\infty}<+\infty$, by Lebesgue's dominated converge theorem 
$$\lim\limits_{n \rightarrow \infty} \pint{\pti(g_n)\xi}{\eta}=\lim\limits_{n \rightarrow \infty}\int g_n\ \der\nu_{\xi,\eta}=\int g\ \der\nu_{\xi,\eta}=\pint{\pti(g)\xi}{\eta}.$$ 
Hence, $\pti(g_n)$ converges to $\pti(g)$ in the weak operator topology.
\end{proof}

\begin{cor}\label{pi(1u)}
Suppose we are in the same conditions of Proposition \ref{bfc}. Then, for a $\sigma$-compact open set $U\cont Y$, it holds that $\pti (1_U)H=\spf\ \pi(\czero(U))H$.
\end{cor}

\begin{proof}
For every $f\in \czero(U)$ and $h\in H$, we have
$$\pi(f)h=\pti(f)h=\pti(1_U)\pti(f)h\cont \pti(1_U)H.$$ 
The function $1_U$ is a projection in the $C^*$-algebra $B(Y)$  since $\overline{1_U}=1_U=1_U1_U$. Hence, the operator $\pti(1_U)$ is an orthogonal projection, which implies its range is closed. Therefore, 
$$\spf\ \pi(\czero(U))H\cont\pti(1_U)H.$$
Conversely, let 
$\{K_n\}_{n\in \nat}$ be an increasing sequence of compact sets such that $U=\cup_{n\in \nat} K_n$. By Urysohn's Lemma, there exists  $g_n\in \cc(U)$ such that $g_{n_{|K_n}}\equiv 1$ and $\|g_n\|_{\infty}=1$. Note that $g_n$ converges pointwise to $1_U$ and hence $\pti(1_U)h$ is the weak limit of $\pi(g_n)h$, for every $h\in H$. The Hahn-Banach theorem then implies $\pti(1_U)h\in \spf \pi(g_n)h \cont \spf\ \pi(\czero(U))H $ and hence $\pti(1_U)H \cont \spf \pi(\czero(U))H $.
\end{proof}
%

Recall that each map $\theta_s:\dls \rightarrow\drs$ is a homeomorphism, that a basis for the topology of $\C$ is formed by the sets $\Theta(s,U)$, $U\cont \dls$ open, and that if $U=\dls$ then $\Theta(s,U)$ was simply denoted by $\Theta_s$. Since $\{\ts\mid s\in S\}$ forms a cover for $\C$, by Proposition \ref{prop: sumofbisections} we have that
$$\CCC=\sps\{f:\ f \in \cc(\ts),\ s\in S\}.$$ 
We have also proved that the maps $\s_{\Theta_s}:\Theta_s\rightarrow \dls$ and $\rr_{\Theta_s}:\Theta_s\rightarrow \drs$ are homeomorphisms. Therefore, the following applications are isometric $\ast$-isomorphisms.
$$\cc(\drs)\rightarrow\cc(\ts), f\mapsto f\circ \rr_{\Theta_s} \text{ and }\cc(\dls)\rightarrow\cc(\ts), f\mapsto f\circ \s_{\Theta_s}.$$
For maps $f\in \cc\big(\drs\big)$ and $g\in \cc\big(\dls\big)$ we will denote 
\begin{equation}\label{fdeltacategoria}
f\delta_s:= f\circ \rr_{\Theta_s}\text{ and } \delta_sg:=g\circ \s_{\Theta_s}
\end{equation} 
It then follows that 
\begin{equation}\label{eq:spandelta_s}
  \CCC=\sps\left\{f\delta_s:\ f \in \cc\big(\drs\big),\ s\in S\right\}.  
\end{equation}\
And
$$\CCC=\sps\left\{\delta_sf:\ f \in \cc\big(\dls\big),\ s\in S\right\}.$$
Recall that the $\ast$-isomorphism $\alpha_s:\czero\big(\dls\big)\rightarrow \czero\big(\drs\big)$ is given by $\alpha_s(f)=f\circ \theta_s^{-1}$. In this case, if $f\in \cc\big(\dls\big)$ we have that
 \begin{equation}\label{eq:permutingdelta}
 	\alpha_s(f)\delta_s
 	=f\circ\theta_s^{-1}\circ\rr_{\ts}
    \stackrel{\eqref{diagram}}{=}f\circ \s_{\ts}
    =\delta_s f.
\end{equation}
Moreover, for $f\in \cc\big(\drs\big)$ and $g\in \cc\left(\drt\right)$, we have 
\begin{equation}\label{prod1}
	\begin{aligned}
		f\delta_s\ast g\delta_t\big([st,x]\big)
		&=f\delta_s\big([s,\theta_t(x)]\big) g\delta_t\big([t,x]\big)
		\\&=f\big(\theta_{st}(x)\big)g\big(\theta_t(x)\big)
		\\&=\left(f\circ \theta_s \circ \theta_s^{-1}\right)\left(g\circ\theta_s^{-1}\right)(\theta_{st}(x))
		\\&=\big((f\circ \theta_s) g\big)\circ\theta_s^{-1}\left(\theta_{st}(x)\right)
		\\&=\alpha_s\left(\alpha_s^{-1}(f)g\right)(\theta_{st}(x))
		\\&= \alpha_s\left(\alpha_s^{-1}(f)g\right)(\rr([st,x]))
		\\&=\alpha_s\left(\alpha_s^{-1}(f)g\right)\delta_{st}([st,x]),
	\end{aligned}
\end{equation}
which gives
\begin{equation}\label{prod2}
f\delta_s\ast g\delta_t=\alpha_s(\alpha_s^{-1}(f)g)\delta_{st}.
\end{equation}

To the end of showing that there is a correspondence between covariant pairs for $(\alpha,S,\czero(X))$ and representations of $\CCC$, below we prove some preliminary results.

\begin{defi}
Let $Z$ be a set, $P$ and $L$ be subsets of $\PA(Z)$. $P$ will be called a \textbf{$\pi$-system} if it is closed under finite intersections, and L will be called a\textbf{ $\lambda$-system} if satisfies: \begin{enumerate}[{\normalfont \rmfamily  (i)}]
\item $Z\in L$.
\item If $A\in L$, then $Z\setminus A \in L$, for every $A\in L$.
\item If $\{A_n\}_{n\in\nat}\cont L$, then $\cup_{n\in \nat}A_n\in L$, for every sequence $\{A_n\}_{n\in \nat}$ of pairwise disjoint open sets.
\end{enumerate}
\end{defi}

\begin{teo}[Dynkin's $\pi-\lambda$ theorem (Theorem 3.2 of \cite{dyinks}) ]\label{dynkins}
Let $Z$ be a set, $P\cont \PA(Z)$ be a $\pi$-system and $L\cont \PA(Z)$ be a $\lambda$-system such that $P\cont L$. Then, $\sigma(P)$, the $\sigma$-algebra generated by $P$, is contained in $L$
\end{teo}

Since the map $\theta_s:\dls\rightarrow \drs$ is a homeomorphism, the following map is a $C^*$-isomorphism
$$\alpha_s:B\big(\dls \big) \longrightarrow
B\big(\drs \big),\  f \longmapsto f\circ\theta_s^{-1}.$$
Note that it extends the previously defined map $\alpha_s:\czero\big(\dls\big)\rightarrow \czero\big(\drs\big)$ and for this reason we keep the same notation. Moreover, let $(\pi,\sigma)$ be a covariant pair for $\alpha$ and $\pti$ be the Borel extension of $\pi$. Since every subset $D_e$ is $\sigma$-compact, by virtue of Corollary \ref{pi(1u)} and item \eqref{defipari2} of Definition \ref{defipar}, we have that for every $e \in E$
    \begin{equation}\label{se=pie}
    \pti(1_e)=\sigma_e,\ 1_e:=1_{D_e}.
    \end{equation} 
Thus, for every $e\in E$ and $f \in B(X)$ we have that
\begin{equation}\label{secomute}
\sigma_e\pti(f)=\pti(1_ef)=\pti(f1_e)=\pti(f)\sigma_e.
\end{equation}
Can the covariance relation be extended to the measurable context? Below, we give a positive answer to this question.

\begin{prop}\label{extendedcovariance}
Let $(\pi,\sigma)$ be a covariant pair for $(\alpha,S,\czero(X))$ on a Hilbert space $H$ and $\pti$ 
 be the Borel extension of $\pi$. Then, for every $f\in B\big(\dls\big)$ the following equation holds $$\pti(\alpha_s(f))\sigma_s=\sigma_s\pti(f).$$
\end{prop}
\begin{proof}
Let $U$ be an open set of $\dls$. Note that $U$ is $\sigma$-compact and recall that there exists a sequence $\{g_n\}\cont \cc(U)\cont\cc(\dls)$ such that $g_n$ converges pointwise to $1_U$ and $\sup\|g_n\|_{\infty}=1$ (see the proof of Corollary \ref{pi(1u)}). Then, for every $\xi,\eta\in H$
$$\begin{aligned}
    \pint{\pti(\alpha_s(1_U))\sigma_s\xi}{\eta}
    &=\lim\limits_{n \rightarrow \infty}\pint{\pti(\alpha_s(g_n))\sigma_s\xi}{\eta}
    \\&=\lim\limits_{n \rightarrow \infty}\pint{\sigma_s\pti(g_n)\xi}{\eta}
    \\&=\pint{\sigma_s\pti(1_U)\xi}{\eta}.
\end{aligned}$$
Thus, the covariance relation holds for characteristic functions of open sets. In what follows, we will check that the subset $\Lambda$, defined down below, is a $\lambda$-system.
$$\Lambda=\left\{A\cont \dls\ |\ A\text{ is a Borel subset such that } \pti(\alpha_s(1_A))\sigma_s=\sigma_s\pti(1_A) \right\},$$ %
Note that $\dls \in \Lambda,$ by the above calculation. Moreover, if $A \in \Lambda$, note that
    $$\begin{aligned}
    \pti(\alpha_s(1_A))\sigma_s+\pti(\alpha_s(1_{\dls\sm A}))\sigma_s        &=\pti(\alpha_s(1_{\lambda(s)}))\sigma_s
    \\&=\sigma_s\pti(1_{\lambda(s)})
    \\&=\sigma_s\pti(1_A)+\sigma_s\pti(1_{\dls\sm A})
    \\&=\pti(\alpha_s(1_A))\sigma_s+\sigma_s\pti(1_{\dls\sm A}).
    \end{aligned}$$ 
Hence, $\dls\setminus A \in \Lambda$.
Finally, suppose $\{A_n\}_{n\in\nat}\cont \Lambda$ is a sequence of pairwise disjoint open sets. Define $A=\cup_{n\in \nat}A_n$ and note that $1_A$ is the pointwise limit of $f_n,$ where $f_n=\sum_{i\leq n}1_{A_i}$. Then, for every $\xi, \eta \in H$ we have that 
    $$\begin{aligned}
	\pint{\pti(\alpha_s(1_A))\sigma_s\xi}{\eta}
	&=\lim\limits_{n \rightarrow \infty}\pint{\pti(\alpha_s(f_n))\sigma_s\xi}{\eta}
	\\&=\lim\limits_{n \rightarrow \infty}\pint{\sigma_s\pti(f_n)\xi}{\eta}
	\\&=\pint{\sigma_s\pti(1_A)\xi}{\eta},
    \end{aligned}.$$
Hence $A\in L$ and $\Lambda$ is a $\lambda$-system. 

Note that the topology of $\dls$ is a $\pi$-system contained in $\Lambda$ and thus $\Lambda$ is the Borel $\sigma$-algebra of $\dls$, by Theorem \ref{dynkins}. In particular, the covariance relation holds for simple functions and, consequently, for positive functions since every positive function is the limit of simple functions (see \cite[Theorem 13.5]{dyinks}) and $\pti$ is weakly continuous. Finally, by linearity, we have that the covariance relation holds for the whole algebra $B\big(\dls\big)$.
\end{proof}

The subsequent lemma generalizes \cite[Lemma 8.4]{exel} and, interestingly, a proof that avoids assuming the countability of $S$ is provided. Before proceeding, it is worth revisiting the detail mentioned in \eqref{eq:spandelta_s}.
\begin{lemma}\label{8.4}
	If $\sum_{s\in S}f_s\delta_s=0$ in $\CCC$ then $\sum_{s\in S}\pi(f_s)\sigma_s=0$ for every covariant pair $(\pi,\sigma)$ associated with $(\alpha,S,\czero(X))$.
\end{lemma}

\begin{proof}
Let $(\pi,\sigma)$ be a covariant pair for $\alpha$ on a Hilbert space $H$. Thus $\sigma:S \rightarrow B(H)$ and $\pi:\czero(X)\rightarrow B(H)$ are representations of $S$ and $\czero(X)$ satisfying the conditions stated in Definition \ref{defipar}. By Proposition $\ref{bfc}$ we have the Borel extension $\pti$ of $\pi$ and the measures $\nu_{\xi,\eta}$ on $X$ satisfying 
$$\nu_{\xi,\eta}(B)=\pint{\pti(1_B)\xi}{\eta},\ \xi, \eta \in H.$$ 
Consider $\xi, \ \eta \in H$ and let $\nu_s$ denote the Borel measure $\nu_{\xi,\sigma_s^*\eta}$ restricted to the Borel subspace $\dls$ of $X$ for each $s \in S$. Moreover, denote $\miss$ the pushforward measure of $\nu_s$ by the homeomorphism $\s_{\ts}^{-1}:\dls\rightarrow \ts$. More precisely, $\miss$ is the Borel measure on $\ts$ given by \begin{equation}\label{miss}
	\miss(B)=\nu_s(\s_{\ts}(B))=\pint{\sigma_s\pti(1_{\s_{\ts}(B)})\xi}{\eta},
\end{equation} 
for every Borel subset $B\cont \ts$. Note that $1_{\s_{\Theta_s}(B)}=1_B\circ \s_{\Theta_s}^{-1}$ and hence 
\begin{equation}\label{intteta}
\int f\ \der \miss=\pint{\sigma_s\pti(f\circ \s_{\Theta_s}^{-1})\xi}{\eta},	
\end{equation}
 for every Borel function $f \in B(\Theta_s)$. In fact, equation \eqref{miss} implies that \eqref{intteta} holds for simple functions and the extension to Borel functions comes from the fact that $\pti$ is weakly continuous.

Now, we check that $\miss$ and $\mitt$ coincide on $\Theta_s\cap\Theta_t$. Let $K\cont \Theta_s\cap\Theta_t$ be a compact subset.  Note that for every $x\in \s(K)$, $[s,x]=[t,x]$. Hence, for every $x\in X$ there exists $e_x\in E$ such that $x\in D_{e_x}$ and $se_x=te_x$. Moreover, since $\s(K)$ is compact, there are $e_1,...,e_n$ such that $\s(K)\cont \cup_{i=1}^nD_{e_i}$ and $se_i=te_i$. In this case, we write $\s(K)$ as a disjoint union 
$$\s(K)=\cup_{i=1}^n A_i, $$
where $A_1=\s(K)\cap D_{e_1}$ and $A_i=\s(K)\cap D_{e_i}\sm (\cup_{j=1}^{i-1}\s(K)\cap D_{e_j})$, for $2\leq i \leq n$.
 Note that $\{A_i\}_{i=1}^n$ is a family of pairwise disjoint and measurable subsets such that $A_i\cont D_{e_i}$. Hence, for every $i\in \{1,...,n\}$ 
 $$\begin{aligned}
\sigma_s\pti(1_{A_i})&=\sigma_s\pti(1_{e_i}1_{A_i})
=\sigma_s\pti(1_{e_i})\pti(1_{A_i})
\\&\stackrel{\eqref{se=pie}}{=}\sigma_s\sigma_{e_i}\pti(1_{A_i})
=\sigma_{se_i}\pti(1_{A_i})
\\&=\sigma_{te_i}\pti(1_{A_i})
=\sigma_{t}\pti(1_{e_i})\pti(1_{A_i})
\\&=\sigma_{t}\pti(1_{A_i})
.\end{aligned}$$ 
Thus, 
$$\miss(K)\stackrel{\eqref{miss}}{=}\sum_{i=1}^n\pint{\sigma_s\pti(1_{A_i})\xi}{\eta}=\sum_{i=1}^n\pint{\sigma_t\pti(1_{A_i})\xi}{\eta}=\mitt(K).$$
Since $\Theta_s\cap\Theta_t$ is $\sigma$-compact, we obtain that every closed subset is a countable union of compact subsets and, therefore, $\miss$ and $\mitt$ coincide in the family of closed subsets of $\Theta_s\cap\Theta_t$, which is a $\pi$-system. It is easy to see that the family of subsets where $\miss$ and $\mitt$ coincide is a $\lambda$-system. Hence, by Theorem \ref{dynkins}, $\miss$ and $\mitt$ coincide on every Borel subset of $\Theta_s\cap\Theta_t$. 

As a consequence of the fact that $\miss$ and $\mitt$ coincide on $\Theta_s\cap\Theta_t$, we obtain that if $F\in B(\Theta_s\cap\Theta_t)$ then $\int F\der \miss=\int F\der \mitt.$ Therefore, by \eqref{intteta}, we have
$$\pint{\sigma_s\pti(F\circ \s_{\Theta_s}^{-1})\xi}{\eta}=\pint{\sigma_t\pti(F\circ \s_{\Theta_t}^{-1})\xi}{\eta}.$$ 
Then, by Proposition \ref{extendedcovariance}, we have 
$$\pint{\pti(F\circ \rr_{\Theta_s}^{-1})\sigma_s\xi}{\eta}=\pint{\pti(F\circ \rr_{\Theta_t}^{-1})\sigma_t\xi}{\eta}.$$
And hence, since $\xi,\ \eta$ have been arbitrarily chosen, we obtain that for every Borel function $F \in B(\Theta_s\cap\Theta_t).$
\begin{equation}\label{equalonintersection}
\pti(F\circ \rr_{\Theta_s}^{-1})\sigma_s=\pti(F\circ \rr_{\Theta_t}^{-1})\sigma_t.
\end{equation}

Let $J=\{s_1,...,s_n\}$ denote the set of all elements $s$ in $S$ such that $f_s$ is nonzero. Moreover, define $M= \cup_{i=1}^n\Theta_{s_i}$ and recall that the Borel $\sigma$-algebra of $M$ is the collection of all unions $\cup_{i=1}^nB_{s_i}$ where $B_{s_i}$ is a Borel set of $\Theta_{s_i}$. This readily follows from the fact that if $V$ is an open subset of the topological space $Z$ then the Borel $\sigma$-algebra of $V$ is the collection of all $B\cont V$ such that $B$ is a Borel subset of $Z$. Additionally, we can suppose that the unions $\cup_{i=1}^nB_{s_i}$ are all disjoint since there exists a family $\{A_{s_1},..., A_{s_n}\}$ of pairwise disjoint Borel subsets such that $A_{s_i}\cont B_{s_i}\cont \Theta_{s_i}$ and $\cup_{i=1}^nB_{s_i}= \cup_{i=1}^nA_{s_i}$. In fact, define $A_{s_1}=B_{s_1}$ and $A_{s_i}=B_{s_i}\sm (\cup_{j=1}^{i-1}B_{s_j})$ for every $i\in \{2,...,n\}$. To check that $A_{s_i}$ is a Borel subset of $\Theta_{s_i}$ note that it is the intersection of measurable subsets of $M$ and hence measurable in $M$. As Borel subsets of $\ts$ are precisely the Borel subsets of $M$ contained in $\Theta_s$, we obtain that each $A_{s_{i}}$ is a Borel subset of $\Theta_{s_{i}}$. 

Define the following measure on $M$ $$\mu\big(\bigsqcup_{i=1}^nB_{s_i}\big)=\sum_{i=1}^{n}\mu_{s_i}(B_{s_{i}}).$$

 This measure is well-defined since for 
 $B=\sqcup_{i=1}^{n}A_{s_{i}}=\sqcup_{j=1}^{n}B_{s_{j}}$ we have
$$\begin{aligned}
 \sum_{i=1 }^n\mu_{s_i}(A_{s_i})&=\sum_{i=1 }^n\sum_{j=1 }^n\mu_{s_i}(A_{s_i}\cap B_{s_j})=\sum_{i=1 }^n\sum_{j=1 }^n\mu_{s_j}(A_{s_i}\cap B_{s_j})\\&=\sum_{j=1 }^n\sum_{i=1 }^n\mu_{s_j}(A_{s_i}\cap B_{s_j})= \sum_{j=1 }^n\mu_{s_j}(B_{s_j}).\end{aligned}$$ 
Hence, $\mu$ is a measure on $M$ extending $\miss$ for every $s \in J$. In particular, for every Borel function $f\in \cc(\Theta_s).$
$$\int f\der \mu=\int f\der \mu_s.$$
Finally, since $\sum_{s\in J}f_s\delta_s=0$, integrating both sides with respect to $\mu$ gives
$$\sum_{s\in J}\int \ f_s\delta_s\ \der \miss=0.$$
Therefore
$$\begin{aligned}
    0&=\sum_{s\in J}\int \ f_s\delta_s\ \der \miss
    \stackrel{\eqref{intteta}}{=}\sum_{s\in J}\pint{\sigma_s\pi((f_s\delta_s)\circ \s_{\Theta_s}^{-1})\xi}{\eta}
    \\&=\sum_{s\in J}\pint{\sigma_s\pi((f_s\circ \rr_{\Theta_s})\circ \s_{\Theta_s}^{-1})\xi}{\eta}
   =\sum_{s\in J}\pint{\sigma_s\pi(f_s\circ\theta_s) \xi}{\eta}
    \\&=\sum_{s\in J}\pint{\sigma_s\pi(\alpha_s^{-1}(f_s))\xi }{\eta}
   =\sum_{s\in J}\pint{\pi(f_s)\sigma_s\xi }{\eta}
    \\&=\pint{\sum_{s\in J}\pi(f_s)\sigma_s\xi }{\eta}	.\end{aligned}$$
Thus $\sum_{s\in J}\pi(f_s)\sigma_s=0$ as $\xi,\eta\in H$ have been arbitrarily chosen.
\end{proof}

Next, we recall the concept of disjointification (see \cite[Remark 2.4]{disjointification}). Loosely speaking, the disjointification of a  family of subsets consists in collecting those small pieces generated in the Venn diagram. Let $\mathcal{F}=\{A_i\}_{i=1}^n$ be a family of sets and $I=\{1,...,n\}$ be the index set. For a non-empty set $J\cont I$ define $$P_J=\bigcap_{i\in J}A_i\sm\bigcup_{i\in I\sm J}A_i.$$ The \textit{disjointification} of $\mathcal{F}$ is the family $\mathcal{D}=\{P_J\ |\ \emptyset\neq J\cont I\}$. Let us briefly recall some properties that the disjointification of the family $\mathcal{F}$ has. Note that for $i\in I$ and a non-empty subset $J\cont I$, either $i\in J$ or $i\in I\sm J$ and, respectively, either $P_J\cont A_{i}$ or $P_J\cap A_{i}=\emptyset$. In particular, the family $\{P_J\ |\ \emptyset\neq J\cont I\}$ is pairwise disjoint.
Moreover, note that for every $i\in I$
$$A_i=\bigsqcup_{\{J\cont I\ | P_J\cont A_{i}\}}P_J.$$
In fact, if $i_0\in I$ and $x\in A_{i_0}$, we have that $x\in P_J$, where $J=\{i\in I\ |\ x\in A_i\}$. And hence $(\cont)$ holds. The other inclusion is easy.

Furthermore, if the family $\{A_i\}_{i=1}^n$ consists of open sets on a given topological space, it is easy to see that $\mathcal{D}$ forms a Borel partition for $\bigcup\limits_{i\in I}A_i$.  

We can now show that there is a correspondence between covariant pairs for $(\alpha,S,\czero(X))$ and representations of $\CCC$. Recall that a covariant pair $(\pi,\sigma)$ for $\alpha$ on $H$ gives rise to a homomorphism $\pi \times \sigma:L_\alpha \rightarrow B(H).$ Note that both  $L_\alpha$ and $\CCC$ have elements of the form $\sum_{s\in s}f_s\delta_s$, where in the former $\delta_s$ is a symbol while in the latter the meaning is explained in the discussion after Corollary \ref{pi(1u)}. Throughout the following, on $L_{\alpha}$ we replace $\delta$ by $\dtt$ and hence the elements of $L_\alpha$ are finite sums of the form $\sum_{s\in s}f_s\dtt_s$. Also, in the middle of the following proof we are going to use the concept of completely orthogonal family of operators defined below.

\begin{defi}\label{def:completelyorthgonalfamilyofoperators}
    Let $H$ be a Hilbert space. A finite family of operators $\{T_i\}_{i=1}^n \cont B(H)$ is called \textbf{completely orthogonal} if $T_i^*T_j=0$ and $T_iT_j^*=0$, whenever $i\neq j$.
\end{defi}

%
%
%
%
Note that if $\{T_i\}_{i=1}^n \cont B(H)$ is completely orthogonal then
\begin{equation}\label{eq:normcompletelyorthogonal}
\left\|\sum_{i=1}^n T_i\right\|=\max_{i=1...n} \big\{\|T_i\|\big\}.
\end{equation}  

\begin{prop}\label{lemma:integration}
Let $(\pi,\sigma)$ be a covariant pair for the system $(\alpha,S, \czero(X))$. Then, $(\pi,\sigma)$ gives rise to a representation $\overline{\pi\times \sigma }$ of $\CCC$, where $$\overline{\pi\times \sigma }\left(\sum_{s\in s}f_s\delta_s\right)=\sum_{s\in s}\pi(f_s)\sigma_s.$$
\end{prop}

\begin{proof}
Define $\overline{\pi\times \sigma }\left(\sum_{s\in s}f_s\delta_s\right)=\sum_{s\in s}\pi(f_s)\sigma_s.$ By Lemma \ref{8.4}, $\overline{\pi\times \sigma }$ is well-defined. Moreover, note that $\overline{\pi\times \sigma}$ is linear and that for $f\in \cc(\drs)$ and $g\in \cc(\drt)$ we have $$\begin{aligned}\overline{\pi\times \sigma}(f\delta_s\ast g\delta_t)
	&\stackrel{\eqref{prod2}}{=}\overline{\pi\times \sigma}\left(\alpha_s(\alpha_s^{-1}(f)g)\delta_{st}\right)
	\\&=\pi(\alpha_s\left(\alpha_s^{-1}(f)g\right)\sigma_{st}
	\\&=\pi\times \sigma\left(\alpha_s(\alpha_s^{-1}(f)g)\dtt_{st}\right)
	\\&=\pi\times \sigma\left(f\dtt_sg\dtt_t\right)
	\\&=\pi\times \sigma(f\dtt_s)\pi\times \sigma( g\dtt_t)
	\\&=\pi(f)\sigma_s\pi( g)\sigma_t
	\\&=\overline{\pi\times \sigma}(f\delta_s)\overline{\pi\times \sigma}(g\delta_t).\end{aligned}$$
Thus, $\overline{\pi\times\sigma}$ is an algebra homomorphism. Furthermore, note that $\CZ=\bigcup_{e\in E}\Theta_e$ and hence every $f\in \cc\big(\CZ\big)$ is of the form $\sum_{e\in E}f_e\delta_e$ (see \cite[Theorem 2.13]{rudin}). In this case, 
$$\begin{aligned}
    \overline{\pi\times \sigma}\left(\overline{f}\right)
    &=\overline{\pi\times \sigma}\left(\sum_{e\in E}\overline{f_e}\delta_e\right)
    =\sum_{e\in E}\pi\left(\overline{f_e}\right)\sigma_e
    \\&\stackrel{\eqref{se=pie}}{=}\sum_{e\in E}\sigma_e\pi\left(\overline{f_e}\right)
    =\sum_{e\in E}\sigma_e^*\pi(f_e)^*
    \\&=\left(\sum_{e\in E}\pi(f_e)\sigma_e\right)^*
    =\left(\overline{\pi\times \sigma}(f)\right)^*.
\end{aligned}$$
It remains to show that if $F\in \cc(U)$, $U\in \bisc$, then $\|\overline{\pi\times \sigma}(F)\|\leq \|F\|_{\infty}$. Write $F=\sum_{i=1}^n f_{s_i}\delta_{s_i}$ and let $\{P_j\}_{j=1}^m$ be the disjointification of $\{\Theta_{i}\}_{i=1}^n$, where $\Theta_i:=\Theta_{s_i}$. We emphasize that $\{P_j\}_{j=1}^m$ is a Borel partition of $\bigcup_{i=1}^n\Theta_i$ and that either $P_j\cap \Theta_{i}=\emptyset$ or $P_j\cont \Theta_{i}$, for any pair $(i,j)$. Moreover 
$$\Theta_i=\bigsqcup_{\{j: P_j\cont \Theta_{i}\}}P_j$$
and, in consequence,
\begin{equation}\label{dridisjoint}D_{\rho(s_i)}=\bigsqcup_{\{j: P_j\cont \Theta_{i}\}} \rr_{\Theta_{i}}(P_j).\end{equation} 
For every $j\in \{1,...,m\}$, define $\mathfrak{L}_j=\{i\ |P_j\cont \Theta_i \}$, $i(j):=\max \ \mathfrak{L}_j$, $F_j$ to be the restriction of $F$ to $P_j$, and define the operator
\begin{equation}\label{tjdefin}
T_j= \pti (F_j\circ \rr_{\Theta_{i(j)}}^{-1})\sigma_{s_{i(j)}}.
\end{equation}
By \eqref{equalonintersection}, for every $i\in \mathfrak{L}_j,$ and for every $G \in B(P_j)$ we have that
\begin{equation}\label{equalonintersection2}
\pti\left(G\circ \rr_{\Theta_i}^{-1}\right)\sigma_{s_i}=\pti\left(G\circ \rr_{\Theta_{i(j)}}^{-1}\right)\sigma_{s_{i(j)}}.
\end{equation}
Therefore, 
\begin{align*}
    \overline{\pi\times \sigma }(F)
   &=\sum_{i=1}^n \pi(f_{s_i})\sigma_{s_i}
    \stackrel{\eqref{dridisjoint}}{=}\sum_{i=1}^n\sum_{\{j: P_j\cont \Theta_{i}\}} \pti\left(1_{\rr_{\Theta_{i}}(P_j)}f_{s_i}\right)\sigma_{s_i}
    \\&=\sum_{i=1}^n\sum_{\{j: P_j\cont \Theta_{i}\}} \pti\left(1_{P_j}\circ\rr_{\Theta_{i}}^{-1}f_{s_i}\right)\sigma_{s_i}
    \\&=\sum_{i=1}^n\sum_{\{j: P_j\cont \Theta_{i}\}} \pti\left(\big(1_{P_j}f_{s_i}\circ\rr_{\Theta_{i}}\big)\circ \rr_{\Theta_{i}}^{-1}\right)\sigma_{s_i}
    \\&=\sum_{j=1}^m\sum_{\{i: P_j\cont \Theta_{i}\}} \pti\left(\big(1_{P_j}f_{s_i}\delta_{s_i}\big)\circ \rr_{\Theta_{i}}^{-1}\right)\sigma_{s_i}
    \\&\stackrel{\eqref{equalonintersection2}}{=}\sum_{j=1}^m\sum_{\{i: P_j\cont \Theta_{i}\}} \pti\left(\big(1_{P_j}f_{s_i}\delta_{s_i}\big)\circ \rr_{\Theta_{i(j)}}^{-1}\right)\sigma_{s_{i(j)}}
    \\&=\sum_{j=1}^m \pti\left(\left(1_{P_j}\sum_{\{i: P_j\cont \Theta_{i}\}}f_{s_i}\delta_{s_i}\right)\circ \rr_{\Theta_{i(j)}}^{-1}\right)\sigma_{s_{i(j)}}
    \\&=\sum_{j=1}^m \pti\big(F_j\circ \rr_{\Theta_{i(j)}}^{-1}\big)\sigma_{s_{i(j)}}
    =\sum_{j=1}^m T_j.
    \end{align*}

Next, we show the operator family $\{T_j\}_{j=1}^m$ is completely orthogonal. Let $k,l\in \{1,...,m\}$ be such that $k\neq l$. Suppose that  $\big(\overline{F_k}\circ \rr_{\Theta_{i(k)}}^{-1}\big) \big(F_l\circ \rr_{\Theta_{i(l)}}^{-1}\big)(x)\neq 0$, for some $x \in X$. In this case, we obtain 
$$\overline{F_k}\big([s_{i(k)},\theta_{s_{i(k)}}^{-1}(x)]\big)F_l\big([s_{i(l)},\theta_{s_{i(l)}}^{-1}(x)]\big)\neq 0.$$ 
Then, since $\rr\big([s_{i(k)},$ $\theta_{s_{i(k)}}^{-1}(x)]\big)=x$, $\rr\big([s_{i(l)},\theta_{s_{i(l)}}^{-1}(x)]\big)=x$, and the support of $F$ is contained in a bisection, we have that
$$[s_{i(k)},\theta_{s_{i(k)}}^{-1}(x)]= [s_{i(l)},\theta_{s_{i(l)}}^{-1}(x)].$$
Thus, $P_k\cap P_l\neq \emptyset$, leading to a contradiction. Therefore, $\big(\overline{F_k}\circ \rr_{\Theta_{i(k)}}^{-1}\big)\big(F_l\circ \rr_{\Theta_{i(l)}}^{-1})= 0$ and, similarly,  $\big(\overline{F_k}\circ \s_{\Theta_{i(k)}}^{-1}\big)\big(F_l\circ \s_{\Theta_{i(l)}}^{-1}\big)= 0$. A simple computation then gives
$$\begin{aligned}
    T_k^*T_l&=\sigma_{s_{i(k)}}^*\pti(\overline{F_k}\circ \rr_{\Theta_{i(k)}}^{-1})\pti(F_l\circ \rr_{\Theta_{i(l)}}^{-1})\sigma_{s_{i(l)}}
    \\&=\sigma_{s_{i(k)}}^*\pti(\overline{F_k}\circ \rr_{\Theta_{i(k)}}^{-1}F_l\circ \rr_{\Theta_{i(l)}}^{-1})\sigma_{s_{i(l)}}
    =0.
\end{aligned}$$
And
$$\begin{aligned}
\ \ \ \ \ \ \ \ \ \ \ \ \ \ \ \ \ \ \ \ \ \ \ \ \
T_kT_l^*&\ \ =\pti(F_k\circ \rr_{\Theta_{i(k)}}^{-1})\sigma_{s_{i(k)}}(\pti(F_l\circ \rr_{\Theta_{i(l)}}^{-1})\sigma_{s_{i(l)}})^*
    \\&\stackrel{\eqref{extendedcovariance}}{=}\sigma_{s_{i(k)}}\pti(\alpha_{s_{i(k)}}^{-1}(F_k\circ \rr_{\Theta_{i(k)}}^{-1}))(\sigma_{s_{i(l)}}\pti(\alpha_{s_{i(l)}}^{-1}( F_l\circ \rr_{\Theta_{i(l)}}^{-1})))^*
    \\&\ \ =\sigma_{s_{i(k)}}\pti(F_k\circ \s_{\Theta_{i(k)}}^{-1})\pti(\overline{F_l}\circ \s_{\Theta_{i(l)}}^{-1})\sigma_{s_{i(l)}}^*
    \\&\ \ =\sigma_{s_{i(k)}}\pti(F_k\circ \s_{\Theta_{i(k)}}^{-1}\overline{F_l}\circ \s_{\Theta_{i(l)}}^{-1})\sigma_{s_{i(l)}}^*
    =0.
\end{aligned}$$
And, hence
$$\begin{aligned}
\left\|\overline{\pi\times \sigma }(F)\right\|
&=\left\|\sum_{j=1}^{m}T_j\right\|
\stackrel{\eqref{eq:normcompletelyorthogonal}}{=}\max_{j=1,...,m}\left\{\|T_j\|\right\}
\\&=\max_{j=1,...,m}\left\{\|\pti \left(F_j\circ \rr_{\Theta_{i(j)}}^{-1}\right)\sigma_{s_{i(j)}}\|\right\}
\\&\leq \max_{j=1,...,m}\{\|F_j\|_{\infty}\}
=\|F\|_{\infty}
\end{aligned}$$
\end{proof}
Conversely, we have the following disintegration method
\begin{prop}\label{lemma:disintegration}
Let $\Pi:\CCC\rightarrow B(H)$ be a representation of $\CCC$. Then, there exists a covariant pair $(\pi,\sigma)$ for $\alpha$ such that $\Pi=\overline{\pi\times\sigma}$.
\end{prop}

\begin{proof}
 We start by constructing $\sigma$. Let $\Pi_s:B(\Theta_s)\rightarrow B(H)$ denote the Borel extension of the continuous linear map $\Pi$ restricted to ${\czero(\Theta_s)}$ (see Proposition \ref{bfc}). Define for every $s\in S$
        $$\sigma_s=\Pi_s(1_{\Theta_s}).$$
We will utilize the notation from the proof of Proposition $\ref{bfc}$ to verify that the map $\sigma: S \rightarrow B(H),\ s\mapsto \sigma_s,$ constitutes a representation of $S$. Note that $\sigma_s$ is contractive since the following holds for every $\xi,\eta\in H$
$$
|\pint{\sigma_s\xi}{\eta}|=
|\pint{\Pi_s(1_{\Theta_s})\xi}{\eta}|
=|\nu_{\xi,\eta}(\Theta_s)|
\leq \|\nu_{\xi,\eta}\|
=\|\tau_{\xi,\eta}\|\leq \|\xi\|\|\eta\|
.$$
Moreover, since $\Theta_s$ is $\sigma$-compact, there is an increasing sequence $\{g_n^s\}\cont \cc(\Theta_s)$ that converges pointwise to $1_{\Theta_s}$ for every $s\in S$. For $s,t \in S$ and $g\in \cc(\Theta_s)$ note that $\{g\ast g_n^t\}\cont \cc(\Theta_{st})$ is uniformly bounded and converges to $g\circ \rr_{\Theta_s}^{-1}\circ \rr_{\Theta_{st}}\cont \cc(\Theta_{st}).$ Therefore, for $\xi,\eta\in H$ we have that
$$\begin{aligned}
	\pint{\Pi(g)\sigma_t\xi}{\eta}
	&=\lim\limits_{n \rightarrow \infty}\pint{\Pi(g)\Pi(g_n^t)\xi}{\eta}
	=\lim\limits_{n \rightarrow \infty}\pint{\Pi(g\ast g_n^t)\xi}{\eta}
	=\pint{\Pi(g\circ \rr_{\Theta_s}^{-1}\circ \rr_{\Theta_{st}})\xi}{\eta}
	.\end{aligned}$$
Which gives
\begin{equation}\label{eq:aaa}
    \Pi(g)\sigma_t=
\Pi(g\circ \rr_{\Theta_s}^{-1}\circ \rr_{\Theta_{st}}).
\end{equation}
Similarly, if $h\in \cc(\Theta_t)$ we have that
\begin{equation}\label{eq:bbb}\sigma_s\Pi(h)
=\Pi(h\circ \s_{\Theta_t}^{-1}\circ \s_{\Theta_{st}}).\end{equation}
Hence
$$\begin{aligned}
	\pint{\sigma_s\sigma_t\xi}{\eta}
	&=\lim\limits_{n \rightarrow \infty}\pint{\Pi(g_n^s)\sigma_t\xi}{\eta}
	\\&=\lim\limits_{n \rightarrow \infty}\pint{\Pi(g_n^s\circ \rr_{\Theta_s}^{-1}\circ \rr_{\Theta_{st}})\xi}{\eta}
	\\&=\pint{\Pi(1_{\Theta_{st}})\xi}{\eta}
	=\pint{\sigma_{st}\xi}{\eta},
	\end{aligned}$$
proving that $\sigma_{st}=\sigma_s\sigma_t$. 

On the other hand, recall that $\CZ=X$ and hence define $\pi:\czero(X)\rightarrow B(H)$ to be the extension of the  contractive $\ast$-homomorphism $\Pi_{|\cc(\CZ)}$. Explicitly, since every $f\in \cc(\CZ)$ is of the form $\sum_{e\in E}f_e\delta_e$, we have $$\pi(f)=\sum_{e\in E}\Pi(f_e\delta_e).$$ 
We now show the covariance relation. Let $s$ be in $S$, $f$ be in $\cc\big(\dls\big)$. For $\xi,\eta \in H$, we have that
    $$\begin{aligned}
    \pint{\sigma_s\pi(f)\xi}{\eta}
    &=\pint{\sigma_s\Pi\big(f\delta_{\lambda(s)}\big)\xi}{\eta}
    \stackrel{\eqref{eq:bbb}}{=}\pint{\Pi(f\delta_{\lambda_s}\circ \s_{\Theta_{\lambda(s)}}^{-1}\circ \s_{\Theta_{s\lambda(s)}})\xi}{\eta}
    \\&	=\pint{\Pi(f\circ\s_{\Theta_{\lambda_s}}\circ \s_{\Theta_{\lambda(s)}}^{-1}\circ \s_{\Theta_{s}})\xi}{\eta}
    =\pint{\Pi(f\circ \s_{\Theta_{s}})\xi}{\eta}\\&=\pint{\Pi(\delta_sf)\xi}{\eta}\stackrel{\eqref{eq:permutingdelta}}{=}\pint{\Pi(\alpha_s(f)\delta_s)\xi}{\eta}
    =\pint{\Pi(\alpha_s(f)\circ\rr_{\Theta_{s}})\xi}{\eta}
    \\&=\pint{\Pi(\alpha_s(f)\circ \rr_{\Theta_{\rho(s)}} \circ \rr_{\Theta_{\rho(s)}}^{-1} \circ\rr_{\Theta_{\rho(s)s}})\xi}{\eta}
    \stackrel{\eqref{eq:aaa}}{=}\pint{\pi(\alpha_{s}(f))\sigma_s\xi}{\eta}.
    \end{aligned}$$ 
The fact that $\spf\ \pi(\czero(D_e))H=\sigma_e(H)$ follows from Corollary \ref{pi(1u)}. To finish, we check that $\Pi=\overline{\pi\times \sigma}.$ Let $f\delta_s$ be in $\CCC$. Note that \begin{equation}
\begin{aligned}
\overline{\pi\times \sigma}(f\delta_s)
&=\pi(f)\sigma_s=\Pi(f\delta_{ \rho(s)})\sigma_s
\\&=\Pi(f\circ \rr_{\Theta_{\rho(s)}} \circ \rr_{\Theta_{\rho(s)}}^{-1}\circ \rr_{\Theta_{\rho(s)s}} )
\\&=\Pi(f\circ \rr_{\Theta_{s}})=\Pi(f\delta_s).
\end{aligned}	
\end{equation}
\end{proof}
 We are now ready to prove the main result of this section.
\begin{teo}\label{thm:A(C)semicrossed}
Let $(S,E,\lambda,\rho)$ be a restriction semigroup, $X$ be a second countable locally compact Hausdorff space, $\theta:S\rightarrow \mathcal{I}(X)$ be an étale action and $\alpha$ be the étale action induced by $\theta$. Then, $\A(\C(\theta,S,X))$ and $\czero(X)\rtimes_\alpha S$ are isomorphic.
\end{teo}

\begin{proof}
By Remark \ref{converg}, $L'_{\alpha}=\left\{\sum_{s\in S}f_s\dtt_s\ |\ f_s\in \cc\big(\drs\big) \right\}$ is a dense subalgebra of $\czero(X)\rtimes_\alpha S.$ Hence, define the following map
$$\psi:L'_{\alpha}\longrightarrow \CCC , \qquad \psi\left(\sum f_s\dtt_s\right)\longmapsto \left(\sum f_s\delta_s\right).$$ 
From \eqref{eq:spandelta_s} and \eqref{prod1}, we obtain that $\psi$ is a surjective algebra homomorphism. Moreover, for any covariant pair $(\pi,\sigma)$ for $\alpha$ we have that \begin{equation}\label{same norm}
    \left\|\overline{\pi\times \sigma}\left(\psi\left(\sum_{s\in S} f_s\delta_s\right)\right)\right\|
    =\Big\|\sum_{s\in S}\pi(f_s)\sigma_s\Big\|=\left\|\pi\times\sigma\left(\sum_{s\in S} f_s\dtt_s\right)\right\|.
\end{equation} Thus, By Propositions \ref{lemma:integration} and \ref{lemma:disintegration}, we obtain that $\psi$ is isometric (on the quotient) and then it extends to an isomorphism.
\end{proof}

\begin{cor}\label{corolario}
Let $(S,E,\lambda,\rho)$ be a restriction semigroup, and suppose that $\echap$ is second countable. Then, the map $\delta_s\mapsto 1_{\Theta_s}$ defines an isomorphism between $\A(S)$ and $\A(\C(\theta,S, \echap)).$
\end{cor}

\begin{proof}
Combine the proof of Theorem \ref{thm:A(S)semicrossed} with the proof of Theorem \ref{thm:A(C)semicrossed}.
\end{proof}

\begin{cor}
If $\C$ is an étale category where $\CZ$ is second countable, then $\A(\C)$ is isomorphic to a semicrossed product.
\end{cor}

\begin{proof}
By Theorem \ref{thm:everycategoryisagermscategory}, $\C$ is isomorphic to $\C\big(\theta, \bisc, \CZ\big)$ and hence the algebras are isomorphic.
\end{proof}

\subsection{The reduced case}

Let $(S,E,\lambda,\rho)$ be a left-ample restriction semigroup, and suppose that $\echap$ is second countable. Moreover, let $\theta$ be the canonical action of $S$ on $\echap$ and $\C$ be the category of germs $\C(\theta,S,\echap)$. Equations \eqref{normaparrep} and \eqref{same norm} tell us that not only $\A(S)$ and $\A(\C)$ are isomorphic but also that for any representation $\sigma$ of $S$ the following holds
\begin{equation}\label{sameclosure}
\overline{\comp[S]}^{\|\cdot\|_\sigma}\cong \overline{\CCC}^{\|\cdot\|_{\overline{\pi_{\sigma}\times \sigma}}}.
\end{equation}
Recall the Definition \ref{def:repregularcategory} of the regular representation of $\C$, $\Pi: \CCC \rightarrow B(\ell_2(\C)$, given by
    $$\Pi_f(\delta_{z})
    =\sum_{x \in \C_{\rr(\gamma)}}f(x)\delta_{xz},\ f \in \CCC.
    $$ 
By Proposition \ref{leftmultinvariant},  the subset $\widetilde{S}\cont \C$ is closed by left composition and $\ell_2(\widetilde{S})\cont \ell_2(\C)$ is an invariant subspace of the regular representation $\Pi$. Let $(\pi,\sigma)$ be a covariant pair such that $\Pi=\overline{\pi\times\sigma}$. Since the subsets $\Theta_s$ are compact and open, we do not need to take the Borel extension of $\Pi$ to define $\sigma$ (\cf Proposition \ref{lemma:disintegration}) and hence $\sigma_s=\Pi(1_{\Theta_s}).$ Note that for every $[t,\varphi]\in \C$ and $s\in S$ we have
\begin{equation}\label{eq:sigma_partial_isometry}
\begin{aligned}
 \sigma_s(\delta_{[t,\varphi]})
 =\big[\theta_t(\varphi)\in \dls\big]\ 1_{\Theta_{s}}\big([s,\theta_t(\varphi)]\big) \delta_{[st,\varphi]}
 =\big[\theta_t(\varphi)\in \dls\big] \delta_{[st,\varphi]}.
 \end{aligned}
 \end{equation}
In particular, by Proposition \ref{propdomaincomp} and Equation \eqref{stildeconditiontocompose}, if $[t,\barlt]\in \widetilde{S}$ then it holds that
$$
\sigma_s\big(\delta_{[t,\barlt]}\big)=\big[\rho(t)\leq \lambda(s)\big]\ \delta_{[st,\barlst]}.
$$ 
Thus, by \eqref{formularegulardes}, we note that that $\sigma_{|\ell_2(\widetilde{S})}:S\rightarrow B(\ell_2(\widetilde{S}))$ is unitarily equivalent to the regular representation of $S$, $\phi':S\rightarrow B(\ell_2(s))$. In fact, the unitary operator implementing the equivalence is the one induced by $s\stackrel{\varPsi}{\mapsto} [s,\barls]$, defined in \eqref{sstilde}. Moreover, since $\phi'_e=\pi_{\phi'}(1_e)$ (\cf Theorem \ref{universal}), $\sigma_e=\Pi(1_e)$, and $\sps\{1_e\mid e \in E\}$ is a dense subalgebra of $\czero(\echap)$, we have that $\pi_{|\ell_2(\widetilde{S})}$ and $\pi_{\phi'}$ are unitarily equivalent via the same unitary operator. Thus, $\pi_{\phi'}\times \phi'$ is unitarily equivalent to $\Pi_{|\ell_2(\widetilde{S})}$. 

By Lemma \ref{stilde equal}, If $E$ is finite then $\widetilde{S}=\C$ and hence $\pi_{\phi'}\times \phi'$ and $\Pi_{|\ell_2(\widetilde{S})}$ induce the same norm on $\CCC$.

If $S$ is an inverse semigroup, then we also have $\A_r(S)=\A_r(\C)$ since $\A_r(S)=C^*(S)$ and $\A_r(\C)=C^*(\C)$ (\cf subsection \ref{groupoidcase}). The result then follows from \cite[Theorem 3.5]{khoshkam2002regular}.

The case where $E$ is countable remains open. We conclude this subsection by showing that $\sigma$ is a representation by partial isometries. By Equation \eqref{eq:sigma_partial_isometry}, we just need to show 
\begin{prop}
Let $[t,\varphi]$ and $[t',\varphi']$ be in $\C$, and suppose that $\theta_t(\varphi)\in \dls$, $\theta_t'(\varphi')\in \dls$ and $[st,\varphi]=[st',\varphi']$. Then, $[t,\varphi]=[t',\varphi']$.
\end{prop}

\begin{proof}
Since $[st,\varphi]=[st',\varphi']$, $\varphi=\varphi'$ and there exists $e\in E$ such that $\varphi \in D_e$ and $ste=st'e.$ Recall that $S$ is left-ample and hence $$t\lambda(st)e=t'\lambda(st')e.$$ Defining $h=\lambda(st)\lambda(st')e$, we have that $th=t'h$ and that $\varphi \in D_h$ since $\theta_t(\varphi)\in \dls,$ and $\theta_t'(\varphi)\in \dls$. Therefore, $[t,\varphi]=[t',\varphi']$. 
\end{proof}

\subsection{The groupoid case}\label{groupoidcase}

We now show how our work fits in the theory of groupoid $C^*$-algebras. Throughout this subsection, let $\G$ be an étale groupoid. Recall that $\CC(\G)$ has an involution given by $f^*(x)=\overline{f(x^{-1})},$ and that $C^*(\G)$ is the completion of $\CC(\G)$ equipped with the norm induced by the class of all $*$-representations. 
If $U$ is a compact open bisection of $\G$, then the $C^*$-algebra $\operatorname{C}(U)$ is unital and hence is generated by unitary elements. Recall that a unitary element of $\operatorname{C}(U)$ is just a function $f$ whose range is contained in $\mathbb{T}$, that is, $|f(x)|=1$ for every $x\in U$. 
\begin{lemma}\label{lemma:amplecase}
If $U$ is a compact open bisection of $\G$ and $f\in \operatorname{C}(U)$ is a unitary map, then $f$ is a partial isometry on $\CC(\G)$. Moreover, if $\pi:\CC(\G)\rightarrow B(H)$ is a representation in the sense of Definition \ref{definicaorepcate}, we have $\pi(f^*)=\pi(f)^*.$
\end{lemma}

\begin{proof}
Note that $f^*$ belongs to $\operatorname{C}\big(U^{-1}\big)$ and hence $f^*\ast f \in \cc\big(\GZ\big)$. By item \eqref{li4} of Lemma \ref{lemma: productwelldefined}, if $x\in U$ we have that 
$$\begin{aligned}
f\ast f^*\ast f(x)
=f(x)(f^*f)(\s(x))
=f(x)f^*(x^{-1})f(x)
=f(x)|f(x)|^2
=f(x).
\end{aligned}$$
Then, since $f\ast f^*\ast f\in \operatorname{C}(U)$, we have that $f\ast f^* \ast f=f$ and, similarly, $f^*\ast f \ast f^*=f^*$. To finish, note that the contraction $\pi(f^*)$ is a generalized inverse for $\pi(f)$ and $\pi(f^*)=\pi(f)^*$, by Corollary \ref{mbekhtacorolario}.
\end{proof}
We can now prove our first main result.
\begin{teo}\label{thm:gpdcase1}
Suppose that an étale groupoid $\G$ has a cover of compact open bisections $\mathcal{F}$. Then, $\A(\G)$ coincides with the $C^*$-algebra of $\G$, $C^*(\G)$.
\end{teo}

\begin{proof}
By Proposition \ref{prop: sumofbisections} and the above discussion, we have 
$$ \CC (\G)=\sps\{f:\ f \in \cc(U),\ \ran(f)\cont \mathbb{T},\ U\in \mathcal{F} \}.$$
Then, it follows by Lemma \ref{lemma:amplecase} that every representation $\pi:\CC(\G)\rightarrow B(H)$ is a $*$-homomorphism. This completes the proof.
\end{proof}

The next theorem proves that a similar fact holds without assuming that $\G$ is covered by compact bisection, but now we need an assumption regarding $\GZ$.

\begin{teo}\label{thm:gpdcase2}
Suppose that an étale groupoid $\G$ has a second countable unit space. Then, $\A(\G)=C^*(\G)$.
\end{teo}

\begin{proof}
Suppose that $\G$ is the groupoid of germs of an action $\theta$ of an inverse semigroup $S$ on a second countable locally compact Hausdorff space $X$, $\G=\G(\theta,S,X)$. Let $\Pi:\G\rightarrow B(H)$ be a representation of $\CC(\G)$. We now proceed to show that $\Pi$ is a $*$-homomorphism. From Proposition \ref{lemma:disintegration}, we have that $\Pi=\overline{\pi\times \sigma}$, for a covariant pair $(\pi,\sigma)$ for $(\alpha, S, \czero(X))$, where $\alpha$ is the induced by $\theta.$ Moreover, recall that $\sigma$ is a $*$- representation of $S$ (see Subsection \ref{subsection:inversesemigroupcase}). By \eqref{eq:spandelta_s}, to show that $\Pi(f^*)=\Pi(f)^*$ for every $f$ in $\CC(\G)$, it suffices to prove that $\Pi\big((f\delta_s)^*\big)=\Pi(f\delta_s)^*$ for every $f\in \cc\big(\drs\big)$, and for every $s\in S.$ Furthermore, recall that the inverse of a germ $[s,x]$ is $[s^*,\theta_{s}(x)]$. Then,
$$\begin{aligned}
 (f\delta_s)^*\big([s^*,x]\big)
 &=\overline{f\delta_s\big([s,\theta_{s^*}(x)]\big)}
 =\overline{f(x)}
 =(\delta_{s^*}\overline{f})\big([s^*,x]\big)
\\& \stackrel{\eqref{eq:permutingdelta}}{=}
 \alpha_{s^*}\big(~\overline{f}~\big)\delta_{s^*}\big([s^*,x]\big),
\end{aligned}$$
for every $s\in S$ and $f\in \cc\big(\drs\big)$. Which gives $(f\delta_s)^*=\alpha_{s^*}(\overline{f})\delta_{s*}$ since these functions have their support contained in $\Theta_{s^*}$. Therefore, for $\xi,\eta \in H$ we have that
$$\begin{aligned}
&~~\pint{\Pi(f\delta_s)\xi}{\eta}
 =\pint{\pi(f)\sigma_s\xi}{\eta}
 =\pint{\xi}{\sigma_{s^*}\pi\big(\overline{ f} \big)\eta}
 \\&=\pint{\xi}{\pi\left(\alpha_{s^*}\big(~\overline{f }~\big)\right)\sigma_{s^*}\eta}
 =\pint{\xi}{\Pi\big((f\delta_s)^*\big)}.
\end{aligned}$$
Which gives $\Pi(f\delta_s)^*= \Pi\big((f\delta_s)^*\big)$ and hence completes the proof.
\end{proof}
We finish this section with a result for the reduced version of these algebras.
\begin{teo}
If $\G$ is an étale groupoid, then  $\A_r(\G)=C^*_{red}(\G)$.
\end{teo}
 
 \begin{proof}
 The regular representation defined for étale categories in \eqref{formularreprregularcategorias} is precisely the same for étale groupoids \big(see for instance \cite[Section 9.3]{sims}\big) and hence the result follows.
 \end{proof}

\section*{Acknowledgement}
The second named author was financed in part by the
Coordenação de Aperfeiçoamento de Pessoal de Nível Superior - Brasil (CAPES) - Finance Code 001.

\bibliographystyle{abbrv}
\bibliography{ECRSOA}

\end{document}